\documentclass[a4paper,10pt]{amsart}
\usepackage[top=0.8in, bottom=0.8in, left=1.25in, right=1.25in]{geometry}

\usepackage{fancybox}
\usepackage{amsmath}
\usepackage{amssymb}
\usepackage{amsthm}
\usepackage{mathrsfs}
\usepackage{epstopdf}
\usepackage{epstopdf}
\usepackage{array} 
\usepackage{amscd}
\usepackage{color}

\usepackage[all]{xypic}

\usepackage{amsfonts}
\usepackage{amssymb}
\usepackage{mathrsfs}

\DeclareFontFamily{U}{rsfs}{%
\skewchar\font127}
\DeclareFontShape{U}{rsfs}{m}{n}{%
<-6>rsfs5<6-8.5>rsfs7<8.5->rsfs10}{}
\DeclareSymbolFont{rsfs}{U}{rsfs}{m}{n}
\DeclareSymbolFontAlphabet
{\mathrsfs}{rsfs}
\DeclareRobustCommand*\rsfs{%
\@fontswitch\relax\mathrsfs}

\theoremstyle{plain}
\newtheorem{thm}{Theorem}[section]
\newtheorem{prop}[thm]{Proposition}
\newtheorem{lem}[thm]{Lemma}

\newtheorem{cor}[thm]{Corollary}

\newtheorem{claim}[thm]{Claim}

\newtheorem{prop-defi}[thm]{Proposition-Definition}
\newtheorem{thm-defi}[thm]{Theorem-Definition}
\newtheorem{lem-defi}[thm]{Lemma-Definition}

\newtheorem{conj}[thm]{Conjecture}

\newdimen\argwidth
\def\db[#1\db]{
 \setbox0=\hbox{$#1$}\argwidth=\wd0
 \setbox0=\hbox{$\left[\box0\right]$}
  \advance\argwidth by -\wd0
 \left[\kern.3\argwidth\box0 \kern.3\argwidth\right]}

\newcommand{\cC}{\mathcal{C}}

\newcommand{\eE}{\mathcal{E}}
\newcommand{\fF}{\mathcal{F}}
\newcommand{\gG}{\mathcal{G}}
\newcommand{\hH}{\mathcal{H}}

\newcommand{\lL}{\mathcal{L}}

\newcommand{\oO}{\mathcal{O}}

\newcommand{\Hom}{\mathop{\rm Hom}\nolimits}

\newcommand{\dR}{\mathbf{R}}
\newcommand{\dL}{\mathbf{L}}

\newcommand{\Hilb}{\mathop{\rm Hilb}\nolimits}

\newcommand{\Pic}{\mathop{\rm Pic}\nolimits}

\newcommand{\id}{\textrm{id}}

\newcommand{\ch}{\mathop{\rm ch}\nolimits}

\newcommand{\Ext}{\mathop{\rm Ext}\nolimits}
\newcommand{\Spec}{\mathop{\rm Spec}\nolimits}

\newcommand{\Coh}{\mathop{\rm Coh}\nolimits}

\newcommand{\cneq}{\mathrel{\raise.095ex\hbox{:}\mkern-4.2mu=}}
\newcommand{\eqcn}{\mathrel{=\mkern-4.5mu\raise.095ex\hbox{:}}}

\newcommand{\ext}{\mathop{\rm ext}\nolimits}

\newcommand{\DT}{\mathop{\rm DT}\nolimits}

\newcommand{\Imm}{\mathop{\rm Im}\nolimits}

\newcommand{\Ker}{\mathop{\rm Ker}\nolimits}

\newcommand{\RHom}{\mathop{\dR\mathrm{Hom}}\nolimits}

\makeatletter
 
  \@addtoreset{equation}{section}
\makeatother

\title[{Stable pairs and GV type invariants for CY 4-folds}]
{Stable pairs and Gopakumar-Vafa type invariants \\ for Calabi-Yau 4-folds}
\date{}
\author{Yalong Cao}
\address{Kavli Institute for the Physics and Mathematics of the Universe (WPI),The University of Tokyo Institutes for Advanced Study, The University of Tokyo, Kashiwa, Chiba 277-8583, Japan}
\email{yalong.cao@ipmu.jp}
\author{Davesh Maulik}
\address{Massachusetts Institute of Technology, Departement of Mathematics, 77 Massachusetts Avenue Cambridge, MA 02139, US.}
\email{maulik@mit.edu}
\author{Yukinobu Toda}
\address{Kavli Institute for the Physics and Mathematics of the Universe (WPI),The University of Tokyo Institutes for Advanced Study, The University of Tokyo, Kashiwa, Chiba 277-8583, Japan}
\email{yukinobu.toda@ipmu.jp}

\begin{document}
\maketitle
\begin{abstract}
As an analogy to Gopakumar-Vafa conjecture on CY 3-folds, Klemm-Pandharipande defined GV type invariants on CY 4-folds using GW theory
and conjectured their integrality. In this paper, we define stable pair type invariants on CY 4-folds and use them to interpret these GV type invariants.
Examples are computed for both compact and non-compact CY 4-folds to support our conjectures. 
\end{abstract}

\tableofcontents

\section{Introduction}
\subsection{Background}

Gromov-Witten invariants are rational numbers counting stable maps from complex curves to algebraic varieties (or symplectic manifolds).
They are not necessarily integers because of multiple cover contributions. In~\cite{KP}, Klemm-Pandharipande gave a definition of
Gopakumar-Vafa type invariants on Calabi-Yau 4-folds using GW theory and conjectured that they are integers. 
For dimensional reasons, GW invariants for genus $g\geqslant 2$ always vanish on Calabi-Yau 4-folds, so the integrality conjecture only applies in genus $0$ and $1$.
In our previous paper \cite{CMT}, we gave a sheaf-theoretic interpretation of $g=0$ GV type invariants using 
$\DT_4$ invariants \cite{CL, BJ} of one-dimensional stable sheaves, analogous to the work of Katz for 3-folds \cite{Katz}.
 
In this paper, we propose a sheaf-theoretic approach to both genus $0$ and $1$ GV type invariants using stable pairs on CY 4-folds.
For CY 3-folds, a Pairs/GV conjecture was first developed in work of Pandharipande and Thomas \cite{PT,PT2}. 
Our paper may be viewed as an analogue of their work in the setting of CY 4-folds.

\subsection{GV type invariants on CY 4-folds}
Let $X$ be a smooth projective CY 4-fold.
As mentioned above, Gromov-Witten invariants
vanish for genus $g\geqslant2$ for dimensional reasons, so we only consider the genus 0 and 1 cases.

The genus 0 GW invariants on $X$ are defined using
insertions:
 for integral classes $\gamma_i \in H^{m_i}(X, \mathbb{Z}), \
1\leqslant i\leqslant n$,
one defines
\begin{equation}
\mathrm{GW}_{0, \beta}(\gamma_1, \ldots, \gamma_n)
=\int_{[\overline{M}_{0, n}(X, \beta)]^{\rm{vir}}}
\prod_{i=1}^n \mathrm{ev}_i^{\ast}(\gamma_i),
\nonumber \end{equation}
where $\mathrm{ev}_i \colon \overline{M}_{0, n}(X, \beta)\to X$
%\begin{align*}
%\mathrm{ev}_i \colon \overline{M}_{g, n}(X, \beta)
%\to X
%\end{align*}
is the $i$-th evaluation map.

The invariants
\begin{align}\label{intro:n}
n_{0, \beta}(\gamma_1, \ldots, \gamma_n) \in \mathbb{Q}
\end{align}
are defined in \cite{KP} by the identity
\begin{align*}
\sum_{\beta>0}\mathrm{GW}_{0, \beta}(\gamma_1, \ldots, \gamma_n)q^{\beta}=
\sum_{\beta>0}n_{0, \beta}(\gamma_1, \ldots, \gamma_n) \sum_{d=1}^{\infty}
d^{n-3}q^{d\beta}.
\end{align*}

For genus 1,
virtual dimensions of GW moduli spaces without marked points are zero, so
the GW invariants
\begin{align*}
\mathrm{GW}_{1, \beta}=\int_{[\overline{M}_{1, 0}(X, \beta)]^{\rm{vir}}}
1 \in \mathbb{Q}
\end{align*}
can be defined
without insertions.
The
invariants
\begin{align}\label{intro:n1}
n_{1, \beta} \in \mathbb{Q}
\end{align}
 are defined in~\cite{KP} by the identity
\begin{align*}
\sum_{\beta>0}
\mathrm{GW}_{1, \beta}q^{\beta}=
&\sum_{\beta>0} n_{1, \beta} \sum_{d=1}^{\infty}
\frac{\sigma(d)}{d}q^{d\beta}
+\frac{1}{24}\sum_{\beta>0} n_{0, \beta}(c_2(X))\log(1-q^{\beta}) \\
&-\frac{1}{24}\sum_{\beta_1, \beta_2}m_{\beta_1, \beta_2}
\log(1-q^{\beta_1+\beta_2}),
\end{align*}
where $\sigma(d)=\sum_{i|d}i$ and $m_{\beta_1, \beta_2}\in\mathbb{Z}$ are called meeting invariants which can be inductively determined by genus 0 GW invariants.
In~\cite{KP}, both of the invariants (\ref{intro:n}), (\ref{intro:n1})
are conjectured to be integers, and GW invariants on $X$ are computed  to support the conjectures in many examples
by either localization techniques or mirror symmetry.

\subsection{Our proposal}
The aim of this paper is to give a sheaf-theoretic interpretation for the above GV-type invariants (\ref{intro:n}), (\ref{intro:n1}) via stable pairs, 
using 
Donaldson-Thomas theory for CY 4-folds introduced by Cao-Leung~\cite{CL} and Borisov-Joyce~\cite{BJ}.

We consider the moduli space $P_n(X,\beta)$ of stable pairs $(s:\oO_X\to F)$ with $\ch(F)=(0,0,0,\beta,n)$. 
By Theorem \ref{vir class of pair moduli}, one can construct a virtual class  
\begin{equation}\label{pair moduli vir class intro}[P_n(X, \beta)]^{\rm{vir}}\in H_{2n}\big(P_n(X, \beta),\mathbb{Z}\big), \end{equation}
which depends on the choice of orientation of a certain (real) line bundle over $P_n(X, \beta)$. On each connected component
of $P_n(X, \beta)$, there are two choices of orientation, which affect the corresponding contribution to the virtual class
(\ref{pair moduli vir class intro}) by a sign (for each connected component). 

When $n=0$, the virtual dimension of the virtual class (\ref{pair moduli vir class intro}) is zero. 
By integrating, we define the stable pair invariant 
\begin{equation}
P_{0,\beta}:=\int_{[P_{0}(X,\beta)]^{\rm{vir}}}1 \in \mathbb{Z}. \nonumber \end{equation}
When $n=1$, the (real) virtual dimension of the virtual class (\ref{pair moduli vir class intro}) is two. We use
insertions to define invariants as follows. For integral classes $\gamma_i \in H^{m_i}(X, \mathbb{Z})$, $1\leqslant i\leqslant n$, let 
\begin{align}\label{intro, insert}
\tau \colon H^{m}(X)\to H^{m-2}(P_1(X,\beta)), \
\tau(\gamma)=\pi_{P\ast}(\pi_X^{\ast}\gamma \cup\ch_3(\mathbb{F}) ),
\end{align}
where $\pi_X$, $\pi_P$ are projections from $X \times P_1(X,\beta)$
to corresponding factors, $\mathbb{I}=(\pi_X^*\oO_X\to \mathbb{F})$ is the universal pair, and $\ch_3(\mathbb{F})$ is the
Poincar\'e dual to the fundamental cycle of $\mathbb{F}$.

Then we define stable pair invariants
\begin{align*}P_{1,\beta}(\gamma_1,\ldots,\gamma_n):=\int_{[P_1(X,\beta)]^{\rm{vir}}} \prod_{i=1}^{n}\tau(\gamma_i). \end{align*}

We propose the following interpretation of (\ref{intro:n}), (\ref{intro:n1}) using stable pair invariants.
\begin{conj}\label{intro:conj:GW/GV g=0}\emph{(Conjecture \ref{conj:GW/GV g=0})}
For a suitable choice of orientation, we have 
\begin{align*}
P_{1,\beta}(\gamma_1,\ldots,\gamma_n)=\sum_{\begin{subarray}{c}\beta_1+\beta_2=\beta  \\ \beta_1, \beta_2\geqslant0 \end{subarray} }n_{0,\beta_1}(\gamma_1,\ldots,\gamma_n)\cdot P_{0,\beta_2},  \end{align*}
where the sum is over all possible effective classes, and we set $n_{0,0}(\gamma_1,\ldots,\gamma_n):=0$ and $P_{0,0}:=1$.

In particular, when $\beta$ is irreducible, 
\begin{align*}P_{1,\beta}(\gamma_1,\ldots,\gamma_n)=n_{0,\beta}(\gamma_1,\ldots,\gamma_n). \end{align*}
\end{conj}
%We propose the following sheaf theoretical interpretation of (\ref{intro:n1}) using $P_{0,\beta}$.
\begin{conj}\emph{(Conjecture \ref{conj:GW/GV g=1})}\label{intro:conj:GW/GV g=1}
For a suitable choice of orientation, we have 
\begin{align*}
\sum_{\beta \geqslant 0}
P_{0, \beta}q^{\beta}=
\prod_{\beta>0} M\big(q ^{\beta}\big)^{n_{1, \beta}},
\end{align*}
where $M(q)=\prod_{k\geqslant 1}(1-q^{k})^{-k}$ is the MacMahon function and $P_{0,0}:=1$.
\end{conj}  
For instance, when Picard number of $X$ is one, for an irreducible curve class $\beta$, the above identity implies 
\begin{equation}P_{0,\beta}=n_{1,\beta}, \nonumber \end{equation}
\begin{equation}P_{0,2\beta}=n_{1,2\beta}+3 n_{1,\beta}+\binom{n_{1,\beta}}{2}, \nonumber \end{equation}
\begin{equation}P_{0,3\beta}=n_{1,3\beta}+n_{1,\beta}\cdot n_{1,2\beta}+6 n_{1,\beta}+6\binom{n_{1,\beta}}{2}+\binom{n_{1,\beta}}{3}, \nonumber \end{equation}
by comparing coefficients of $q^{\beta}$, $q^{2\beta}$ and $q^{3\beta}$. \\

One issue with our current proposal (as in our earlier conjecture \cite{CMT}) is that we do not have a general mechanism for choosing the orientation in the above conjectures.  Currently, in the cases we examine in this paper, we choose orientations on a case-by-case basis to show the correct matching.  It would be very interesting to construct canonical choices of orientation for these moduli spaces and study our conjectures using them.
\\

Our proposal is based on a
heuristic argument given in Section~\ref{heuristic argument},
where we show Conjecture \ref{intro:conj:GW/GV g=0}, \ref{intro:conj:GW/GV g=1}
assuming the CY 4-fold $X$ to be `ideal', i.e. curves in $X$ deform in some family of expected dimensions.
Apart from that, we verify our conjecture in examples as follows.

\subsection{Verifications of the conjecture I: compact examples}
We first prove our conjectures for some special compact Calabi-Yau 4-folds.
 
${}$ \\
\textbf{Sextic 4-folds}.
Let $X\subseteq \mathbb{P}^5$ be a degree six smooth hypersurface and $[l]\in H_2(X,\mathbb{Z})\cong H_2(\mathbb{P}^5,\mathbb{Z})$ be the line class. We check our conjectures for $\beta=[l]$ and $2[l]$.
\begin{prop}\emph{(Proposition \ref{sextic g=0}, \ref{sextic g=1})}
Let $X$ be a smooth sextic 4-fold and $[l]\in H_2(X,\mathbb{Z})$ be the line class. 
Then Conjecture \ref{intro:conj:GW/GV g=0} and \ref{intro:conj:GW/GV g=1} are true for $\beta=[l]$ and $2[l]$.
\end{prop}

%${}$ \\\textbf{Some complete intersections}.

${}$ \\
\textbf{Elliptic fibrations}.
We consider a projective CY 4-fold $X$ which admits an elliptic fibration
\begin{align*}\pi: X\to \mathbb{P}^3, \end{align*}
given by a Weierstrass model (\ref{elliptic fib}). Let $f$ be a general fiber of $\pi$ and $h$ be a hyperplane in $\mathbb{P}^3$, set
\begin{align*} 
B=\pi^{\ast}h, \ E=\iota(\mathbb{P}^3)\in H_{6}(X,\mathbb{Z}),
\end{align*}
where $\iota$ is a section of $\pi$.
Then we have
\begin{prop}\emph{(Proposition \ref{prop g=0 elliptic fib}, \ref{prop g=1 elliptic fib})}
\begin{enumerate}
\item Conjecture \ref{intro:conj:GW/GV g=0} is true for fiber class $\beta=[f]$ and $\gamma=B^2$ or $B\cdot E$. 
\item Conjecture \ref{intro:conj:GW/GV g=1} is true for multiple fiber classes $\beta=r[f]$ \emph{($r\geqslant1$)}. 
\end{enumerate}
\end{prop}
In the above cases, we can directly compute the pair invariants and check the compatibility with the computation 
of GW invariants in~\cite{KP}. 

${}$ \\
\textbf{Product of elliptic curve and Calabi-Yau 3-fold}.
Let $X=Y\times E$ be a product of a Calabi-Yau 3-fold and an elliptic curve $E$. 
We check our conjectures when the curve class comes from either $Y$ or $E$.
\begin{thm}\label{thm intro prod}\emph{(Theorem \ref{thm on g=0 irr from CY3}, \ref{g=1 product of elliptic curve and CY3}, 
Proposition \ref{g=1 conj for irr class from CY3})} 
Let $X=Y\times E$ be given as above. Then
\begin{enumerate}
\item Conjecture \ref{intro:conj:GW/GV g=0} is true for any irreducible curve class $\beta\in H_2(Y)\subseteq H_2(X)$, 
provided that $Y$ is a complete intersection in a product of projective spaces.
\item Conjecture \ref{intro:conj:GW/GV g=1} is true for any irreducible curve class $\beta\in H_2(Y)\subseteq H_2(X)$.
\item Conjecture \ref{intro:conj:GW/GV g=1} is true for classes $\beta=r[E]$ \emph{($r\geqslant1$)}.
\end{enumerate}
\end{thm}
The proof of these results is briefly reviewed here. 

For (1), when $\beta\in H_2(Y)\subseteq H_2(X)$ is an irreducible curve class, we have an isomorphism \begin{equation}P_n(X,\beta)\cong P_n(Y,\beta)\times E. \nonumber \end{equation}
The corresponding virtual class satisfies (see Proposition \ref{prop on pair moduli on CY3}):
\begin{equation}[P_n(X,\beta)]^{\mathrm{vir}}=[P_n(Y,\beta)]_{\mathrm{pair}}^{\mathrm{vir}}\otimes [E], \nonumber \end{equation}
for certain choice of orientation in defining the LHS, where the virtual class of $P_n(Y,\beta)$ is defined using 
the deformation-obstruction theory of pairs (Lemma \ref{lem on pair moduli on CY3}) instead of the deformation-obstruction theory of complexes in the derived category used by \cite{PT}.

In this case, we have a forgetful morphism 
\begin{equation}\label{map f} f: P_1(Y,\beta)\to M_{1,\beta}(Y), \quad (\oO_Y\to F)\mapsto F,  \end{equation}
to the moduli space $M_{1,\beta}(X)$ of 1-dimensional stable sheaves $F$ with $[F]=\beta$ and $\chi(F)=1$.
We show that the map satisfies Manolache's virtual push-forward formula (Proposition \ref{prop on pair obs on CY3}),
\begin{equation}\int_{[P_1(Y,\beta)]_{\mathrm{pair}}^{\mathrm{vir}}}1=\int_{[M_{1,\beta}(Y)]^{\mathrm{vir}}}1. \nonumber \end{equation}
Then Conjecture \ref{intro:conj:GW/GV g=0} can be reduced to Katz's conjecture on CY 3-fold $Y$ \cite{Katz} (Corollary \ref{reduce to Katz conj}).
Combining with our previous proof of Katz's conjecture for primitive classes \cite[Cor. A.6]{CMT}, we can conclude (1) of 
Theorem \ref{thm intro prod}.

As for (3), this is one of few cases where we can compute non-primitive curve classes and form generating series. The point is 
to identify pair moduli spaces on $X$ with Hilbert schemes of points on $Y$ and use computation of zero dimensional 
DT invariants of $Y$.

${}$  \\
\textbf{Hyperk\"ahler 4-folds}. When the CY 4-fold $X$ is a hyperk\"ahler 4-fold, GW invariants 
vanish, and so do the GV type invariants. To verify our conjectures, we are left to prove the vanishing
of pair invariants. A cosection map from the (trace-free) obstruction space is constructed and shown to be surjective and compatible with Serre duality (Proposition \ref{surj cosection}).
We expect the following vanishing result then follows. 
\begin{claim}\emph{(Claim \ref{vanishing for hk4})}
Let $X$ be a projective hyperk\"{a}hler 4-fold and $P_n(X,\beta)$ be the moduli space of stable pairs with $n\neq0$ or $\beta\neq 0$. 
Then the virtual class satisfies
\begin{equation}[P_n(X,\beta)]^{\rm{vir}}=0.  \nonumber \end{equation}\end{claim}
At the moment, a Kiem-Li type theory of cosection localization for D-manifolds 
is not available in the literature.
We believe that when such a theory is established, our claim should follow automatically. 
Nevertheless, we have the following evidence for the claim.

1. At least when $P_n(X,\beta)$ is smooth, Proposition \ref{surj cosection} gives the vanishing of virtual class.

2. If there is a complex analytic version of $(-2)$-shifted symplectic geometry \cite{PTVV} and the corresponding construction of virtual classes \cite{BJ},
one could prove the vanishing result as in $\mathrm{GW}$ theory, i.e. taking a generic complex structure in the $\mathbb{S}^{2}$-twistor family 
of the hyperk\"ahler 4-fold which does not support coherent sheaves and then vanishing of virtual classes follows from their deformation invariance.  

\subsection{Verifications of the conjecture II: local 3-folds and surfaces}
For a Fano 3-fold $Y$, we consider the non-compact CY 4-fold 
\begin{align*}
X=K_Y.
\end{align*}
In this case, the stable pair moduli space $P_n(X, \beta)$ is compact (Proposition \ref{prop on pair moduli on K_Y}), so we can formulate Conjecture~\ref{intro:conj:GW/GV g=0}, \ref{intro:conj:GW/GV g=1} here (even though the target is not projective).

When the curve class $\beta\in H_2(X)$ is irreducible, we study this as follows.
Similar to the case of the product of a CY 3-fold and an elliptic curve, for a certain choice of orientation,
the virtual class of $P_n(X, \beta)$ satisfies (Proposition \ref{prop on pair moduli on K_Y})
\begin{equation}[P_n(X,\beta)]^{\mathrm{vir}}=[P_n(Y,\beta)]_{\mathrm{pair}}^{\mathrm{vir}}, \nonumber \end{equation}
under the isomorphism 
\begin{equation}P_n(X,\beta)\cong P_n(Y,\beta). \nonumber \end{equation}
And we have a virtual push-forward formula (Proposition \ref{prop on pair obs on Fano3})
\begin{equation}f_*[P_1(Y,\beta)]_{\mathrm{pair}}^{\mathrm{vir}}=[M_{1,\beta}(Y)]^{\mathrm{vir}}, \nonumber \end{equation}
where $f:P_1(Y,\beta)\to M_{1,\beta}(Y)$, $(\oO_X\to F)\mapsto F$ is the morphism forgetting the section, $M_{1,\beta}(Y)$ is the moduli scheme of 1-dimensional stable sheaves $E$ on $Y$ with $[E]=\beta$ and $\chi(E)=1$.
Then Conjecture \ref{intro:conj:GW/GV g=0} is easily reduced to our previous conjecture \cite[Conjecture 0.2]{CMT}. Combined
with computations in \cite{Cao}, we have 
\begin{thm}\emph{(Proposition \ref{verify g=0 conj for fano hypersurface}, \ref{verify g=1 conj for fano hypersurface})}
Let $X=K_Y$ be given as above. Then
\begin{enumerate}
\item Conjecture \ref{intro:conj:GW/GV g=0} is true for any irreducible curve class $\beta\in H_2(X)\cong H_2(Y)$, 
provided that (i) $Y\subseteq \mathbb{P}^4$ is a smooth hypersurface of degree $d\leqslant4$, or (ii) $Y=S\times \mathbb{P}^1$ for a
toric del Pezzo surface $S$.  
\item Conjecture \ref{intro:conj:GW/GV g=1} is true for an irreducible curve class $\beta\in H_2(X)\cong H_2(Y)$
when $Y=\mathbb{P}^3$. %provided that (\textbf{TBC})
\end{enumerate}
\end{thm}

Similarly for a smooth projective surface $S$, we 
consider the 
non-compact CY 4-fold 
\begin{align*}
X=\mathrm{Tot}_{S}(L_1\oplus L_2),
\end{align*}
where $L_1$, $L_2$ are line bundles on $S$ 
satisfying $L_1\otimes L_2\cong K_S$. 
%Conjecture~\ref{intro:conj:GW/GV g=0}, \ref{intro:conj:GW/GV g=1} previously formulated on projective CY 4-folds. 
In particular, when $\beta$ is irreducible and $L_i\cdot \beta<0$ ($i=1,2$), the moduli space $P_n(X, \beta)$ of stable pairs on $X$ is
compact and smooth (Lemma \ref{stable pair on del-pezzo}, Proposition \ref{prop on del-pezzo}).
So pair invariants are well-defined and we can also study our conjectures in this case. In particular, we have 
\begin{prop}\emph{(Proposition \ref{verify conj on del-pezzo surface})}
Let $S$ be a del-Pezzo surface and $L^{-1}_1$, $L^{-1}_2$ be two ample line bundles on $S$ such that $L_1\otimes L_2\cong K_S$. 
Denote $\beta\in H_2(X,\mathbb{Z})\cong H_2(S,\mathbb{Z})$ to be an irreducible curve class on $X=\mathrm{Tot}_S(L_1\oplus L_2)$.
Then Conjecture \ref{intro:conj:GW/GV g=0}, \ref{intro:conj:GW/GV g=1} are true for $\beta$.
\end{prop}
In fact such a del Pezzo surface must be $\mathbb{P}^2$ or $\mathbb{P}^1\times \mathbb{P}^1$ (see the proof of Proposition \ref{verify conj on del-pezzo surface}), and the corresponding $X$ is given by 
\begin{align*} \mathrm{Tot}_{\mathbb{P}^2}(\oO(-1)\oplus \oO(-2)), \quad   \mathrm{Tot}_{\mathbb{P}^1\times \mathbb{P}^1}(\oO(-1,-1)\oplus \oO(-1,-1)).  \end{align*}
By using computations due to Kool and Monavari \cite{KM}, one can check Conjecture \ref{intro:conj:GW/GV g=0}, \ref{intro:conj:GW/GV g=1}
for small degree curve classes on such $X$ (see Section \ref{small deg section} for details).

\subsection{Verifications of the conjecture III: local curves}
Let $C$ be a smooth projective curve. 
We consider a CY 4-fold $X$ given by 
\begin{align*}
X=\mathrm{Tot}_{C}(L_1\oplus L_2\oplus L_3),
\end{align*}
 where 
$L_1, L_2, L_3$ are line bundles on $C$ satisfying 
$L_1\otimes L_2\otimes L_3\cong \omega_C$. 
The
three dimensional complex torus $T=(\mathbb{C}^{\ast})^{\times 3}$
acts on $X$ fiberwise over $C$. 
The $T$-equivariant GW invariants
\begin{align*}
\mathrm{GW}_{g, d[C]}(X) \in \mathbb{Q}(\lambda_1, \lambda_2, \lambda_3)
\end{align*}
can be defined via equivariant residue. 
Here $\lambda_i$ are the equivariant parameters
with respect to the $T$-action. 

On the other hand, there 
is a two dimensional subtorus $T_0\subseteq (\mathbb{C}^{\ast})^{3}$ which preserves the CY 4-form on $X$. 
We may define equivariant pair invariants 
\begin{align*}
P_{n,d[C]}(X)\in\mathbb{Q}(\lambda_1, \lambda_2)
\end{align*}
as rational functions in terms of equivariant parameters of $T_0$ following a localization principle for $\mathrm{DT}_4$ invariants 
(see Section \ref{localiza for stable pair}, \cite{CL}, \cite[Sect. 4.2]{CMT}). 

When $C=\mathbb{P}^1$ and $X=\mathcal{O}_{\mathbb{P}^{1}}(l_1,l_2,l_3)$, we explicitly determine $P_{1,d[C]}(X)$ for $d\leqslant2$ (Proposition \ref{compute deg two pair inv}). Note in this case 
$P_{0,[\mathbb{P}^{1}]}(X)=0$ and there are no insertions, so an equivariant analogue of Conjecture \ref{intro:conj:GW/GV g=0} is given by the following conjecture: 
\begin{conj}\label{intro:equi genus zero conj in deg two} \emph{(Conjecture \ref{equi genus zero conj in deg two})}
Let $X=\mathcal{O}_{\mathbb{P}^{1}}(l_1,l_2,l_3)$ for $l_1+l_2+l_3=-2$. Then 
\begin{equation}\mathrm{GW}_{0,2}(X)=P_{1,2[\mathbb{P}^{1}]}(X)+\frac{1}{8}P_{1,[\mathbb{P}^{1}]}(X). \nonumber \end{equation}
\end{conj}
We can verify the above equivariant conjecture in a large number of examples.
\begin{thm} \emph{(Theorem \ref{g=0 thm on local curve})}
Conjecture \ref{intro:equi genus zero conj in deg two} is true if $|l_1|\leqslant10$ and $|l_2|\leqslant10$. 
\end{thm}
When $C$ is an elliptic curve and $L_i$'s are general degree zero line bundles on $C$, one can define pair invariants and explicitly compute them.
\begin{thm}\emph{(Theorem \ref{local elliptic curve})}
Let $C$ be an elliptic curve, $L_i \in \Pic^0(C)$ $(i=1,2,3)$ general line bundles satisfying $L_1 \otimes L_2 \otimes L_3 \cong \omega_C$ and $X=\mathrm{Tot}_C(L_1\oplus L_2\oplus L_3)$. 

Then stable pair invariants $P_{0, d[C]}(X)$ are well-defined and fit into the generating series
\begin{align*}
\sum_{d\geqslant 0} P_{0, d[C]}(X)\,q^m=M(q),
\end{align*}
where $M(q):=\prod_{k\geqslant 1}(1-q^{k})^{-k}$ is the MacMahon function.
\end{thm}
Similarly, if we have $n_{1,\beta}$ \big($\beta\in H_{2}(X,\mathbb{Z})$\big) many such elliptic curves, then they contribute to pair invariants according to the formula:
\begin{equation}\sum_{\beta\geqslant 0} P_{0, \beta}q^\beta=\prod_{\beta>0}M(q^{\beta})^{n_{1,\beta}}. \nonumber \end{equation}
This calculation arises in the heuristic argument for our genus one conjecture (Conjecture \ref{intro:conj:GW/GV g=1}) in the `ideal' situation as families of rational curves do not contribute to pair invariants $P_{0, \beta}$'s (see Section \ref{heuristic argument} for more details).

\subsection{Speculation on the generating series of stable pair invariants}
\label{subsec:speculation}
As before, if we allow insertions, we can use the virtual class (\ref{pair moduli vir class intro}) and insertions to define 
stable pair invariants of $P_n(X,\beta)$ for any $n$. 

For $\gamma\in H^{4}(X, \mathbb{Z})$, $\tau(\gamma)\in H^2(P_n(X,\beta),\mathbb{Z})$, so we may define
\begin{align*}P_{n,\beta}(\gamma):=\int_{[P_n(X,\beta)]^{\rm{vir}}} \tau(\gamma)^n. \end{align*}
Our computations and geometric arguments 
indicate that we may have the following formula, 
which generalizes the formula in Conjecture~\ref{intro:conj:GW/GV g=0},
%Conjecture~\ref{intro:conj:GW/GV g=1}
\begin{align}\label{form of general P_n}P_{n,\beta}(\gamma)=\sum_{\begin{subarray}{c}\beta_0+\beta_1+\cdots+\beta_n=\beta  \\ \beta_0,\beta_1,\cdots,\beta_n\geqslant0 \end{subarray} }P_{0,\beta_0}\cdot \prod_{i=1}^n n_{0,\beta_i}(\gamma). \end{align}
To group these invariants into a generating series, we introduce notation 
%\begin{align*}\exp(\gamma):=\sum_{n=0}^{\infty}\frac{\tau(\gamma)^k}{n!}. \end{align*}
%\begin{align*}P_{n,\beta}(\exp(\gamma)):=\sum_{k=0}^{\infty}\int_{[P_n(X,\beta)]^{\rm{vir}}} \frac{\tau(\gamma)^k}{k!}
%=\frac{P_{n,\beta}(\gamma)}{n!}. \end{align*}
\begin{align*}
\mathrm{PT}(X)(\exp(\gamma))
:= \sum_{n,\beta}\frac{P_{n,\beta}(\gamma)}{n!}y^n q^{\beta}. \end{align*}
Assuming Conjecture \ref{intro:conj:GW/GV g=1}, then (\ref{form of general P_n}) is equivalent to the following 
Gopakumar-Vafa type formula:
\begin{align*}
\mathrm{PT}(X)(\exp(\gamma))
=\prod_{\beta}\Big(\exp(yq^{\beta})^{n_{0,\beta}(\gamma)}\cdot 
M(q^{\beta})^{n_{1,\beta}}\Big), \end{align*}
where $n_{0,\beta}(\gamma)$ and $n_{1,\beta}$ are genus $0$ and $1$ GV type invariants of $X$ (\ref{intro:n}), (\ref{intro:n1}) respectively and $M(q)=\prod_{k\geqslant 1}(1-q^{k})^{-k}$ is the MacMahon function.
As mentioned before, GW invariants on CY 4-folds vanish for $g>1$, so they do not form a nice generating series as
in the 3-folds case. Here the advantage of considering stable pair invariants is we can use them to form 
a generating series which is conjecturally of GV form. 

A heuristic explanation of the formula will be given in Section \ref{heuristic argument}.
Some more analysis will be pursued in a future work.

\subsection{Notation and convention}
In this paper, all varieties and schemes are defined over $\mathbb{C}$. 
For a morphism $\pi \colon X \to Y$ of schemes, 
and for $\fF, \gG \in \mathrm{D^{b}(Coh(\textit{X\,}))}$, we denote by 
$\dR \hH om_{\pi}(\fF, \gG)$ 
the functor $\dR \pi_{\ast} \dR \hH om_X(\fF, \gG)$. 
We also denote by $\mathrm{ext}^i(\fF, \gG)$ the dimension of 
$\Ext^i_X(\fF, \gG)$. 

A class $\beta\in H_2(X,\mathbb{Z})$ is called \textit{irreducible} (resp. \textit{primitive}) if it is not the sum of two non-zero effective classes
(resp. if it is not a positive integer multiple of an effective class).

\subsection{Acknowledgement}
We are very grateful to Martijn Kool and Sergej Monavari for helpful discussions on stable pairs
on local surfaces and generously sharing their computational results.
We would like to thank Dominic Joyce for helpful comments on our preprint and a responsible referee for very careful reading of our paper and
helpful suggestions which improves
the exposition of the paper.
Y. C. is partly supported by the Royal Society Newton International Fellowship, the World Premier International Research Center Initiative (WPI), MEXT, Japan
and JSPS KAKENHI Grant Number JP19K23397. 
D. M. is partly supported by NSF FRG grant DMS-1159265.
Y. T. is supported by World Premier International Research Center Initiative (WPI), MEXT, Japan, 
and Grant-in Aid for Scientific Research grant (No. 26287002) from MEXT, Japan.

\section{Definitions and conjectures}
Throughout this paper, unless stated otherwise,
 $X$ is always denoted to be a smooth projective Calabi-Yau 4-fold, i.e. $K_X\cong \oO_X$.

\subsection{GW/GV conjecture on CY 4-folds}\label{subsection GW/GV conj}
Let $\overline{M}_{g, n}(X, \beta)$
%\begin{align*}
%\overline{M}_{g, n}(X, \beta)
%\end{align*}
be the moduli space of genus $g$, $n$-pointed stable maps
to $X$ with curve class $\beta$.
Its virtual dimension is given by
\begin{align*}
-K_X \cdot \beta+(\dim X-3)(1-g)+n=1-g+n.
\end{align*}
For integral classes
\begin{align}\label{gamma}
\gamma_i \in H^{m_i}(X, \mathbb{Z}), \
1\leqslant i\leqslant n,
\end{align}
the GW invariant is defined by
\begin{align}\label{GWinv}
\mathrm{GW}_{g, \beta}(\gamma_1, \ldots, \gamma_n)
=\int_{[\overline{M}_{g, n}(X, \beta)]^{\rm{vir}}}
\prod_{i=1}^n \mathrm{ev}_i^{\ast}(\gamma_i),
\end{align}
where $\mathrm{ev}_i \colon \overline{M}_{g, n}(X, \beta)\to X$
%\begin{align*}
%\mathrm{ev}_i \colon \overline{M}_{g, n}(X, \beta)
%\to X
%\end{align*}
is the $i$-th evaluation map.

For $g=0$, the
virtual dimension of $\overline{M}_{0, n}(X, \beta)$
is $n+1$, and (\ref{GWinv})
is zero unless
\begin{align}\label{sum:m}
\sum_{i=1}^{n}(m_i-2)=2.
\end{align}
In analogy with the Gopakumar-Vafa conjecture for CY 3-folds \cite{GV}, Klemm-Pandharipande \cite{KP} defined invariants $n_{0, \beta}(\gamma_1, \ldots, \gamma_n)$ on CY 4-folds by the identity
\begin{align*}
\sum_{\beta>0}\mathrm{GW}_{0, \beta}(\gamma_1, \ldots, \gamma_n)q^{\beta}=
\sum_{\beta>0}n_{0, \beta}(\gamma_1, \ldots, \gamma_n) \sum_{d=1}^{\infty}
d^{n-3}q^{d\beta},
\end{align*}
and conjecture the following
\begin{conj}\emph{(\cite[Conjecture~0]{KP})}\label{KPconj}
The invariants $n_{0, \beta}(\gamma_1, \ldots, \gamma_n)$
are integers.
\end{conj}
For $g=1$, the virtual dimension of
$\overline{M}_{1, 0}(X, \beta)$ is zero, so
no insertions are needed.
The genus one GW invariant
\begin{align*}
\mathrm{GW}_{1, \beta}=
\int_{[\overline{M}_{1, 0}(X, \beta)]^{\rm{vir}}}\in \mathbb{Q}
\end{align*}
is also expected to be described in terms of
certain integer valued invariants.

Let $S_1, \cdots, S_k$ be a basis of the free part of $H^4(X, \mathbb{Z})$ and
\begin{align*}
\sum_{i, j} g^{ij} [S_i \otimes S_j] \in H^8(X \times X, \mathbb{Z})
\end{align*}
be the $(4,4)$-component of K\"unneth
decomposition of the diagonal.
For $\beta_1, \beta_2 \in H_2(X, \mathbb{Z})$,
the meeting number
\begin{align*}
m_{\beta_1, \beta_2} \in \mathbb{Z}
\end{align*}
is introduced in~\cite{KP}
as a virtual number of rational curves
of class $\beta_1$ meeting rational curves of
class $\beta_2$. They are uniquely determined by
the following rules: \\

(i) The meeting invariants are symmetric,
$m_{\beta_1, \beta_2}=m_{\beta_2, \beta_1}$. \\

(ii) If either $\deg(\beta_1)\leqslant 0$ or $\deg(\beta_2)\leqslant 0$,
we have $m_{\beta_1, \beta_2}=0$. \\

(iii)
If $\beta_1 \neq \beta_2$, then
\begin{align*}
m_{\beta_1, \beta_2}=
\sum_{i, j} n_{0, \beta_1}(S_i) g^{ij} n_{0, \beta_2}(S_j)
+m_{\beta_1, \beta_2-\beta_1}+m_{\beta_1-\beta_2, \beta_2}.
\end{align*}

(iv)
If $\beta_1=\beta_2=\beta$, we have
\begin{align*}
m_{\beta, \beta}=n_{0, \beta}(c_2(X))
+\sum_{i, j}n_{0, \beta}(S_i)g^{ij} n_{0, \beta}(S_j)
-\sum_{\beta_1+\beta_2=\beta}
m_{\beta_1, \beta_2}.
\end{align*}
The invariants $n_{1, \beta}$ are
uniquely defined by the identity
\begin{align*}
\sum_{\beta>0}
\mathrm{GW}_{1, \beta}q^{\beta}=
&\sum_{\beta>0}n_{1, \beta}\, \sum_{d=1}^{\infty}
\frac{\sigma(d)}{d}q^{d\beta} \\
&+\frac{1}{24}\sum_{\beta>0} n_{0, \beta}(c_2(X))\log(1-q^{\beta}) \\
&-\frac{1}{24}\sum_{\beta_1, \beta_2}m_{\beta_1, \beta_2}
\log(1-q^{\beta_1+\beta_2}),
\end{align*}
where $\sigma(d)=\sum_{i|d}i$.
\begin{conj}\emph{(\cite[Conjecture~1]{KP})}\label{KPconj 1}
The invariants $n_{1, \beta}$ are integers.
\end{conj}
For $g\geqslant 2$, GW invariants vanish for dimension reasons, so the GW/GV type integrality conjecture on CY 4-folds only applies 
for genus 0 and 1.
In \cite{KP}, GW invariants are computed directly in many examples using localization or mirror symmetry to support the conjectures. 

\subsection{Review of $\mathrm{DT_ 4}$ invariants}
Let us first introduce the set-up of $\mathrm{DT_ 4}$ invariants.
We fix an ample divisor $\omega$ on $X$
and take a cohomology class
$v \in H^{\ast}(X, \mathbb{Q})$.

%We denote
The coarse moduli space $M_{\omega}(v)$
%\begin{align*}
%M_{\omega}(v)
%\end{align*}
of $\omega$-Gieseker semistable sheaves
$E$ on $X$ with $\ch(E)=v$ exists as a projective scheme.
We always assume that
$M_{\omega}(v)$ is a fine moduli space, i.e.
any point $[E] \in M_{\omega}(v)$ is stable and
there is a universal family
\begin{align}\label{universal}
\eE \in \Coh(X \times M_{\omega}(v)).
\end{align}

In~\cite{BJ, CL}, under certain hypotheses,
the authors construct 
a $\mathrm{DT}_{4}$ virtual
class
\begin{align}\label{virtual}
[M_{\omega}(v)]^{\rm{vir}} \in H_{2-\chi(v, v)}(M_{\omega}(v), \mathbb{Z}), \end{align}
where $\chi(-,-)$ is the Euler pairing.
Notice that this class may not necessarily be algebraic.

Roughly speaking, in order to construct such a class, one chooses at
every point $[E]\in M_{\omega}(v)$, a half-dimensional real subspace
\begin{align*}\Ext_{+}^2(E, E)\subset \Ext^2(E, E)\end{align*}
of the usual obstruction space $\Ext^2(E, E)$, on which the quadratic form $Q$ defined by Serre duality is real and positive definite. 
Then one glues local Kuranishi-type models of form 
%More precisely, one considers a real analytic space $M^{DT_4}_{\omega}(v)$  homeomorphic to $M_{\omega}(v)$
%and produces an open covering by charts $U(E)$ which admit a Kuranishi-type description
\begin{equation}\kappa_{+}=\pi_+\circ\kappa: \Ext^{1}(E,E)\to \Ext_{+}^{2}(E,E),  \nonumber \end{equation}
where $\kappa$ is a Kuranishi map of $M_{\omega}(v)$ at $E$ and $\pi_+$ is the projection 
according to the decomposition $\Ext^{2}(E,E)=\Ext_{+}^{2}(E,E)\oplus\sqrt{-1}\cdot\Ext_{+}^{2}(E,E)$.  \\
%From this description, one expects an associated virtual fundamental class on this compact space.

In \cite{CL}, local models are glued in three special cases: 
\begin{enumerate}
\item when $M_{\omega}(v)$ consists of locally free sheaves only; 
\item  when $M_{\omega}(v)$ is smooth;
\item when $M_{\omega}(v)$ is a shifted cotangent bundle of a derived smooth scheme. 
\end{enumerate}
And the corresponding virtual classes are constructed using either gauge theory or algebro-geometric perfect obstruction theory.

The general gluing construction is due to Borisov-Joyce \cite{BJ}\,\footnote{One needs to assume that $M_{\omega}(v)$ can be given a $(-2)$-shifted symplectic structure as in Claim 3.29 \cite{BJ} to apply their constructions. In the stable pairs case, we show this can be done in Lemma \ref{exist of -2 str}. }, based on Pantev-T\"{o}en-Vaqui\'{e}-Vezzosi's theory of shifted symplectic geometry \cite{PTVV} and Joyce's theory of derived $C^{\infty}$-geometry.
The corresponding virtual class is constructed using Joyce's
D-manifold theory (a machinery similar to Fukaya-Oh-Ohta-Ono's theory of Kuranishi space structures used in defining Lagrangian Floer theory).

In this paper, all computations and examples will only involve the virtual class constructions in situations (2), (3), mentioned above. We briefly 
review them as follows:  
\begin{itemize}
\item When $M_{\omega}(v)$ is smooth, the obstruction sheaf $Ob\to M_{\omega}(v)$ is a vector bundle endowed with a quadratic form $Q$ via Serre duality. Then the $\DT_4$ virtual class is given by
\begin{equation}[M_{\omega}(v)]^{\rm{vir}}=\mathrm{PD}(e(Ob,Q)).   \nonumber \end{equation}
Here $e(Ob, Q)$ is the half-Euler class of 
$(Ob,Q)$ (i.e. the Euler class of its real form $Ob_+$), 
and $\mathrm{PD}(-)$ is its 
Poincar\'e dual. 
Note that the half-Euler class satisfies %$e(Ob,Q)=0$ if $\mathrm{rk}(Ob)$ is odd and $e(Ob,Q)^{2}=(-1)^{\frac{\mathrm{rk}(Ob)}{2}}e(Ob)$
\begin{align*}
e(Ob,Q)^{2}&=(-1)^{\frac{\mathrm{rk}(Ob)}{2}}e(Ob),  \textrm{ }\mathrm{if}\textrm{ } \mathrm{rk}(Ob)\textrm{ } \mathrm{is}\textrm{ } \mathrm{even}, \\
 e(Ob,Q)&=0, \textrm{ }\mathrm{if}\textrm{ } \mathrm{rk}(Ob)\textrm{ } \mathrm{is}\textrm{ } \mathrm{odd}. 
\end{align*}
\item When $M_{\omega}(v)$ is a shifted cotangent bundle of a derived smooth scheme, roughly speaking, this means that at any closed point $[F]\in M_{\omega}(v)$, we have Kuranishi map of type
\begin{equation}\kappa \colon
 \Ext^{1}(F,F)\to \Ext^{2}(F,F)=V_F\oplus V_F^{*},  \nonumber \end{equation}
where $\kappa$ factors through a maximal isotropic subspace $V_F$ of $(\Ext^{2}(F,F),Q)$. Then the $\DT_4$ virtual class of $M_{\omega}(v)$ is, 
roughly speaking, the 
virtual class of the perfect obstruction theory formed by $\{V_F\}_{F\in M_{\omega}(v)}$. 
When $M_{\omega}(v)$ is furthermore smooth as a scheme, 
then it is
simply the Euler class of the vector bundle 
$\{V_F\}_{F\in M_{\omega}(v)}$ over $M_{\omega}(v)$. 
\end{itemize}
${}$ \\
\textbf{On orientations}.
To construct the above virtual class (\ref{virtual}) with coefficients in $\mathbb{Z}$ (instead of $\mathbb{Z}_2$), we need an orientability result 
for $M_{\omega}(v)$, which is stated as follows.
Let  
\begin{equation}\label{det line bdl}
 \lL:=\mathrm{det}(\dR \hH om_{\pi_M}(\eE, \eE))
 \in \Pic(M_{\omega}(v)), \quad  
\pi_M \colon X \times M_{\omega}(v)\to M_{\omega}(v),
\end{equation}
be the determinant line bundle of $M_{\omega}(v)$, equipped with a symmetric pairing $Q$ induced by Serre duality.  An \textit{orientation} of 
$(\mathcal{L},Q)$ is a reduction of its structure group (from $O(1,\mathbb{C})$) to $SO(1, \mathbb{C})=\{1\}$; in other words, we require a choice of square root of the isomorphism
\begin{equation}\label{Serre duali}Q: \lL\otimes \lL \to \oO_{M_{\omega}(v)}  \end{equation}
to construct the virtual class (\ref{virtual}).
An orientability result was first obtained for $M_{\omega}(v)$ when the CY 4-fold $X$ satisfies
$\mathrm{Hol}(X)=SU(4)$ and $H^{\rm{odd}}(X,\mathbb{Z})=0$ \cite[Theorem 2.2]{CL2} and it has recently 
been generalized to arbitrary CY 4-folds by \cite{CGJ}. 
Notice that, if an orientation exists, the set of orientations forms a torsor for $H^{0}(M_{\omega}(v),\mathbb{Z}_2)$.

\subsection{Stable pair invariants on CY 4-folds}\label{subsection stable pair invs}
The notion of stable pairs on a CY 4-fold $X$ can be defined similarly
as in the case of threefolds ~\cite{PT}. It consists of data
\begin{align*}
(F, s), \ F \in \Coh(X), \ s \colon \oO_X \to F
\end{align*}
where $F$ is a pure one dimensional sheaf
and $s$ is surjective in dimension one.

For $\beta \in H_2(X, \mathbb{Z})$ and $n\in \mathbb{Z}$,
let
\begin{align}\label{PT:CY4}
P_n(X, \beta)
\end{align}
be the moduli space of stable pairs $(F, s)$ on $X$
such that $[F]=\beta$, $\chi(F)=n$.
It is a projective scheme parametrizing two-term complexes
\begin{align*}
I=(\oO_X \stackrel{s}{\to} F) \in D^b(\Coh(X))
\end{align*}
in the derived category of coherent sheaves on $X$.

Similar to moduli spaces of stable sheaves,
the stable pair moduli space (\ref{PT:CY4})
admits a deformation-obstruction theory, whose tangent, obstruction and `higher' obstruction spaces are
given by
\begin{align*}
\Ext^1(I, I)_0, \ \Ext^2(I, I)_{0},  \ \Ext^3(I, I)_{0},
\end{align*}
where $(-)_0$ denotes the trace-free part. Note that Serre duality gives an isomorphism $\Ext_0^1\cong (\Ext_0^3)^{\vee}$
and a non-degenerate quadratic form on $\Ext_0^2$. Moreover, we have
\begin{lem}\label{exist of -2 str}
The stable pair moduli space $P_n(X, \beta)$ can be given the structure of a $(-2)$-shifted symplectic derived scheme in the sense of Pantev-T\"{o}en-Vaqui\'{e}-Vezzosi \cite{PTVV}.
\end{lem}
\begin{proof}
By \cite[Theorem 2.7]{PT}, $P_n(X, \beta)$ is a disjoint union of connected components of the moduli stack of perfect complexes of coherent sheaves of trivial determinant on $X$, whose $(-2)$-shifted symplectic structure is constructed by~\cite[Theorem~0.1]{PTVV} (see \cite[Sect. 3.2, pp. 48]{PTVV} for pull-back to determinant fixed substack).
\end{proof}
Let $\mathbb{I}=(\mathcal{O}_{X\times P_n(X, \beta)}\rightarrow \mathbb{F})$ be the universal pair, the determinant line bundle
\begin{align*}
\mathcal{L}:=\mathrm{det}(\dR \hH om_{\pi_P}(\mathbb{I}, \mathbb{I})_0)
\in \Pic(P_n(X, \beta))
\end{align*}
is endowed with a non-degenerate quadratic form $Q$ defined by Serre duality,
where $\pi_P\colon X\times P_n(X, \beta)\rightarrow P_n(X, \beta)$ is the projection.
Similarly as before, the orientability issue for the pair moduli space $P_n(X, \beta)$ is whether the structure group of the quadratic line bundle $(\mathcal{L},Q)$ can be reduced from $O(1,\mathbb{C})$ to $SO(1,\mathbb{C})=\{1\}$. 
By \cite{CGJ}, these moduli spaces are always orientable.

\begin{thm}\label{vir class of pair moduli}
Let $X$ be a CY 4-fold and $\beta\in H_{2}(X,\mathbb{Z})$ and $n\in\mathbb{Z}$ be an integer.
Then $P_n(X, \beta)$ has a virtual class 
\begin{align}\label{pair moduli vir class}
[P_n(X, \beta)]^{\rm{vir}}\in H_{2n}\big(P_n(X, \beta),\mathbb{Z}\big),  \end{align}
in the sense of Borisov-Joyce \cite{BJ}, depending on the choice of orientation.
\end{thm}
\begin{proof}
By Lemma \ref{exist of -2 str}, $P_n(X,\beta)$ has a $(-2)$-shifted symplectic structure. By \cite{CGJ}, $P_n(X,\beta)$ is orientable
in the sense stated above. Then we may apply \cite[Thm. 1.1]{BJ} to $P_n(X,\beta)$.
\end{proof}
When $n=0$, the virtual dimension of the virtual class (\ref{pair moduli vir class}) is zero. 
We define the stable pair invariant
\begin{align*}
P_{0, \beta} \cneq \int_{[P_0(X, \beta)]^{\rm{vir}}}1 \in \mathbb{Z},
\end{align*}
as the degree of the virtual class.

When $n=1$, the (real) virtual dimension of the virtual class (\ref{pair moduli vir class}) is two, so we consider
insertions as follows. 
For integral classes $\gamma_i \in H^{m_i}(X, \mathbb{Z})$, $1\leqslant i\leqslant n$, let 
\begin{align*}
\tau \colon H^{m}(X)\to H^{m-2}(P_1(X,\beta)), \quad 
\tau(\gamma):=(\pi_{P})_{\ast}(\pi_X^{\ast}\gamma \cup\ch_3(\mathbb{F}) ),
\end{align*}
where $\pi_X$, $\pi_P$ are projections from $X \times P_1(X,\beta)$
to corresponding factors, $\mathbb{I}=(\pi_X^*\oO_X\to \mathbb{F})$ is the universal pair and $\ch_3(\mathbb{F})$ is the
Poincar\'e dual to the fundamental cycle of $\mathbb{F}$.

We define the stable pair invariant
\begin{align*}P_{1,\beta}(\gamma_1,\ldots,\gamma_n):=\int_{[P_1(X,\beta)]^{\rm{vir}}} \prod_{i=1}^{n}\tau(\gamma_i). \end{align*}

\subsection{Relations with GW/GV conjecture on $\mathrm{CY_4}$}\label{subsec:geometric}
We use the stable pair invariants defined in Section \ref{subsection stable pair invs} to 
give a sheaf-theoretic approach to the GW/GV conjecture in Section \ref{subsection GW/GV conj}.

\begin{conj}\label{conj:GW/GV g=0}(Genus 0)
For a suitable choice of orientation, we have 
\begin{align*}
P_{1,\beta}(\gamma_1,\ldots,\gamma_n)=\sum_{\begin{subarray}{c}\beta_1+\beta_2=\beta  \\ \beta_1, \beta_2\geqslant0 \end{subarray} }n_{0,\beta_1}(\gamma_1,\ldots,\gamma_n)\cdot P_{0,\beta_2},  \end{align*}
where the sum is over all possible effective classes, and we set $n_{0,0}(\gamma_1,\ldots,\gamma_n):=0$, $P_{0,0}:=1$.

\end{conj}

\begin{conj}\label{conj:GW/GV g=1}(Genus 1)
For a suitable choice of orientation, we have  
\begin{align*}
\sum_{\beta \geqslant 0}
P_{0, \beta}\,q^{\beta}=
\prod_{\beta>0} M\big(q ^{\beta}\big)^{n_{1, \beta}},
\end{align*}
where $M(q)=\prod_{k\geqslant 1}(1-q^{k})^{-k}$ is the MacMahon function and $P_{0,0}:=1$.
\end{conj}

\subsection{Heuristic approach to conjectures}\label{heuristic argument}
In this subsection, we give a heuristic argument to explain why we expect Conjecture \ref{conj:GW/GV g=0}, \ref{conj:GW/GV g=1} (and equality (\ref{form of general P_n})) to be true.  Even in this heuristic discussion, we ignore questions of orientation.

Let $X$ be an `ideal' $\mathrm{CY_{4}}$
in the sense that all curves of $X$ deform in families of expected dimensions, and have expected generic properties, i.e.
\begin{enumerate}
\item
any rational curve in $X$ is a chain of smooth $\mathbb{P}^1$'s with normal bundle $\mathcal{O}_{\mathbb{P}^{1}}(-1,-1,0)$, and
moves in a compact 1-dimensional smooth family of embedded rational curves, whose general member is smooth with 
normal bundle $\mathcal{O}_{\mathbb{P}^{1}}(-1,-1,0)$. 
\item
any elliptic curve $E$ in $X$ is smooth, super-rigid, i.e. 
the normal bundle is 
$L_1 \oplus L_2 \oplus L_3$
for general degree zero line bundle $L_i$ on $E$
satisfying $L_1 \otimes L_2 \otimes L_3=\oO_E$. 
Furthermore any two elliptic curves are 
disjoint and disjoint from all families of rational curves on $X$. 

\item
there is no curve in $X$ with genus $g\geqslant 2$.
\end{enumerate}

${}$ \\
\textbf{$P_0(X, \beta)$ and genus $1$ conjecture}.
%Note that for any stable pair $I=(\mathcal{O}_{X}\rightarrow F)$, $\textrm{Supp}(F)$ is a Cohen-Macaulay curve
Under our ideal assumptions, 
a one-dimensional Cohen-Macaulay scheme $C$ supported in one of our families of rational curves has $\chi(\oO_C)\geqslant 1$, so for any stable pair $I=(\oO_X \to F) \in P_0(X, \beta)$,
the sheaf $F$ can only be supported on some rigid elliptic curves in $X$. 
For a rigid elliptic curve $E$ with $[E]=\beta$ and `general' normal 
bundle (i.e. direct sum of three degree zero general line bundles on $E$), its contribution to the pair invariant is
\begin{equation} \sum_{m\geqslant 0} P_{0, m[E]} q^m=M(q), \textrm{ }\textrm{where} \textrm{ }M(q)=\prod_{k\geqslant 1}(1-q^{k})^{-k},   
\nonumber \end{equation}
by a localization calculation (see Theorem \ref{local elliptic curve}). 
Similarly, if we have $n_{1,\beta}$ \big($\beta\in H_{2}(X,\mathbb{Z})$\big) many such elliptic curves, then they contribute to pair invariants according to the formula:
\begin{equation}\sum_{\beta\geqslant0} P_{0, \beta}q^\beta=\prod_{\beta>0}M(q^{\beta})^{n_{1,\beta}}. \nonumber \end{equation}

${}$ \\
\textbf{$P_1(X, \beta)$ and genus $0$ conjecture}. 
Given a stable pair $I=(\oO_X\to F)\in P_1(X,\beta)$, $F$ may be supported on a union of rational curves and elliptic curves. 
Let $C:=\mathrm{supp}(F)$, then $C=C_1\sqcup C_2$ is a disjoint union of 
`rational curve components' and `elliptic curve components'. Note a Cohen-Macaulay scheme $D$ in $\mathrm{Tot}_{\mathbb{P}^1}(-1,-1,0)$
(resp. in $\mathrm{Tot}_{E}(L_1\oplus L_2\oplus L_3)$, where $E$ a smooth elliptic curve and $L_i$'s are degree zero general line bundles on $E$) satisfies $\chi(\oO_{D})\geqslant 1$ (resp. $\chi(\oO_{D})\geqslant 0$). 

Thus from the exact sequence
\begin{align*}0\to \oO_C\to F \to Q\to0, \end{align*}
we know if $C_1\neq \emptyset$, then $Q=0$ and $F\cong \oO_{C_1\sqcup C_2}$ (with $\chi(\oO_{C_1})=1$, $\chi(\oO_{C_2})=0$). Note that when $C_1=\emptyset$, i.e. when $F$ is supported
on elliptic curves, once we include insertions, these stable pairs do not contribute to the invariant
\begin{equation}\int_{[P_{1}(X,\beta)]^{\rm{vir}}}\tau(\gamma). \nonumber \end{equation}
So we only consider the case when $F\cong \oO_{C_1\sqcup C_2}$ with $C_1$ supported on rational curves in
a one-dimensional family $\{C_t\}_{t\in T}$.   We may further assume the support of $C_1$ is smooth with normal bundle $\mathcal{O}_{\mathbb{P}^{1}}(-1,-1,0)$ due to the presence of insertions, at which point it must have multiplicity $1$ as well.
%From the exact sequence
%\begin{align*} 0\to I_C\to I_{C_2}\to \oO_{C_1} \to 0,  \end{align*}
%we have distinguished triangles
%\begin{align*} \RHom(\oO_{C_1},I_C)\to\RHom(I_{C_2},I_C)\to\RHom(I_C,I_C),  \end{align*}
%\begin{align*} \RHom(I_{C_2},I_C)\to\RHom(I_{C_2},I_{C_2})\to\RHom(I_{C_2},\oO_{C_1}),  \end{align*}
%By taking the cohomology, we can obtain 
%\begin{align*}
%&\Ext^1(I_C,I_C)\cong \Ext^{1}(I_{C_2},I_{C_2})\oplus \mathbb{C},  \\
%&\Ext^2(I_C,I_C)\cong\Ext^{2}(I_{C_2},I_{C_2}). 
%\end{align*}

Since the families of rational curves are disjoint from the elliptic curves, the moduli space $P_1(X,\beta)$ of stable pairs is a disjoint union of product of 
rational curve families (with curve class $\beta_1$) and $P_0(X,\beta_2)$ (where $\beta_1+\beta_2=\beta$).
And a direct calculation shows the corresponding virtual class factors as the product of the fundamental class of those rational curve families and $[P_0(X,\beta_2)]^{\mathrm{vir}}$.
For $\gamma\in H^{4}(X)$, we then have
\begin{align*}\int_{[P_{1}(X,\beta)]^{\rm{vir}}}\tau(\gamma)=\sum_{\begin{subarray}{c}\beta_1+\beta_2=\beta  \\ \beta_1, \beta_2\geqslant0 \end{subarray} }n_{0,\beta_1}(\gamma)\cdot P_{0,\beta_2}.  \end{align*}

${}$ \\
\textbf{$P_n(X, \beta)$ and generating series}. 
For the moduli space $P_{n,\beta}(X)$ of stable pairs with $n\geqslant 1$, we want to compute 
\begin{align*}\int_{[P_{n}(X,\beta)]^{\rm{vir}}}\tau(\gamma)^n, \quad \gamma\in H^4(X,\mathbb{Z}),  \end{align*}
when $X$ is an ideal CY 4-fold.
Let $\{Z_i\}_{i=1}^n$ be 4-cycles which represent the class $\gamma$. For dimension reasons,
we may assume for any $i\neq j$
the rational curves which meet with $Z_i$ are
disjoint from those with $Z_j$. 
The insertions cut out the moduli space and pick up stable pairs whose support intersects with all $\{Z_i\}_{i=1}^n$.
We denote the moduli space of such `incident' stable pairs by
\begin{align*}Q_{n}(X,\beta;\{Z_i\}_{i=1}^n)\subseteq P_{n}(X,\beta).  \end{align*}
Then we claim that 
\begin{align}\label{Q_n:identity}
Q_{n}(X,\beta;\{Z_i\}_{i=1}^n)=\coprod_{\begin{subarray}{c}
\beta_0+\beta_1+\cdots+\beta_n=\beta  
\end{subarray}}P_{0}(X,\beta_0)\times Q_{1}(X,\beta_1;Z_1)\times \cdots \times Q_{1}(X,\beta_n;Z_n), \end{align}
where $Q_{1}(X,\beta_i;Z_i)$ is the moduli space of stable pairs 
supported on rational curves (in class $\beta_i$) which meet with $Z_i$. 

Indeed let us take a stable pair $(\oO_X \to F)$ in 
$Q_{n}(X,\beta;\{Z_i\}_{i=1}^n)$. 
Then $F$ decomposes into a direct sum 
$\oplus_{i=0}^n F_i$, where 
$F_0$ is supported on elliptic curves
and 
each $F_i$ for $1\leqslant i\leqslant n$ is 
supported on rational curves which meet with 
$Z_i$. 
As explained before, a Cohen-Macaulay scheme $C$ supported in the family of rational curves (resp. elliptic curves) satisfies $\chi(\oO_C)\geqslant 1$ (resp. $\chi(\oO_C)\geqslant 0$), so $\chi(F_0)\geqslant 0$ and $\chi(F_i)\geqslant 1$ for $1\leqslant i \leqslant n$. Hence $\chi(F_0)=0$ and 
$\chi(F_i)=1$ for $1\leqslant i \leqslant n$. 
Therefore (\ref{Q_n:identity}) holds. 

Moreover each $Q_{1}(X,\beta_i;Z_i)$ consists of 
finitely many rational curves which meet with $Z_i$, 
whose number is exactly $n_{0, \beta_i}(\gamma)$. 
%\begin{align*}Q_{n}(X,\beta;\gamma)=\coprod_{\begin{subarray}{c}
%\beta_0+\beta_1+\cdots+\beta_n=\beta 
%\end{subarray}}P_{0}(X,\beta_0)\times Q_{1}(X,\beta_1;Z_1)\times \cdots \times Q_{1}(X,\beta_n;Z_n). \end{align*}
By counting the number of points in $P_{0}(X,\beta_0)$ and 
$Q_{1}(X,\beta_i;Z_i)$'s, we obtain  
\begin{align*}
P_{n, \beta}(\gamma) :=\int_{[P_{n}(X,\beta)]^{\rm{vir}}}\tau(\gamma)^n=\int_{[Q_{n}(X,\beta;\gamma)]^{\mathrm{vir}}}1 
=\sum_{\begin{subarray}{c}\beta_0+\beta_1+\cdots+\beta_n=\beta  \\ \beta_0,\beta_1,\cdots,\beta_n\geqslant0 \end{subarray} }P_{0,\beta_0}\cdot \prod_{i=1}^n n_{0,\beta_i}(\gamma). \end{align*}
The above arguments give a heuristic explanation for the formula 
\begin{align*}
\sum_{n, \beta}\frac{P_{n, \beta}(\gamma)}{n!}y^n q^{\beta}
=\prod_{\beta}\Big(\exp(yq^{\beta})^{n_{0,\beta}(\gamma)}\cdot 
M(q^{\beta})^{n_{1,\beta}}\Big)
\end{align*}
mentioned in Section~\ref{subsec:speculation}.

\section{Compact examples}
In this section, we verify Conjectures \ref{conj:GW/GV g=0} and \ref{conj:GW/GV g=1} for certain compact Calabi-Yau 4-folds.

\subsection{Sextic 4-folds}
Let $X$ be a smooth sextic 4-fold, i.e. a smooth degree six hypersurface of $\mathbb{P}^{5}$.
By the Lefschetz hyperplane theorem, $H_2(X,\mathbb{Z})\cong H_2(\mathbb{P}^{5},\mathbb{Z})\cong\mathbb{Z}$.
In order to verify our conjectures, we may use deformation invariance and assume $X$ is general in the (projective) space $\mathbb{P}\big(H^0(\mathbb{P}^{5},\oO(6))\big)$ of degree six hypersurfaces. 

${}$ \\
\textbf{Genus 0}.
For the genus 0 conjecture, we have:
\begin{prop}\label{sextic g=0}
Let $X$ be a smooth sextic 4-fold and $[l]\in H_2(X,\mathbb{Z})$ be the line class. 

Then Conjecture \ref{conj:GW/GV g=0} is true for $\beta=[l]$ and $2[l]$.
\end{prop}
\begin{proof}
In such cases, $P_{0,\beta}(X)=0$ by Proposition \ref{sextic g=1}. So we only need to 
show 
\begin{align*}P_{1,\beta}(X)(\gamma_1,\ldots,\gamma_n)=n_{0,\beta}(\gamma_1,\ldots,\gamma_n). \end{align*}
We consider $\beta=2[l]$ as the degree one case follows from the same argument.
A Cohen-Macaulay curve $C$ in $X$ with $[C]=\beta$ has $\chi(\oO_C)=1$. 
For a stable pair $(\oO_X\to F)\in P_1(X,\beta)$, there is an exact sequence 
\begin{align*}0\to \oO_C\to F\to Q\to 0,  \end{align*}
where $C$ is the support of $F$. Since $1=\chi(F)=\chi(\oO_C)+\chi(Q)$, we must have 
$Q=0$ and $F\cong \oO_C$.

When $X$ is a general sextic, $C$ is either 
a smooth conic or a pair of distinct intersecting lines (see e.g. \cite[Proposition 1.4]{Cao2}). 
The morphism\,\footnote{The map is well-defined as $\oO_C$ is stable (\cite[Proposition 2.2]{Cao2}).}
\begin{align*}P_1(X,\beta)\to M_{1,\beta}(X), \quad I=(\oO_X\to F)\mapsto F, \end{align*}
to the moduli space $M_{1,\beta}(X)$ of one dimensional stable sheaves, with $[F]=\beta$ and $\chi(F)=1$, is 
an isomorphism. Furthermore, under the isomorphism, we have identifications 
\begin{align*}\Ext^1(I,I)_0\cong \Ext^1(F,F)\cong \mathbb{C}, \end{align*}
\begin{align*}\Ext^2(I,I)_0\cong \Ext^2(F,F)=0, \end{align*}
of deformation and obstruction spaces (ref. \cite[Proposition 2.2]{Cao2}). So one can identify virtual classes
\begin{align*}[P_1(X,\beta)]^{\mathrm{vir}}=[M_{1,\beta}(X)]^{\mathrm{vir}},  \end{align*}
for a certain choice of orientation. Then Conjecture \ref{conj:GW/GV g=0} reduces to our previous conjecture 
\cite[Conjecture 0.2]{CMT}, which has been verified in this setting in \cite[Theorem 2.4]{Cao2}. 
\end{proof}

${}$ \\
\textbf{Genus 1}.
From \cite[Table 2, pp. 33]{KP}, we know genus one GV type invariants of $X$ are zero for degree one and two classes. 
In these cases, pair invariants are obviously zero.
\begin{prop}\label{sextic g=1}
Let $X$ be a smooth sextic 4-fold and $[l]\in H_2(X,\mathbb{Z})$ be the line class. 

Then Conjecture \ref{conj:GW/GV g=1} is true for $\beta=[l]$ and $2[l]$.
\end{prop}
\begin{proof}
Let $\beta=[l]$ or $2[l]$. For a stable pair $(\oO_X\to F)\in P_0(X,\beta)$, there is an exact sequence 
\begin{align*}0\to \oO_C\to F\to Q\to 0,  \end{align*}
where $C$ is the support of $F$ and $Q$ is zero dimensional. 
A Cohen-Macaulay curve $C$ in $X$ with $[C]=\beta$ has $\chi(\oO_C)\geqslant 1$, contradicting with $\chi(F)=0$.
So $P_0(X,\beta)=\emptyset$.
\end{proof}

\subsection{Elliptic fibration}
For $Y=\mathbb{P}^3$, we take general elements
\begin{align*}
u \in H^0(Y, \oO_Y(-4K_Y)), \
v \in H^0(Y, \oO_Y(-6K_Y)).
\end{align*}
Let $X$
be a CY 4-fold with an elliptic fibration
\begin{align}\label{elliptic fib}
\pi \colon X \to Y
\end{align}
given by the equation
\begin{align*}
zy^2=x^3 +uxz^2+vz^3
\end{align*}
in the $\mathbb{P}^2$-bundle
\begin{align*}
\mathbb{P}(\oO_Y(-2K_Y) \oplus \oO_Y(-3K_Y) \oplus \oO_Y) \to Y,
\end{align*}
where $[x:y:z]$ is the homogeneous coordinate of the above
projective bundle. A general fiber of
$\pi$ is a smooth elliptic curve, and any singular
fiber is either a nodal or cuspidal plane curve.
Moreover, $\pi$ admits a section $\iota$ whose image
corresond to fiber point $[0: 1: 0]$.

Let $h$ be a hyperplane in $\mathbb{P}^3$, $f$ be a general fiber of $\pi \colon X\rightarrow Y$ and set
\begin{align}\label{div:BE}
B=\pi^{\ast}h, \ E=\iota(\mathbb{P}^3)\in H_{6}(X,\mathbb{Z}).
\end{align}

${}$ \\
\textbf{Genus 0}. We consider the stable pair moduli space $P_{1}(X,[f])$ for the fiber class of $\pi$ and verify Conjecture \ref{conj:GW/GV g=0} in this case.
\begin{lem}
Let $[f]$ be the fiber class of the elliptic fibration (\ref{elliptic fib}). Then we have an isomorphism
\begin{equation}P_{1}(X,[f])\to X, \nonumber \end{equation}
under which the virtual class satisfies 
\begin{equation}[P_{1}(X,[f])]^{\mathrm{vir}}=\pm \mathrm{PD}(c_3(X)), \nonumber \end{equation}
where the sign corresponds to the choice of orientation in defining the LHS.
\end{lem}
\begin{proof}
Since $[f]$ is irreducible, we have a morphism 
\begin{equation}\phi: P_{1}(X,[f])\to M_{1,[f]}(X)\cong X, \nonumber \end{equation}
to the moduli space $M_{1,[f]}(X)$ of 1-dimensional stable sheaves on $X$ with Chern character $(0,0,0,[f],1)$ (which is isomorphic to $X$ by 
\cite[Lem. 2.1]{CMT}). The fiber of $\phi$ over $F$ is $\mathbb{P}(H^0(X,F))$ (ref. \cite[pp. 270]{PT2}).

By \cite[Lem. 2.2]{CMT}, any $F\in M_{1,[f]}(X)$ is scheme-theoretically supported on a fiber, and $F=(i_t)_*m_x^{\vee}$ for some 
$x\in X_t:=\pi^{-1}(t)$, where $i_t: X_t\to X$ is the inclusion and $m_x$ is the maximal ideal sheaf of $x$ in $X_t$. 
By Serre duality, we have  
\begin{equation}H^1(X,F)\cong H^1(X_t,m_x^{\vee})\cong H^0(X_t,m_x)^{\vee}=0. \nonumber \end{equation}
Hence $H^0(X,F)\cong \mathbb{C}$, and $\phi$ is an isomorphism.

Next, we compare the obstruction theories. Let $I=(\oO_X\to F)\in P_{1}(X,[f])$ be a stable pair.
By applying $\RHom_X(-,F)$ to $I\to \oO_X \to F$, we obtain a distinguished triangle 
\begin{equation}\RHom_X(F,F)\to \RHom_X(\oO_X,F) \to \RHom_X(I,F), \nonumber \end{equation}
whose cohomology gives an exact sequence 
\begin{equation}\label{ell fib 1} 0=H^1(X,F) \to \Ext^1_X(I,F) \to \Ext^2_X(F,F) \to H^2(X,F)=0.  \end{equation}
From the distinguished triangle
\begin{align*} 
F \to I[1] \to \oO_X[1], \end{align*}
we have the diagram
\begin{align*}
\xymatrix{    &  \dR \Gamma(\oO_X)[1] \ar@{=}[r]\ar[d] &
\dR \Gamma(\oO_X)[1] \ar[d] \\
\RHom_X(I, F) \ar[r] & \RHom_X(I, I)[1] \ar[r] \ar[d] &
\RHom_X(I, \oO_X)[1] \ar[d] \\
&  \RHom_X(I, I)_0[1]  &  \RHom_X(F, \oO_X)[2], }
\end{align*}
where the horizontal and vertical arrows are distinguished triangles.
By taking cones, we obtain a distinguished triangle 
\begin{align*}
\RHom_X(I, F) \to \RHom_X(I, I)_0[1] \to \RHom_X(F, \oO_X)[2],
\end{align*}
whose cohomology gives an exact sequence 
\begin{equation}\label{ell fib 2} 0\to \Ext^1_X(I,F) \to \Ext^2_X(I,I)_0 \to H^1(X,F)^{\vee}=0.   \end{equation}
Combining (\ref{ell fib 1}) and (\ref{ell fib 2}), we can identify the obstruction spaces
\begin{align*}
\Ext^2_X(I,I)_0\cong \Ext^2_X(F,F). \end{align*}
Then under the isomorphism $\phi$, their virtual classes can be identified. The identification of 
the virtual class of $M_{1,[f]}(X)$ with the Poincar\'e dual of the third Chern class of $X$ can be found in \cite[Lem. 2.1]{CMT}.
\end{proof}
Then by \cite[Prop. 2.3]{CMT}, we have the following
\begin{prop}\label{prop g=0 elliptic fib}
Let $\pi: X\to Y$ be the elliptic fibration (\ref{elliptic fib}). 
Then Conjecture \ref{conj:GW/GV g=0} is true for fiber class $\beta=[f]$ and $\gamma=B^2$ or $B\cdot E$ (\ref{div:BE}).
\end{prop}

${}$ \\
\textbf{Genus 1}. We consider the stable pair moduli space $P_{0}(X,r[f])$ for multiple fiber classes $r[f]$ ($r\geqslant1$) of $\pi$ and confirm Conjecture \ref{conj:GW/GV g=1} in this case.
\begin{lem}\label{lem:pair for elliptic}
For any $r \in \mathbb{Z}_{\geqslant1}$, there exists an isomorphism
\begin{equation}P_{0}(X,r[f])\cong \Hilb^{r}(\mathbb{P}^{3}),  \nonumber \end{equation}
under which the virtual class is given by
\begin{equation}[P_{0}(X,r[f])]^{\rm{vir}}=(-1)^r\cdot[\Hilb^{r}(\mathbb{P}^{3})]^{\rm{vir}},  \nonumber \end{equation}
for certain choice of orientation in defining the LHS, where $[\Hilb^{r}(\mathbb{P}^{3})]^{\rm{vir}}$ is the $\mathrm{DT_{3}}$ virtual class \cite{Thomas}.
\end{lem}
\begin{proof}
The proof is similar to the one in \cite[Proposition 6.8]{Toda2}. 
We show that the natural morphism
\begin{align}\label{HtoP}
\pi^{\ast} \colon \Hilb^r(\mathbb{P}^3) \to P_0(X, r[f])
\end{align}
is an isomorphism.
Let $(s:\mathcal{O}_{X}\rightarrow F) \in P_0(X, r[f])$
be a stable pair. 
By  
the Harder-Narasimhan and Jordan-H\"{o}lder filtrations, 
we have 
\begin{equation}0=F_{0}\subseteq F_{1}\subseteq F_{2}\subseteq\cdot\cdot\cdot\subseteq F_{n}=F,  \nonumber \end{equation}
where the quotient $E_{i}=F_{i}/F_{i-1}$'s are non-zero stable sheaves with decreasing slopes
\begin{equation}\frac{\chi(E_{1})}{r_{1}}\geqslant \frac{\chi(E_{2})}{r_{2}}\geqslant \cdots\geqslant \frac{\chi(E_{n})}{r_{n}}.  \nonumber \end{equation}
Here the slope of a zero dimensional sheaf is defined to be infinity.
%\begin{equation}p(E_{1})\geqslant p(E_{2})\geqslant \cdot\cdot\cdot\geqslant p(E_{n}).  \nonumber \end{equation}

Since $F$ is a pure one dimensional sheaf, so $E_1=F_1$ can not be zero dimensional ($r_1\geqslant 1$). 
Therefore $\ch(E_{i})=(0,0,0,r_{i}[f],\chi(E_{i}))$ for some $r_{i}\geqslant 1$. %so the above inequalities are equivalent to
%\begin{equation}\frac{\chi(E_{1})}{r_{1}}\geqslant \frac{\chi(E_{2})}{r_{2}}\geqslant \cdots\geqslant \frac{\chi(E_{n})}{r_{n}}.  \nonumber \end{equation}
The stability of $E_{i}$ implies that it is scheme theoretically supported on some fiber $X_{p_{i}}=\pi^{-1}(p_{i})$ of $\pi$, i.e. $E_{i}=(\iota_{p_{i}})_*(E_{i}')$ for some $\iota_{p_{i}}:X_{p_{i}}\hookrightarrow X$ and stable sheaf $E_{i}'\in \Coh(X_{p_{i}})$.

Since $s:\mathcal{O}_{X}\rightarrow F$ is surjective in dimension one, so is the composition $\mathcal{O}_{X}\rightarrow F\twoheadrightarrow E_{n}$.
By adjunction, there is an isomorphism
\begin{equation}\Hom_{X}(\mathcal{O}_{X},E_{n})\cong \Hom_{X_{p_{n}}}(\mathcal{O}_{X_{p_{n}}},E_{n}')\neq0, \nonumber \end{equation}
which implies that $\chi(E_{n}')\geqslant 0$, hence $\chi(E_{n})\geqslant 0$. Then
\begin{equation}\label{sum chi F}0=\chi(F)=\sum_{i=1}^{n}\chi(E_{i})\geqslant 0     \end{equation}
implies that $\chi(E_{i})=0$ for any $i$, and hence $E_{n}'\cong \mathcal{O}_{X_{p_{n}}}$ ~\cite[Proposition 1.2.7]{HL}.

By diagram chasing, we obtain a morphism
$I_{X_{p_{n}}}\rightarrow F_{n-1}$
for the 
ideal sheaf $I_{X_{p_{n}}}\subseteq \mathcal{O}_{X}$ of $X_{p_{n}}$, which is surjective in dimension one. 
Then so is the composition 
\begin{align}\label{compose}I_{X_{p_{n}}}\rightarrow F_{n-1}\twoheadrightarrow E_{n-1}.
\end{align}
We have the isomorphism 
\begin{equation}\Hom_{X}(I_{X_{p_{n}}},E_{n-1})\cong \Hom_{X_{p_{n-1}}}(\iota_{p_{n-1}}^{*}I_{X_{p_{n}}},E_{n-1}')\neq0.  \nonumber \end{equation}
Notice that $I_{X_{p_{n}}}\cong\pi^{*}I_{p_{n}}$ for ideal sheaf $I_{p_{n}}\subseteq \mathcal{O}_{\mathbb{P}^{3}}$ of $p_{n}\in\mathbb{P}^{3}$ by the flatness of $\pi$, so
\begin{align*}
&\iota_{p_{n-1}}^{*}I_{X_{p_{n}}}\cong\pi^{*}N^{\vee}_{\{p_{n-1}\}/\mathbb{P}^{3}}\cong (\mathcal{O}_{X_{p_{n-1}}})^{\oplus3}, \quad 
\textrm{if} \quad p_{n-1}=p_{n}, \\
&\iota_{p_{n-1}}^{*}I_{X_{p_{n}}}\cong \mathcal{O}_{X_{p_{n-1}}},
\quad \textrm{if} \quad p_{n-1} \neq p_n.
\end{align*}
In either case, similarly as before, we have $E_{n-1}'\cong \mathcal{O}_{X_{p_{n-1}}}$.
Moreover the morphism (\ref{compose}) is a pull-back 
of a surjection $I_{p_n} \to \oO_{p_{n-1}}$ 
by $\pi^{\ast}$. 

By repeating the above argument, 
we see that 
each $E_i$ is isomorphic to $\oO_{X_{p_i}}$, 
$s \colon \oO_X \to F$ is surjective 
and given by a pull back of a surjection 
$\oO_{\mathbb{P}^3} \to \oO_Z$ by $\pi^{\ast}$ 
for some 
zero dimensional subscheme $Z \subset \mathbb{P}^3$
with length $n$. 
Using the section $\iota$ of $\pi \colon X\rightarrow\mathbb{P}^{3}$, 
we have the morphism 
$\iota^{\ast} \colon 
P_{0}(X,r[f])\rightarrow \Hilb^{r}(\mathbb{P}^{3})
$, which gives an inverse of (\ref{HtoP}). 
Therefore the morphism (\ref{HtoP}) is an isomorphism.

%We denote the kernel of $I_{X_{p_{n}}}\rightarrow E_{n-1}$ by $I_{X_{(p_{n-1}+p_{n})}}$. By diagram tracing, we have a morphism
%$I_{X_{(p_{n-1}+p_{n})}}\rightarrow F_{n-2}$
%which is surjective in dimension one, so is the composition $I_{X_{p_{n}}}\rightarrow I_{X_{(p_{n-1}+p_{n})}}\rightarrow F_{n-2}\twoheadrightarrow E_{n-2}$.
%Then we can show $E_{n-2}'\cong \mathcal{O}_{X_{p_{n-2}}}$ as above. Inductively, we have $E_{i}\cong \mathcal{O}_{X_{p_{i}}}$ for all $i$.

%Thus $s:\mathcal{O}_{X}\rightarrow F$ is surjective and we have
%\begin{equation}0\rightarrow I_{C}\rightarrow \mathcal{O}_{X}\rightarrow F\rightarrow 0, \nonumber \end{equation}
%for some Cohen-Macaulay curve $C$ with $[C]=\sum[X_{p_{i}}]$ for some $p_{i}\in\mathbb{P}^{3}$. Conversely, for any
%Cohen-Macaulay curve $C$ with $\chi(\mathcal{O}_{C})=0$ and $C=r[f]$ (where $[f]\in H_2(X,\mathbb{Z})$ is the fiber class of $\pi$), then $(\mathcal{O}_{X}\rightarrow \mathcal{O}_{C})$ gives a stable pair in $P_{0}(X,r[f])$.

It remains to compare the virtual 
classes.  
We take $I_{Z}\in \Hilb^{r}(\mathbb{P}^{3})$ 
and use the spectral sequence
\begin{equation}
\Ext^{*}_{\mathbb{P}^{3}}(I_{Z},I_{Z}\otimes R^{*}\pi_{*}\mathcal{O}_{X})\Rightarrow \Ext^{*}_{X}(\pi^{*}I_{Z},\pi^{*}I_{Z}),   \nonumber \end{equation}
where $R^{*}\pi_{*}\mathcal{O}_{X}\cong\mathcal{O}_{\mathbb{P}^{3}}\oplus K_{\mathbb{P}^{3}}[-1]$. This gives canonical isomorphisms
\begin{align*}
&\Ext^{1}_{X}(\pi^{*}I_{Z},\pi^{*}I_{Z})\cong \Ext^{1}_{\mathbb{P}^{3}}(I_{Z},I_{Z}), \\
&\Ext^{2}_{X}(\pi^{*}I_{Z},\pi^{*}I_{Z})\cong \Ext^{2}_{\mathbb{P}^{3}}(I_{Z},I_{Z})\oplus \Ext^{2}_{\mathbb{P}^{3}}(I_{Z},I_{Z})^{\vee}.
\end{align*}
Furthermore, Kuranishi maps for deformations of $\pi^{*}I_{Z}$ on $X$ can be identified with Kuranishi maps for deformations of $I_{Z}$ on $\mathbb{P}^{3}$. Similar to ~\cite[Theorem 6.5]{CL}, we are done.
\end{proof}
\begin{prop}\label{prop g=1 elliptic fib}
Let $\pi: X\to Y$ be the elliptic fibration (\ref{elliptic fib}) and $[f]$ be the fiber class. 
Then Conjecture \ref{conj:GW/GV g=1} is true for $\beta=r[f]$ $(r\geqslant1)$, i.e.
\begin{equation}\sum_{r=0}^{\infty}P_{0,r[f]}q^{r}=M(q)^{-20},  \nonumber \end{equation}
for certain choice of orientation in defining the LHS, where $M(q)=\prod_{k\geqslant 1}(1-q^{k})^{-k}$ is the MacMahon function and we define $P_{0,0[f]}=1$.
\end{prop}
\begin{proof}
Combining Lemma \ref{lem:pair for elliptic} (where we choose sign to be $(-1)^{r}$ according to the parity of $r$) and the generating series for zero-dimensional DT invariants \cite{MNOP, Li, LP}, we obtain
the formula. Notice from \cite[Table~7]{KP}, we have $n_{1,[f]}=-20$ and $n_{1,k[f]}=0$ for $k\neq1$ (which can also be checked from GW theory).
\end{proof}

\subsection{Quintic fibration}
We consider a compact Calabi-Yau 4-fold $X$ which admits a quintic 3-fold fibration structure
\begin{equation}\pi: X\rightarrow \mathbb{P}^{1},  \nonumber \end{equation}
i.e. $\pi$ is a proper morphism whose general fiber is a smooth quintic 3-fold $Y\subseteq\mathbb{P}^{4}$.
Examples of such $\mathrm{CY}$ 4-fold contain resolution of degree 10 orbifold hypersurface in $\mathbb{P}^{5}(1,1,2,2,2,2)$ and hypersurface of bidegree $(2,5)$ in $\mathbb{P}^{1}\times \mathbb{P}^{4}$ (see \cite[pp. 33-37]{KP}). 

In this section, we discuss the irreducible curve class in a quintic fiber for these two examples.
Here we only consider genus 1 invariants.
%and leave the genus 0 case to the next section.   

${}$ \\
\textbf{Genus 1}. Conjecture \ref{conj:GW/GV g=1} predicts that for an irreducible class $\beta$
and a suitable choice of orientation, we have
\begin{equation}P_{0,\beta}=n_{1,\beta}:=\mathrm{GW}_{1,\beta}+\frac{1}{24}\mathrm{GW}_{0,\beta}(c_{2}(X)).\nonumber \end{equation}

Note that the genus 1 invariants $n_{1,\beta}$ for irreducible $\beta$ are zero for both quintic fibration examples in \cite{KP}, where  their computations of $\mathrm{GW}_{1,\beta}$ are based on BCOV theory \cite{BCOV}. The pair invariant $P_{0,\beta}$ is obviously zero in this case since
we have:
\begin{lem}\label{lem on emply pair moduli}
Let $\beta\in H_{2}(X,\mathbb{Z})$ be an irreducible class. The pair moduli space $P_{0}(X,\beta)$ is empty if and only if 
any curve $C\in \mathrm{Chow}_{\beta}(X)$ in the Chow variety is a smooth rational curve.
\end{lem}
\begin{proof}
$\Leftarrow)$ Given a stable pair $(s:\mathcal{O}_{X}\rightarrow F)\in P_{0}(X,\beta)$, then $F$ is a torsion-free sheaf (in fact a line bundle) over a curve $C\cong\mathbb{P}^{1}$. Since $\chi(F)=0$, so $F=\mathcal{O}_{C}(-1)$ contradicting with the surjectivity of $s$ in dimension 1.

$\Rightarrow)$ For $C\in \mathrm{Chow}_{\beta}(X)$ in an irreducible class $\beta$, the restriction map $(\mathcal{O}_{X}\rightarrow\mathcal{O}_{C})$ gives a stable pair. Since $P_{0}(X,\beta)$ is empty, we have
\begin{equation}\chi(\mathcal{O}_{C})=1-h^{1}(C,\mathcal{O}_{C})>0,  \nonumber \end{equation}
i.e. $h^{1}(C,\mathcal{O}_{C})=0$, which implies that $C$ is a smooth rational curve.
\end{proof}
With this lemma, we can verify Conjecture \ref{conj:GW/GV g=1} for irreducible classes in more examples. 
\begin{prop}\label{prop g=1 quintic fib}
Conjecture \ref{conj:GW/GV g=1} is true for irreducible class $\beta\in H_{2}(X,\mathbb{Z})$ when $X$ is either \\
(1) one of the quintic fibrations in \cite{KP}; \\ 
(2) a smooth complete intersection in a projective space; \\ 
(3) one of the complete intersections in Grassmannian varieties in \cite{GJ}.
\end{prop}
\begin{proof}
In all above cases, any curve $C$ in an irreducible class $[C]=\beta$ is a smooth $\mathbb{P}^{1}$, by Lemma \ref{lem on emply pair moduli}, $P_{0,\beta}(X)=\emptyset$ and hence $P_{0,\beta}=0$.
Meanwhile for those examples in (1) and (3), Klemm-Pandharipande \cite{KP} and Gerhardus-Jockers \cite{GJ} used BCOV theory \cite{BCOV} to compute genus 1 GW invariants and found that $n_{1,\beta}=0$.
As for (2), we have Popa's computation of genus 1 GW invariants using hyperplane principle developed by Li-Zinger \cite{Popa, Zinger}.
\end{proof}

\subsection{Product of elliptic curve and CY 3-fold}
In this subsection, we consider a CY 4-fold of type $X=Y\times E$, where $Y$ is a projective CY 3-fold and $E$ is an elliptic curve. 

${}$ \\
\textbf{Genus 0}.
We study Conjecture \ref{conj:GW/GV g=0} for an irreducible curve class of $X=Y\times E$. 
If $\beta=[E]$, $P_{1,\beta}=0$, the conjecture is obviously true (in fact for any $r\geqslant1$, 
one can show Conjecture \ref{conj:GW/GV g=0} is true for $\beta=r[E]$). Below we consider curve classes coming from the CY 3-fold.
%Let $\beta\in H_2(Y,\mathbb{Z})\subseteq H_2(X,\mathbb{Z})$ be an irreducible curve class.
\begin{lem}\label{lem on pair moduli on CY3}
Let $\beta\in H_2(Y,\mathbb{Z})$ be an irreducible curve class on a CY 3-fold $Y$. 

Then the pair deformation-obstruction theory of $P_n(Y,\beta)$ is perfect in the sense of \cite{BF, LT}. Hence we have 
an algebraic virtual class
\begin{equation}[P_n(Y,\beta)]_{\mathrm{pair}}^{\mathrm{vir}}\in A_{n-1}(P_n(Y,\beta),\mathbb{Z}). \nonumber \end{equation}
\end{lem}
\begin{proof}
For any stable pair $I_Y=(s:\oO_Y\to F)\in P_n(Y,\beta)$ with $\beta$ irreducible, we know $F$ is stable (ref. \cite[pp. 270]{PT2}), hence 
\begin{equation}\Ext^3_Y(F,F)\cong \Hom_Y(F,F)^{\vee}\cong\mathbb{C}. \nonumber \end{equation}
Applying $\RHom_Y(-,F)$ to $I_Y\to \oO_Y \to F$, we obtain a distinguished triangle 
\begin{equation}\label{dist triangle CY}\RHom_Y(F,F)\to \RHom_Y(\oO_Y,F) \to \RHom_Y(I_Y,F), \end{equation}
whose cohomology gives an exact sequence 
\begin{equation}0=H^2(Y,F) \to \Ext^2_Y(I_Y,F)\to \Ext^3_Y(F,F) \to 0 \to \Ext^3_Y(I_Y,F) \to 0.  \nonumber \end{equation}
Hence $\Ext^i_Y(I_Y,F)=0$ for $i\geqslant 3$ and $\Ext^2_Y(I_Y,F)\cong\Ext^3_Y(F,F)\cong\mathbb{C}$.
By truncating $\Ext^2_Y(I_Y,F)=\mathbb{C}$, the pair deformation theory is perfect. 
\end{proof}
In particular, when $n=1$, the virtual class $[P_1(Y,\beta)]_{\mathrm{pair}}^{\mathrm{vir}}$ has zero degree. 
We show the following virtual push-forward formula.
\begin{prop}\label{prop on pair obs on CY3}
Let $\beta\in H_2(Y,\mathbb{Z})$ be an irreducible curve class on a CY 3-fold $Y$. Then 
\begin{equation}\int_{[P_1(Y,\beta)]_{\mathrm{pair}}^{\mathrm{vir}}}1=\int_{[M_{1,\beta}(Y)]^{\mathrm{vir}}}1, \nonumber \end{equation}
where $M_{1,\beta}(Y)$ is the moduli scheme of 1-dimensional stable sheaves on $Y$ with Chern character $(0,0,0,\beta,1)$.
\end{prop}
\begin{proof}
Since $\beta$ is irreducible, there is a morphism 
\begin{equation}\label{map f} f: P_1(Y,\beta)\to M_{1,\beta}(Y), \quad (\oO_Y\to F)\mapsto F,  \end{equation}
whose fiber over $[F]$ is $\mathbb{P}(H^0(Y,F))$.
Let $\mathbb{F} \to M_{1,\beta}(Y) \times Y$ be the universal sheaf. 
Then the above map identifies $P_1(Y, \beta)$ with 
$\mathbb{P}(\pi_{M\ast}\mathbb{F})$
where $\pi_M \colon M_{1, \beta}(Y) \times Y \to M_{1, \beta}(Y)$
is the projection. 
Then the universal stable pair is given by 
\begin{align*}
\mathbb{I}=
(\mathcal{O}_{Y\times P_1(Y, \beta)} 
\stackrel{s}{\rightarrow} \mathbb{F}^{\dag}), \quad
\mathbb{F}^{\dag}:= (\id_Y \times f)^{\ast}\mathbb{F}\otimes \oO(1),
\end{align*} 
where $\oO(1)$ is the tautological line bundle on 
$\mathbb{P}(\pi_{M\ast}\mathbb{F})$
and $s$ is the tautological map. 

Let $\pi_P:P_1(Y,\beta)\times Y\to P_1(Y,\beta)$ be the projection, there exists a distinguished triangle
\begin{equation}\label{dis1} (\dR \hH om_{\pi_P}(\mathbb{F}^{\dag}, \mathbb{F}^{\dag})[1])^{\vee}
\rightarrow (\dR \hH om_{\pi_P}(\mathbb{I}, \mathbb{F}^{\dag}))^{\vee} 
\rightarrow (\dR \hH om_{\pi_P}(\mathcal{O}_{Y\times P_1(Y, \beta)}, \mathbb{F}^{\dag}))^{\vee}.   \end{equation}
%induced from distinguished triangle $\mathbb{I}\to \mathcal{O}_{X\times P_1(X, \beta)} \to \mathbb{F}^{\dag}$
By considering a derived extension of the morphism $f$ (\ref{map f}), the first two terms in (\ref{dis1}) are the restriction of cotangent complexes of the corresponding derived schemes to the classical underlying schemes. 
They are obstruction theories (see \cite[Sect. 1.2]{STV}),
which fit into a commutative diagram 
\begin{align*}
\xymatrix{
\big(\dR \hH om_{\pi_P}(\mathbb{F}^{\dag}, \mathbb{F}^{\dag})[1]\big)^{\vee} \ar[r] \ar[d] & \big(\dR \hH om_{\pi_P}(\mathbb{I}, \mathbb{F}^{\dag})\big)^{\vee} \ar[r] \ar[d]_{  } & \big(\dR\hH om_{\pi_P}(\mathcal{O}_{Y\times P_1(Y, \beta)}, \mathbb{F}^{\dag})\big)^{\vee}\ar[d]_{  } \\  
f^*\mathbb{L}_{M_{1,\beta}(Y)}\ar[d]  \ar[r] &  \mathbb{L}_{P_1(Y,\beta)}  \ar[r] \ar[d] & \mathbb{L}_{P_1(Y,\beta)/M_{1,\beta}(Y)}\ar[d]  \\
\tau^{\geqslant -1}(f^*\mathbb{L}_{M_{1,\beta}(Y)})  \ar[r] &  \tau^{\geqslant -1}\mathbb{L}_{P_1(Y,\beta)}  \ar[r]  &  \tau^{\geqslant -1}\mathbb{L}_{P_1(Y,\beta)/M_{1,\beta}(Y)},  }
\end{align*}
where the bottom vertical arrows are truncation functors.

Note the above obstruction theories are not perfect. To kill $h^{-2}$, as in \cite[Sect. 4.4]{ht}, we consider the top part of trace map  
\begin{align*}t: \dR \hH om_{\pi_P}(\mathbb{F}^{\dag}, \mathbb{F}^{\dag})[1]\to 
\dR^3 \pi_{P*}(\oO_{Y\times P_1(Y, \beta)})[-2],\end{align*}
whose cone is $\Big(\tau^{\leqslant 1}\big(\dR \hH om_{\pi_P}(\mathbb{F}^{\dag}, \mathbb{F}^{\dag})[1]\big)\Big)[1]$.
Then we have a commutative diagram 
\begin{align*}
\xymatrix{   \big(\dR^3 \pi_{P*}(\oO_{Y\times P_1(Y, \beta)})[-2]\big)^{\vee} \ar@{=}[r]\ar[d]_{t^{\vee}}   &  
\big(\dR^3 \pi_{P*}(\oO_{Y\times P_1(Y, \beta)})[-2]\big)^{\vee}  \ar[d]^{\alpha} &  \\
\big(\dR \hH om_{\pi_P}(\mathbb{F}^{\dag}, \mathbb{F}^{\dag})[1]\big)^{\vee} \ar[r] \ar[d]& 
\big(\dR \hH om_{\pi_P}(\mathbb{I}, \mathbb{F}^{\dag})\big)^{\vee} \ar[r] \ar[d] &
\big(\dR\hH om_{\pi_P}(\mathcal{O}_{Y\times P_1(Y, \beta)}, \mathbb{F}^{\dag})\big)^{\vee}   \\
\big(\tau^{\leqslant1}\big(\dR \hH om_{\pi_P}(\mathbb{F}^{\dag}, \mathbb{F}^{\dag})[1]\big)\big)^{\vee}\ar[r] & 
 \mathrm{Cone}(\alpha). &  }
\end{align*}
By taking cones, we obtain a distinguished triangle 
\begin{equation}\big(\tau^{\leqslant1}\big(\dR \hH om_{\pi_P}(\mathbb{F}^{\dag}, \mathbb{F}^{\dag})[1]\big)\big)^{\vee}
\rightarrow \mathrm{Cone}(\alpha) \rightarrow
\big(\dR\hH om_{\pi_P}(\mathcal{O}_{Y\times P_1(Y, \beta)}, \mathbb{F}^{\dag})\big)^{\vee}. \nonumber \end{equation}
Since $\big(\dR^3 \pi_{P*}(\oO_{Y\times P_1(Y, \beta)})[-2]\big)^{\vee}$ is a vector bundle concentrated in degree $-2$ and 
$\tau^{\geqslant -1}(-)$ has cohomology in degree greater than $-2$, so we have a commutative diagram
\begin{align*}
\xymatrix{\big(\tau^{\leqslant1}\big(\dR \hH om_{\pi_P}(\mathbb{F}^{\dag}, \mathbb{F}^{\dag})[1]\big)\big)^{\vee} \ar[r] \ar[d] & 
\mathrm{Cone}(\alpha)   \ar[r] \ar[d]_{  } & \big(\dR\hH om_{\pi_P}(\mathcal{O}_{Y\times P_1(Y, \beta)}, \mathbb{F}^{\dag})\big)^{\vee}\ar[d]_{  } \\  
\tau^{\geqslant -1}(f^*\mathbb{L}_{M_{1,\beta}(Y)})  \ar[r] &  \tau^{\geqslant -1}\mathbb{L}_{P_1(Y,\beta)}  \ar[r]  &  \tau^{\geqslant -1}\mathbb{L}_{P_1(Y,\beta)/M_{1,\beta}(Y)}.} \end{align*}
To kill $h^1$ of the left upper term, we consider the inclusion 
\begin{align*}\oO_{P_1(Y,\beta)}[1] \to \tau^{\leqslant1}\big(\dR \hH om_{\pi_P}(\mathbb{F}^{\dag}, \mathbb{F}^{\dag})[1]\big),  \end{align*}
whose restriction to a closed point $I=(\oO_Y\to F)$ induces an isomorphism $\mathbb{C} \to \Hom(F,F)$.
The cone of the inclusion is $\tau^{[0,1]}\big(\dR \hH om_{\pi_P}(\mathbb{F}^{\dag}, \mathbb{F}^{\dag})[1]\big)$. 

Then we have a commutative diagram 
\begin{align*}
\xymatrix{
\mathrm{Cone}(\beta) \ar[r] \ar[d] & \big(\tau^{[0,1]}\big(\dR \hH om_{\pi_P}(\mathbb{F}^{\dag}, \mathbb{F}^{\dag})[1]\big)\big)^{\vee}  \ar[d]_{  } &   \\  
(\dR \hH om_{\pi_P}(\mathcal{O}_{Y\times P_1(Y, \beta)}, \mathbb{F}^{\dag}))^{\vee}[-1]\ar[d]_{\beta}  \ar[r] &   \big(\tau^{\leqslant1}\big(\dR \hH om_{\pi_P}(\mathbb{F}^{\dag}, \mathbb{F}^{\dag})[1]\big)\big)^{\vee} \ar[r] \ar[d] &  \mathrm{Cone}(\alpha)  \\
(\oO_{P_1(Y,\beta)}[1])^{\vee} \ar@{=}[r] &  (\oO_{P_1(Y,\beta)}[1])^{\vee}.   &   }
\end{align*}
As $(\oO_{P_1(Y,\beta)}[1])^{\vee}$ is a vector bundle concentrated on degree $1$, so we get commutative diagram
\begin{align*}
\xymatrix{\big(\tau^{[0,1]}\big(\dR \hH om_{\pi_P}(\mathbb{F}^{\dag}, \mathbb{F}^{\dag})[1]\big)\big)^{\vee} \ar[r] \ar[d]_{\phi_1} & 
\mathrm{Cone}(\alpha)   \ar[r] \ar[d]_{\phi_2} & \mathrm{Cone}(\beta)[1] \ar[d]_{\phi_3} \\  
\tau^{\geqslant -1}(f^*\mathbb{L}_{M_{1,\beta}(Y)})  \ar[r] &  \tau^{\geqslant -1}\mathbb{L}_{P_1(Y,\beta)}  \ar[r]  &  \tau^{\geqslant -1}\mathbb{L}_{P_1(Y,\beta)/M_{1,\beta}(Y)}.} \end{align*}
It is easy to see that $\phi_1$ and $\phi_2$ define perfect obstruction theories. By diagram chasing on cohomology, $\phi_3$ defines a perfect relative obstruction theory.
Then we apply Manolache's virtual push-forward formula \cite{Manolache}:
\begin{equation}f_*[P_1(Y,\beta)]_{\mathrm{pair}}^{\mathrm{vir}}=c\cdot [M_{1,\beta}(Y)]^{\mathrm{vir}}, \nonumber \end{equation}
where the coefficient $c$ is the degree of the virtual class of the relative obstruction theory $\phi_3$ and can be shown to be $1$ by 
base-change to a closed point.
\end{proof}
Now we come back to CY 4-fold $X=Y\times E$ and show the virtual class $[P_1(Y,\beta)]_{\mathrm{pair}}^{\mathrm{vir}}$
defined using pair deformation-obstruction theory naturally arises in this setting.
\begin{prop}\label{prop on pair moduli on CY3}
Let $X=Y\times E$ be a product of a CY 3-fold $Y$ with an elliptic curve $E$. For an irreducible curve class $\beta\in H_2(Y,\mathbb{Z})\subseteq H_2(X,\mathbb{Z})$, we have an isomorphism 
\begin{equation}P_n(X,\beta)\cong P_n(Y,\beta)\times E. \nonumber \end{equation}
The virtual class of $P_n(X,\beta)$ satisfies 
\begin{equation}[P_n(X,\beta)]^{\mathrm{vir}}=[P_n(Y,\beta)]_{\mathrm{pair}}^{\mathrm{vir}}\otimes [E], \nonumber \end{equation}
for certain choice of orientation in defining the LHS.
Here $[P_n(Y,\beta)]_{\mathrm{pair}}^{\mathrm{vir}}\in A_{n-1}(P_n(Y,\beta),\mathbb{Z})$ is the virtual class defined in Lemma \ref{lem on pair moduli on CY3}.
\end{prop}
\begin{proof}
As $\beta$ is irreducible, for $I_X=(s: \oO_X\to E)\in P_n(X,\beta)$, $E$ is stable (ref. \cite[pp. 270]{PT2}), hence 
$E$ is scheme theoretically supported on some $Y\times \{t\}$, $t\in E$ (e.g. \cite[Lem. 2.2]{CMT}). 

Let $i_t: Y\times \{t\} \to X$ be the inclusion, then $E=(i_t)_*F$ for some $F\in \Coh(Y)$. By adjunction, we have 
\begin{equation}\Hom_X(\oO_X,E)\cong \Hom_Y(\oO_Y,F).  \nonumber \end{equation}
Hence, the morphism  
\begin{equation}\label{iso of pair on cy3}P_n(Y,\beta)\times E\to P_n(X,\beta), \end{equation}
\begin{equation}\big(I_Y:=(s:i_t^*\oO_X\to F),t\big)\mapsto (s:\oO_X\to (i_t)_*F):=I_X  \nonumber \end{equation}
is bijective on closed points. Next, we compare their deformation-obstruction theories.

Denote $i=i_t$. From the distinguished triangle
\begin{align}\label{triangle CY3-fold 1}
i_{\ast}F \to I_X[1] \to \oO_X[1], \end{align}
we have the diagram
\begin{align*}
\xymatrix{    &  \dR \Gamma(\oO_X)[1] \ar@{=}[r]\ar[d] &
\dR \Gamma(\oO_X)[1] \ar[d] \\
\RHom_X(I_X, i_{\ast}F) \ar[r] & \RHom_X(I_X, I_X)[1] \ar[r] \ar[d] &
\RHom_X(I_X, \oO_X)[1] \ar[d] \\
&  \RHom_X(I_X, I_X)_0[1]  &  \RHom_X(i_{\ast}F, \oO_X)[2], }
\end{align*}
where the horizontal and vertical arrows are distinguished triangles.
By taking cones, we obtain a distinguished triangle 
\begin{align}\label{triangle CY3-fold 2}
\RHom_X(I_X, i_{\ast}F) \to \RHom_X(I_X, I_X)_0[1] \to \RHom_X(i_{\ast}F, \oO_X)[2].
\end{align}
On the other hand, from the distinguished triangle 
$$I_X\to \oO_X \to i_*F, $$
and the isomorphism (see e.g. \cite[Proposition-Definition 3.3]{CMT}):
$$\dL i^*i_*F\cong F\oplus (F\otimes N_{Y\times\{t\}/X}^{\vee})[1], \quad \mathrm{where} \,\,\, N_{Y\times\{t\}/X}=\oO_{Y\times\{t\}}, $$
we can obtain the isomorphism
\begin{align}\label{iso of i^*IX}\dL i^{\ast}I_X \cong I_Y \oplus F, \end{align}
which implies that
\begin{align*}\RHom_X(I_X, i_{\ast}F) &\cong\RHom_Y(I_Y, F) \oplus \RHom_Y(F, F ).  \end{align*}
%By Serre duality, we have
%\begin{align*}
%\RHom_Y(F, F \otimes K_Y)[-1] \cong\RHom_Y(F, F)^{\vee}[-4], \\RHom_X(i_{\ast}F, \oO_X)[2]\cong \dR \Gamma(F)^{\vee}[-2].\end{align*}
Therefore by (\ref{triangle CY3-fold 2}),
we have the distinguished triangle
\begin{align*}
\RHom_Y(I_Y, F) \oplus \RHom_Y(F, F) \to \RHom_X(I_X, I_X)_0[1] \to \RHom_X(i_{\ast}F, \oO_X)[2].
\end{align*}
It follows that we have the distinguished triangle
\begin{align}\label{triangle CY3-fold 3}
\RHom_Y(I_Y, F) \to \RHom_X(I_X, I_X)_0[1] \to T,  \end{align}
where $T$ fits into the distinguished triangle
\begin{align}\label{triangle CY3-fold 3.5}\RHom_Y(F, F) \to T \to \RHom_X(i_{\ast}F, \oO_X)[2]. \end{align}
By Serre duality, adjunction and degree shift, (\ref{triangle CY3-fold 3.5}) becomes
\begin{align*}  
T  \to  \RHom_Y(\oO_Y,F)^{\vee}[-2] \to \RHom_Y(F, F)^{\vee}[-2], \end{align*}
whose dual gives a distinguished triangle
\begin{align}\label{triangle CY3-fold 4}
\RHom_Y(F, F)[2] \to \RHom_Y(\oO_Y,F)[2] \to  T^{\vee}. \end{align}
Combining (\ref{dist triangle CY}), (\ref{triangle CY3-fold 4}), we obtain
\begin{equation}T\cong \RHom_Y(I_Y,F)^{\vee}[-2].  \nonumber \end{equation}
Combining with (\ref{triangle CY3-fold 3}) and taking the cohomological long exact sequence, we have 
\begin{align*}\to \Ext_Y^1(I_Y, F) \to \Ext_X^2(I_X, I_X)_0 \to \Ext_Y^1(I_Y, F)^{\vee} \to.  \end{align*}
We claim the above exact sequence breaks into short exact sequences
\begin{align*}0 \to \Ext_Y^0(I_Y, F) \to \Ext_X^1(I_X, I_X)_0 \to \mathbb{C} \to 0, \end{align*}
\begin{align*}0 \to \Ext_Y^1(I_Y, F) \to \Ext_X^2(I_X, I_X)_0 \to \Ext_Y^1(I_Y, F)^{\vee} \to 0,  \end{align*}
since $\Ext_Y^2(I_Y, F)\cong \mathbb{C}$ (see the proof of Lemma \ref{lem on pair moduli on CY3})
and a dimension counting by Riemann-Roch.
The first exact sequence above implies that the map (\ref{iso of pair on cy3}) induces an isomorphism on tangent spaces. The second
exact sequence implies that the obstructions of deforming stable pairs on LHS of (\ref{iso of pair on cy3}) vanish if and only if 
those on RHS of (\ref{iso of pair on cy3}) vanish. Therefore, the map (\ref{iso of pair on cy3}) induces an isomorphism on formal completions of
structure sheaves of both sides at any closed point. So (\ref{iso of pair on cy3}) must be a scheme theoretical isomorphism.

Next, we show $\Ext_Y^1(I_Y, F)\subseteq \Ext_X^2(I_X, I_X)_0$ is a maximal isotropic subspace with respect the Serre duality pairing on
$\Ext_X^2(I_X, I_X)_0$.
For $u \in \Ext_Y^1(I_Y, F)$, the corresponding
element in $\Ext_X^2(I_X, I_X)_0$ is given by
the composition
\begin{align*}
I_X \stackrel{\alpha}{\to} i_{\ast}I_Y \stackrel{i_{\ast}u}{\to}
i_{\ast}F[1] \stackrel{\beta[1]}{\to} I_X[2],
\end{align*}
where the morphism $\alpha$ is the canonical morphism
and $\beta$ is given by (\ref{triangle CY3-fold 1}).
For another $u' \in \Ext_Y^1(I_Y, F)$,
it is enough to show the vanishing of the
composition
\begin{align}\label{compose:pair CY}
I_X \stackrel{\alpha}{\to} i_{\ast}I_Y \stackrel{i_{\ast}u}{\to}
i_{\ast}F[1] \stackrel{\beta[1]}{\to} I_X[2]
\stackrel{\alpha[2]}{\to} i_{\ast}I_Y[2] \stackrel{i_{\ast}u'[2]}{\to}
i_{\ast}F[3] \stackrel{\beta[3]}{\to} I_X[4].
\end{align}
Since $\Ext^0_Y(F,I_Y\otimes K_Y)\cong \Ext^3_Y(I_Y,F)^\vee=0$ (see the proof of Lemma \ref{lem on pair moduli on CY3}), the composition
$i_{\ast}F[1] \stackrel{\beta[1]}{\to} I_X[2]
\stackrel{\alpha[2]}{\to} i_{\ast}I_Y[2]$ can be written as $i_{\ast}\gamma$. Therefore the composition
\begin{align*}
i_{\ast}
I_Y \stackrel{i_{\ast}u}{\to}
i_{\ast}F[1] \stackrel{\beta[1]}{\to} I_X[2]
\stackrel{\alpha[2]}{\to} i_{\ast}I_Y[2] \stackrel{i_{\ast}u'[2]}{\to}
i_{\ast}F[3]
\end{align*}
vanishes, again by $\Ext_Y^3(I_Y, F)=0$.

Moreover, a local Kuranishi map of $P_n(X,\beta)$ at $I_X$ can be identified as
\begin{equation}(\kappa_{I_Y},0):\Ext_Y^0(I_Y, F)\times T_tE \to\Ext_Y^1(I_Y, F),  \nonumber \end{equation}
where $\kappa_{I_Y}$ is a local Kuranishi map 
of $P_n(Y,\beta)$ at $I_Y$. Similarly as \cite[Thm. 6.5]{CL}, we have the desired equality on virtual classes.
\end{proof}
Combining the above result with Proposition \ref{prop on pair obs on CY3}, our genus zero conjecture can be reduced to 
Katz's conjecture \cite{Katz}.
\begin{cor}\label{reduce to Katz conj}
Let $X=Y\times E$ be a product of a CY 3-fold $Y$ with an elliptic curve $E$.

Then Conjecture \ref{conj:GW/GV g=0} holds for an irreducible curve class $\beta\in H_2(Y,\mathbb{Z})\subseteq H_2(X,\mathbb{Z})$ if and only if Katz's conjecture holds for $\beta$.
\end{cor}
\begin{proof}
To have non-trivial invariants, we only need to consider insertions of form
\begin{equation}\gamma=(\gamma_1,[\mathrm{pt}])\in H^2(Y,\mathbb{Z})\otimes H^2(E,\mathbb{Z}). \nonumber \end{equation}
By Proposition \ref{prop on pair obs on CY3} and \ref{prop on pair moduli on CY3}, we have 
\begin{equation}P_{1,\beta}(\gamma)=(\gamma_1\cdot \beta)\int_{[P_1(Y,\beta)]_{\mathrm{pair}}^{\mathrm{vir}}}1=(\gamma_1\cdot \beta)\int_{[M_{1,\beta}(Y)]^{\mathrm{vir}}}1. \nonumber \end{equation}
Then Conjecture \ref{conj:GW/GV g=0} reduces to Katz's conjecture.
\end{proof}
Katz's conjecture has been verified for primitive classes in complete intersection CY 3-folds \cite[Cor. A.6]{CMT}. 
So we obtain 
\begin{thm}\label{thm on g=0 irr from CY3}
Let $Y$ be a complete intersection CY 3-fold in a product of projective spaces, $X=Y\times E$ be the product of $Y$ with an elliptic curve $E$. 
Then Conjecture \ref{conj:GW/GV g=0} is true for an irreducible curve class $\beta\in H_2(Y,\mathbb{Z})\subseteq H_2(X,\mathbb{Z})$.
\end{thm}

${}$ \\
\textbf{Genus 1}. Similar to Lemma \ref{lem:pair for elliptic}, for $X=Y\times E$ and $\beta=r[E]$, we have
\begin{lem}\label{pair inv for elliptic fiber}
For any $r \in \mathbb{Z}_{\geqslant1}$, there exists an isomorphism
\begin{equation}P_{0}(X,r[E])\cong \Hilb^{r}(Y),  \nonumber \end{equation}
under which the virtual class is given by
\begin{equation}[P_{0}(X,r[E])]^{\rm{vir}}=(-1)^r\cdot[\Hilb^{r}(Y)]^{\rm{vir}},  \nonumber \end{equation}
for certain choice of orientation in defining the LHS.

Furthermore, their degrees fit into the generating series  
\begin{equation}\sum_{r=0}^{\infty}P_{0,r[E]}q^{r}=M(q)^{\chi(Y)},  \nonumber \end{equation}
where $M(q)=\prod_{k\geqslant1}(1-q^{k})^{-k}$ is the MacMahon function and we define $P_{0,0[E]}=1$.
\end{lem}
We check Conjecture \ref{conj:GW/GV g=1} for this case.
\begin{thm}\label{g=1 product of elliptic curve and CY3}
Let $X=Y\times E$ be the product of a CY 3-fold $Y$ with an elliptic curve $E$.
Then Conjecture \ref{conj:GW/GV g=1} is true for $\beta=r[E]\in H_{2}(X,\mathbb{Z})$ for any $r\geqslant 1$.
\end{thm}
\begin{proof}
By Lemma \ref{pair inv for elliptic fiber}, we are left to show $n_{1,[E]}=\chi(Y)$ and $n_{1,r[E]}=0$ if $r\geqslant2$. Since genus zero Gromov-Witten invariants $\mathrm{GW}_{0,r[E]}(X)=0$ for any $r\geqslant1$, this is equivalent to
\begin{equation}\sum_{r=1}^{\infty} \mathrm{GW}_{1,r[E]}(X)\,q^{r}=\chi(Y)\cdot\sum_{d=1}^{\infty}\frac{\sigma(d)}{d}q^{d}, \nonumber \end{equation}
where $\sigma(d)=\sum_{i|d}i$. We have an isomorphism
\begin{equation}\overline{M}_{1,0}(X,r[E])\cong \overline{M}_{1,0}(E,r[E])\times Y   \nonumber \end{equation}
for moduli space $\overline{M}_{1,0}(X,r[E])$ of genus 1 stable maps to $X$. Note that $\overline{M}_{1,0}(E,r[E])$ is smooth of expected dimension and consists of $\frac{\sigma(r)}{r}$ points (modulo automorphisms) (see e.g. \cite{P}).
And the genus one invariant for constant map to $Y$ is $\chi(Y)$. So $\mathrm{GW}_{1,r[E]}(X)=\chi(Y)\cdot\frac{\sigma(r)}{r}$.
%\begin{equation}  \nonumber \end{equation}
\end{proof}
When the curve class $\beta\in H_2(Y)\subseteq H_2(X)$ comes from $Y$, we have 
\begin{lem}\label{compute GW inv of class from Y}
Let $X=Y\times E$ be the product of a CY 3-fold $Y$ with an elliptic curve $E$. 
Then for $\beta\in H_2(Y)\subseteq H_2(X)$, we have 
\begin{align*}\mathrm{GW}_{0,\beta}(\gamma)=\deg[\overline{M}_{0,0}(Y,\beta)]^{\mathrm{vir}}\cdot\int_\beta\gamma_1\cdot\int_E\gamma_2,
\quad \mathrm{if} \,\, \gamma=\gamma_1\otimes \gamma_2\in H^2(Y)\otimes H^2(E);  \end{align*}
\begin{align*}\mathrm{GW}_{0,\beta}(\gamma)=0, \quad \mathrm{if} \,\, \gamma\in H^4(Y)\subseteq H^4(X); \quad \quad 
\mathrm{GW}_{1,\beta}=0.  \end{align*}
\end{lem}
\begin{proof}
We have an isomorphism 
\begin{align}\label{equ prod str}\overline{M}_{0,1}(X,\beta)\cong \overline{M}_{0,1}(Y,\beta)\times E,  \end{align}   
under which the virtual class satisfies 
\begin{align*}[\overline{M}_{0,1}(X,\beta)]^{\mathrm{vir}}\cong [\overline{M}_{0,1}(Y,\beta)]^{\mathrm{vir}}\otimes [E]. \end{align*}   
By divisor equation, one can compute genus zero GW invariants. 
$\overline{M}_{1,0}(X,\beta)$ has a similar product structure as (\ref{equ prod str}). The
obstruction sheaf has a trivial factor $TE=\oO_E$ in $E$ direction. So genus one GW invariants vanish.
\end{proof}
Then it is easy to show the following:
\begin{prop}\label{g=1 conj for irr class from CY3}
Let $X=Y\times E$ be the product of a CY 3-fold $Y$ with an elliptic curve $E$. 
Then Conjecture \ref{conj:GW/GV g=1} is true for any irreducible class $\beta\in H_2(Y)\subseteq H_2(X)$.
\end{prop}
\begin{proof}
By Lemma \ref{compute GW inv of class from Y}, we know $n_{1,\beta}=0$. By Proposition \ref{prop on pair moduli on CY3}, the
virtual dimension of $[P_0(Y,\beta)]_{\mathrm{pair}}^{\mathrm{vir}}$ is negative, so $P_{0,\beta}=0$.
\end{proof}

\iffalse
Note that for general curve classes $\beta\in H_2(Y)\subseteq H_2(X)$, $P_{0,\beta}$ \textbf{still trivial}. 
We have the following example: let $p:Y\to \mathbb{P}^2$ be an elliptic fibration given by a Weierstrass model, then
\begin{align}\label{ell fib}\pi=p\times \mathrm{id}: X\to \mathbb{P}^2\times E  \end{align}
is an elliptic fibration. Similar to Proposition \ref{prop g=1 elliptic fib}, we have 
\begin{prop}
Let $\pi: X\to \mathbb{P}^2\times E$ be the above elliptic fibration (\ref{ell fib}) and $[f]$ be the fiber class. 
Then 
\begin{equation}\sum_{r=0}^{\infty}P_{0,r[f]}q^{r}=M(q)^{-9},  \nonumber \end{equation}
for certain choice of orientation in defining the LHS, where $M(q)=\prod_{k\geqslant 1}(1-q^{k})^{-k}$ is the MacMahon function and we define $P_{0,0[f]}=1$.
\end{prop}
The corresponding meeting invariants can be computed as follows: \textbf{TBC}...
\fi

\iffalse
\begin{thm}
Let $X=Y\times E$ be a product of a CY 3-fold $Y$ with an elliptic curve $E$, and $\beta\in H_2(Y,\mathbb{Z})\subseteq H_2(X,\mathbb{Z})$ be
an irreducible curve class on $X$.

Then Conjecture \ref{conj:GW/GV g=0} is true.
\end{thm}
\begin{proof}

\end{proof}
\fi

%\subsection{K3 fibration}TBC....

\subsection{Hyperk\"{a}hler 4-folds}
When the CY 4-fold $X$ is hyperk\"{a}hler, GW invariants on $X$ vanish as they are deformation-invariant and there are no holomorphic curves
for generic complex structures in the $\mathbb{S}^{2}$-twistor family. Another way to see the vanishing is via the cosection localization technique developed by Kiem-Li \cite{KL}.

Roughly speaking, given a perfect obstruction theory \cite{BF, LT} on a Deligne-Mumford moduli stack $M$, the existence of a cosection
\begin{equation}\varphi: \mathrm{Ob}_{M}\rightarrow \mathcal{O}_M   \nonumber \end{equation}
of the obstruction sheaf $\mathrm{Ob}_{M}$ makes virtual class of $M$ localize to closed subspace $Z(\varphi)\subseteq M$ where $\varphi$ is not surjective.
In particular, if $\varphi$ is surjective everywhere (in GW theory this is guaranteed by the existence of holomorphic symplectic forms), then the virtual class of $M$ vanishes. Moreover, by truncating the obstruction theory to remove the trivial factor $\mathcal{O}_M$, one can define a reduced obstruction theory and reduced virtual class.

To verify Conjectures \ref{conj:GW/GV g=0} and \ref{conj:GW/GV g=1} for hyperk\"{a}hler 4-folds, we only need to show the vanishing of stable pair invariants of $P_0(X,\beta)$ and $P_1(X,\beta)$.

${}$ \\
\textbf{Cosection and vanishing of $\mathrm{DT_4}$ virtual classes}.
Fix a stable pair $I\in P_n(X,\beta)$, by taking wedge product with square $\mathrm{At}(I)^{2}$ of the Atiyah class and contraction with the holomorphic symplectic form $\sigma$, we get a surjective map
\begin{equation}\xymatrix@1{
\phi:\Ext^{2}(I,I)_0\ar[r]^{\wedge\frac{\mathrm{At}(I)^{2}}{2}}& \Ext^{4}(I,I\otimes \Omega^{2}_X) \ar[r]^{\quad \lrcorner \sigma} & \Ext^{4}(I,I)
\ar[r]^{tr} & H^{4}(X,\mathcal{O}_X)}.  \nonumber \end{equation}
In fact, we have
\begin{prop}\label{surj cosection}
Let $X$ be a projective hyperk\"{a}hler 4-fold, $I$ be a perfect complex on $X$ and $Q$ be the Serre duality quadratic form on $\Ext^{2}(I,I)_0$. Then the composition map
\begin{equation}\xymatrix@1{
\phi:\Ext^{2}(I,I)_0\ar[r]^{\wedge\frac{\mathrm{At}(I)^{2}}{2}}& \Ext^{4}(I,I\otimes \Omega^{2}_X) \ar[r]^{\quad \lrcorner \sigma} & \Ext^{4}(I,I)
\ar[r]^{tr} & H^{4}(X,\mathcal{O}_X)}.  \nonumber \end{equation}is surjective if either $\ch_{3}(I)\neq0$ or $\ch_{4}(I)\neq0$.
Moreover,   \\
(1) if $\ch_{4}(I)\neq0$, then we have a $Q$-orthogonal decomposition
\begin{equation}\Ext^{2}(I,I)_0=\Ker(\phi)\oplus\mathbb{C}\langle \mathrm{At}(I)^2\lrcorner\textrm{ } \sigma \rangle,   \nonumber \end{equation}
where $Q$ is non-degenerate on each subspace; \\
(2) if $\ch_{4}(I)=0$ and $\ch_{3}(I)\neq0$, then we have a $Q$-orthogonal decomposition
\begin{equation}\Ext^{2}(I,I)_0=\mathbb{C}\langle \mathrm{At}(I)^2\lrcorner\textrm{ } \sigma, \kappa_{X}\circ \mathrm{At}(I)\rangle \oplus
(\mathbb{C}\langle \mathrm{At}(I)^2\lrcorner\textrm{ } \sigma, \kappa_{X}\circ \mathrm{At}(I)\rangle)^{\perp},  \nonumber \end{equation}
where $Q$ is non-degenerate on each subspace. Here $\kappa_{X}$ is the Kodaira-Spencer class which is Serre dual to $\ch_3(I)$.
%(2) if $ch_{4}(F)=0$ and $ch_{3}(F)\neq0$, then $\mathrm{At}(F)^2\lrcorner\textrm{ } \sigma\neq0\in\Ker(\phi)$ is the unique (up to multiplying a constant) element such that $Q(\mathrm{At}(F)^2\lrcorner\textrm{ } \sigma,\Ker(\phi))=0$. In particular, $Q$ is degenerate on the subspace $\Ker(\phi)\subset \Ext^{2}(F,F)$.
\end{prop}
\begin{proof}
See the proof of \cite[Prop. 2.9]{CMT}.
\end{proof}
We claim that the surjectivity of cosection maps leads to the vanishing of virtual classes for stable pair moduli spaces (it also applies to other moduli spaces, e.g. Hilbert schemes of curves/points used in DT/PT correspondence \cite{CK1, CK2}).  
\begin{claim}\label{vanishing for hk4}
Let $X$ be a projective hyperk\"{a}hler 4-fold and $P_n(X,\beta)$ be the moduli space of stable pairs with $n\neq0$ or $\beta\neq 0$. 
Then the virtual class satisfies
\begin{equation}[P_n(X,\beta)]^{\rm{vir}}=0.  \nonumber \end{equation}
\end{claim}
At the moment, Kiem-Li type theory of cosection localization for D-manifolds is not available in the literature. 
We believe that when such a theory is established, our claim should follow automatically. 
Nevertheless, we have the following evidence for the claim.

1. At least when $P_n(X,\beta)$ is smooth, Proposition \ref{surj cosection} gives the vanishing of virtual class.

2. If there is a complex analytic version of $(-2)$-shifted symplectic geometry \cite{PTVV} and the corresponding construction of virtual classes \cite{BJ},
one could prove the vanishing result as in $\mathrm{GW}$ theory, i.e. taking a generic complex structure in the $\mathbb{S}^{2}$-twistor family 
of the hyperk\"ahler 4-fold which does not support coherent sheaves and then vanishing of virtual classes follows from their deformation invariance.

\section{Non-compact examples}

\subsection{Irreducible curve classes on local Fano 3-folds}
Let $Y$ be a Fano 3-fold. When $Y$ embeds into a CY 4-fold $X$, the normal bundle of $Y\subseteq X$ is the canonical bundle $K_Y$ of $Y$. By the negativity of $K_Y$, there exists an analytic neighbourhood of $Y$ in $X$ which is isomorphic to an analytic neighbourhood of $Y$ in $K_Y$. Here we simply consider non-compact CY 4-folds of form $X=K_Y$.

Similar to Lemma \ref{lem on pair moduli on CY3}, we have 
\begin{lem}\label{lem on pair moduli on Y}
Let $\beta\in H_2(Y,\mathbb{Z})$ be an irreducible curve class on a Fano 3-fold $Y$. 

Then the pair deformation-obstruction theory of $P_n(Y,\beta)$ is perfect in the sense of \cite{BF, LT}. Hence we have 
an algebraic virtual class
\begin{equation}[P_n(Y,\beta)]_{\mathrm{pair}}^{\mathrm{vir}}\in A_{n}(P_n(Y,\beta),\mathbb{Z}). \nonumber \end{equation}
\end{lem}
\begin{proof}
For any stable pair $I_Y=(s:\oO_Y\to F)\in P_n(Y,\beta)$ with $\beta$ irreducible, we know $F$ is stable (ref. \cite[pp. 270]{PT2}), hence 
\begin{equation}\Ext^3_Y(F,F)\cong \Hom_Y(F,F\otimes K_Y)^{\vee}=0. \nonumber \end{equation}
Applying $\RHom_Y(-,F)$ to $I_Y\to \oO_Y \to F$, we obtain a distinguished triangle 
\begin{equation}\label{dist triangle Y}\RHom_Y(F,F)\to \RHom_Y(\oO_Y,F) \to \RHom_Y(I_Y,F), \end{equation}
whose cohomology gives an exact sequence 
\begin{equation}0=H^2(Y,F) \to \Ext^2_Y(I_Y,F)\to \Ext^3_Y(F,F) \to 0 \to \Ext^3_Y(I_Y,F) \to 0.  \nonumber \end{equation}
Hence $\Ext^i_Y(I_Y,F)=0$ for $i\geqslant 2$. Then we can apply the construction of \cite{BF, LT}.
\end{proof}
When $n=1$, similar to Proposition \ref{prop on pair obs on CY3}, we have
\begin{prop}\label{prop on pair obs on Fano3}
Let $\beta\in H_2(Y,\mathbb{Z})$ be an irreducible curve class on a Fano 3-fold $Y$. Then 
\begin{equation}f_*[P_1(Y,\beta)]_{\mathrm{pair}}^{\mathrm{vir}}=[M_{1,\beta}(Y)]^{\mathrm{vir}}, \nonumber \end{equation}
where $f:P_1(Y,\beta)\to M_{1,\beta}(Y)$, $(\oO_X\to F)\mapsto F$ is the morphism forgetting the section, $M_{1,\beta}(Y)$ is the moduli scheme of 1-dimensional stable sheaves $E$ on $Y$ with 
$[E]=\beta$ and $\chi(E)=1$.
\end{prop}
Now we come back to CY 4-fold $X=K_Y$. Similar to Proposition \ref{prop on pair moduli on CY3}, we have 
\begin{prop}\label{prop on pair moduli on K_Y}
Let $Y$ be a Fano 3-fold and $X=K_Y$. For an irreducible curve class $\beta\in H_2(X,\mathbb{Z})\cong H_2(Y,\mathbb{Z})$,
we have an isomorphism 
\begin{equation}P_n(X,\beta)\cong P_n(Y,\beta). \nonumber \end{equation}
The virtual class of $P_n(X,\beta)$ satisfies 
\begin{equation}[P_n(X,\beta)]^{\mathrm{vir}}=[P_n(Y,\beta)]_{\mathrm{pair}}^{\mathrm{vir}}, \nonumber \end{equation}
for certain choice of orientation in defining the LHS, 
where $[P_n(Y,\beta)]^{\mathrm{vir}}$ is the virtual class defined in Lemma \ref{lem on pair moduli on Y}.
\end{prop}
\begin{proof}
The proof is the same as the proof of Proposition \ref{prop on pair moduli on CY3}. 
Just note that as in (\ref{triangle CY3-fold 2}), there is a distinguished triangle 
\begin{align*} 
\RHom_X(I_X, i_{\ast}F) \to \RHom_X(I_X, I_X)_0[1] \to \RHom_X(i_{\ast}F, \oO_X)[2],
\end{align*}
where the cohomology of $\RHom_X(I_X, I_X)_0[1]$ is finite dimensional as $F$ has compact support (although $X$ is non-compact) and we may work with a compactification of $X$.
\end{proof}
${}$ \\
\textbf{Genus 0}. 
Combining Proposition \ref{prop on pair obs on Fano3} and Proposition \ref{prop on pair moduli on K_Y}, Conjecture \ref{conj:GW/GV g=0} for irreducible curve classes on $K_Y$ is equivalent to the
genus zero $\mathrm{GV}/\mathrm{DT_4}$ conjecture \cite[Conjecture 0.2]{CMT} on $K_Y$ (see also \cite[Conjecture 1.2]{Cao}), which has been verified in the following cases (ref. \cite[Prop. 2.1, 2.3, Thm. 2.7]{Cao}).
\begin{prop}\label{verify g=0 conj for fano hypersurface}
Conjecture \ref{conj:GW/GV g=0} is true for any irreducible curve class $\beta\in H_2(K_Y,\mathbb{Z})\cong H_2(Y,\mathbb{Z})$ provided that (i) $Y\subseteq \mathbb{P}^4$ is a smooth hypersurface of degree $d\leqslant4$, or (ii) $Y=S\times \mathbb{P}^1$ for a
toric del Pezzo surface $S$.  
\end{prop}

${}$ \\
\textbf{Genus 1}. 
When any curve $C$ in an irreducible class $\beta\in H_2(Y)$ is a smooth rational curve, $P_0(Y,\beta)=\emptyset$ by 
Lemma \ref{lem on emply pair moduli}, so $P_{0,\beta}(X)=0$ (by Proposition \ref{prop on pair moduli on K_Y}). 
In this case, to verify Conjecture \ref{conj:GW/GV g=1}, we are reduced to compute GW invariants and 
show $n_{1,\beta}=0$.
\begin{prop}\label{verify g=1 conj for fano hypersurface}
Let $Y=\mathbb{P}^3$ and $X=K_Y$. 
%Let $Y$ be either $\mathbb{P}^3$ or $\mathbb{P}^1\times S$, where $S$ is a del-Pezzo surface,
%$\mathbb{P}^1\times \mathbb{P}^2$ or $\mathbb{P}^1\times \mathbb{P}^1\times\mathbb{P}^1$  
Then Conjecture \ref{conj:GW/GV g=1} is true for any irreducible curve class $\beta\in H_2(X,\mathbb{Z})\cong H_2(Y,\mathbb{Z})$.
\end{prop}
\begin{proof}
When $Y=\mathbb{P}^3$, $n_{1,\beta}=0$ by \cite[Table 1, pp. 31]{KP}. 
%Here we consider only the case when $\mathbb{P}^1\times \mathbb{P}^2$. \textbf{TBC}
\end{proof}

\subsection{Irreducible curve classes on local surfaces}
Let $(S,\mathcal{O}_S(1))$ be a smooth projective surface and
\begin{align}\label{X:tot}
\pi \colon
X=\mathrm{Tot}_S(L_1 \oplus L_2) \to S
\end{align}
be the total space of direct sum of two line bundles $L_1$, $L_2$ on $S$.
Assuming that
\begin{align}\label{L12}
L_1 \otimes L_2 \cong K_S,
\end{align}
then $X$ is a non-compact CY 4-fold.
For a curve class
\begin{align*}
\beta \in H_2(X, \mathbb{Z})\cong H_2(S, \mathbb{Z}),
\end{align*}
we can consider the moduli space $P_n(X, \beta)$ of stable pairs on $X$, which is in general non-compact. 
In this section, we restrict to the case when the curve class $\beta$ is irreducible such that $L_i\cdot \beta<0$, 
in which case $P_n(X, \beta)$ is compact and smooth. 
%$L^{-1}_i$'s are ample, in which case   $P_n(X, \beta)$ is compact and smooth. 
\begin{lem}\label{stable pair on del-pezzo}
Let $S$ be a smooth projective surface and $\beta\in H_2(S,\mathbb{Z})$ be an irreducible curve class such that 
$K_S\cdot \beta<0$. 
Then the moduli space $P_n(S, \beta)$ of stable pairs on $S$ is smooth.
\end{lem}
\begin{proof}
Similar to the proof of Lemma \ref{lem on pair moduli on Y}, for any stable pair $I_S=(s:\oO_S\to F)\in P_n(S,\beta)$ with $\beta$ irreducible, 
$F$ is stable, hence 
\begin{equation}\Ext^2_S(F,F)\cong \Hom_S(F,F\otimes K_S)=0. \nonumber \end{equation}
Applying $\RHom_S(-,F)$ to $I_S\to \oO_S \to F$, we obtain a distinguished triangle 
\begin{equation}\label{dist triangle S}\RHom_S(F,F)\to \RHom_S(\oO_S,F) \to \RHom_S(I_S,F), \end{equation}
whose cohomology gives an exact sequence 
\begin{equation}0\to \Hom_S(F,F)\to H^0(F)\to \Hom_S(I_S,F)\to \Ext^1_S(F,F)\to   \nonumber \end{equation}
\begin{equation}\to H^1(F)\to \Ext^1_S(I_S,F) \to \Ext^2_S(F,F)=0,   \nonumber \end{equation}
and $\Ext^i_S(I_S,F)=0$ for $i\geqslant 2$. We claim the map $\Ext^1_S(F,F)\to H^1(F)$ above is surjective, then 
$\Ext^1_S(I_S,F)=0$ follows from the exact sequence (so the smoothness of moduli follows).

In fact, we only need to show the surjectivity of 
\begin{equation}H^1(\oO_C)\stackrel{\mathrm{id}}{\to} H^1(\mathcal{H}om(F,F))\subseteq \Ext^1_S(F,F) \to H^1(F), \nonumber \end{equation}
where $C$ is the scheme theoretical support of $F$. However, the above map is simply the multiplication by the section $s$, which fits into 
an exact sequence 
\begin{equation}H^1(\oO_C)\stackrel{s}{\to}  H^1(F) \to H^1(Q)=0, \nonumber \end{equation}
where $Q\cong F/s(\oO_S)$ is zero dimensional.  
\end{proof}
%If $L_i$ ($i=1,2$) are negative, we prove the compactness of $P_n(X,\beta)$.
\begin{prop}\label{prop on del-pezzo}
Let $S$ be a smooth projective surface and $L_1$, $L_2$ be two line bundles on $S$ such that $L_1\otimes L_2\cong K_S$.
%Let $S$ be a del-Pezzo surface and $L^{-1}_1$, $L^{-1}_2$ be two ample line bundles on $S$ such that $L_1\otimes L_2\cong K_S$. 
Then for any irreducible curve class $\beta\in H_2(X,\mathbb{Z})\cong H_2(S,\mathbb{Z})$ such that $L_i\cdot\beta<0$ ($i=1,2$), we have an isomorphism 
\begin{equation}P_n(X,\beta)\cong P_n(S,\beta).  \nonumber \end{equation}
And the virtual class satisfies 
\begin{equation}[P_n(X,\beta)]^{\mathrm{vir}}=[P_n(S,\beta)]\cdot e\Big(-\dR\hH om_{\pi_{P_{S}}}(\mathbb{F}, \mathbb{F} \boxtimes L_1)\Big),  
\nonumber \end{equation}
for certain choice of orientation in defining the LHS. Here $\mathbb{I}_S=(\oO_{S \times P_n(S, \beta)} \to \mathbb{F})\in D^b\big(S \times P_n(S, \beta)\big)$
is the universal stable pair and 
$\pi_{P_{S}} \colon S \times P_n(S, \beta) \to P_n(S, \beta)$ is the projection.
\end{prop}
\begin{proof}
Under assumption $L_i\cdot\beta<0$ and $\beta$ is irreducible, as in the proof of \cite[Prop. 3.1]{CMT}, 
one can show, 
for zero section $i:S\to X$, the morphism  
\begin{equation}\label{map compare pair on S/X}P_n(S,\beta)\to P_n(X,\beta),  \end{equation}
\begin{equation}I_S:=(s:i^*\oO_X\to F)\mapsto (s:\oO_X\to i_*F):=I_X  \nonumber \end{equation}
is bijective on closed points. And we have distinguished triangles
\begin{align}\label{triangle 2-fold 1}
i_{\ast}F \to I_X[1] \to \oO_X[1],
\end{align} 
\begin{align*} 
\RHom_X(I_X, i_{\ast}F) \to \RHom_X(I_X, I_X)_0[1] \to \RHom_X(i_{\ast}F, \oO_X)[2],
\end{align*}
\begin{align*}
\dL i^{\ast}I_X \cong I_S \oplus (F \otimes L_1^{-1}) \oplus(F \otimes L_2^{-1}) \oplus (F \otimes K_S^{-1})[1],
\end{align*}
where the last isomorphism is deduced similarly as (\ref{iso of i^*IX}).

It follows that we have a distinguished triangle
\begin{align}\label{triangle 2-fold 2}
\RHom_S(I_S, F)\oplus \RHom_S(F, F\otimes L_1) \to \RHom_X(I_X, I_X)_0[1] \to T,
\end{align}
where $T$ fits into the distinguished triangle
\begin{align}\label{triangle 2-fold 3}
\RHom_S(F, F\otimes L_2)\oplus \RHom_S(F, F\otimes K_S)[-1] \to T \to \RHom_X(i_{\ast}F, \oO_X)[2].
\end{align}
By Serre duality, degree shift and taking dual, (\ref{triangle 2-fold 3}) becomes 
\begin{align*} 
 \RHom_S(F, F\otimes L_1)[1]\oplus \RHom_S(F, F)[2] \to \RHom_S(\oO_S, F)[2] \to T^{\vee}.
\end{align*}
Combining with (\ref{dist triangle S}), we obtain a distinguished triangle
\begin{equation}\RHom_S(F, F\otimes L_1)[1] \to \RHom_S(I_S, F)[2] \to T^{\vee}, \nonumber \end{equation}
whose dual is 
\begin{equation}\label{triangle 2-fold 5}T \to \RHom_S(I_S, F)^{\vee}[-2] \to \RHom_S(F, F\otimes L_1)^{\vee}[-1].  \end{equation}
By taking cohomology of (\ref{triangle 2-fold 5}), we obtain exact sequences
\begin{equation} 0\to H^0(T) \to \Ext^2_S(I_S, F)^{\vee}\to \Ext^1_S(F, F\otimes L_1)^{\vee} \to H^1(T) \to  \nonumber \end{equation}
\begin{equation}\Ext^1_S(I_S, F)^{\vee}\to \Hom_S(F, F\otimes L_1)^{\vee}=0, \nonumber \end{equation}
where $\Ext^{i\geqslant1}_S(I_S, F)=0$ by the proof of Lemma \ref{stable pair on del-pezzo}. Hence 
\begin{equation} H^0(T)=0, \quad H^1(T)\cong\Ext^1_S(F, F\otimes L_1)^{\vee}.  \nonumber \end{equation}
By taking cohomology of (\ref{triangle 2-fold 2}), we obtain
\begin{equation}\Ext^0_S(I_S, F) \cong \Ext^1_X(I_X, I_X)_0,  
\nonumber \end{equation}
\begin{equation}0\to \Ext^1_S(F, F\otimes L_1) \to \Ext^2_X(I_X, I_X)_0 \to H^1(T) \to \Ext^2_S(I_S, F)\oplus \Ext^2_S(F, F\otimes L_1)=0, 
\nonumber \end{equation}
hence also the exact sequence
\begin{equation}0\to \Ext^1_S(F, F\otimes L_1) \to \Ext^2_X(I_X, I_X)_0\to \Ext^1_S(F, F\otimes L_1)^{\vee} \to 0. 
\nonumber \end{equation}
By the first isomorphism above, we know the map (\ref{map compare pair on S/X}) induces an isomorphism on tangent spaces. Moreover since $P_n(S,\beta)$ 
is smooth (Lemma \ref{stable pair on del-pezzo}) and (\ref{map compare pair on S/X}) is bijective on closed points, so the map (\ref{map compare pair on S/X}) is an isomorphism. 
 
As in Proposition \ref{prop on pair moduli on K_Y}, we can show $\Ext^1_S(F, F\otimes L_1)$ is a maximal isotropic subspace
of $\Ext^2_X(I_X, I_X)_0$ with respect to the Serre duality pairing on $\Ext^2_X(I_X, I_X)_0$.

Since $\Ext^0_S(F, F\otimes L_1)=\Ext^2_S(F, F\otimes L_1)=0$, $\Ext^1_S(F, F\otimes L_1)$ is constant over $P_n(S,\beta)$, 
so it forms a maximal isotropic subbundle of the obstruction bundle of $P_n(X,\beta)$ whose fiber over $I_X\in P_n(X,\beta)$ is $\Ext^2_X(I_X, I_X)_0$.
Then the virtual class has the desired property \cite{CL}.
\end{proof}
It is easy to check Conjecture \ref{conj:GW/GV g=0}, \ref{conj:GW/GV g=1} for irreducible curve classes on 
$\mathrm{Tot}_S(L_1\oplus L_2)$ in the following setting.
\begin{prop}\label{verify conj on del-pezzo surface}
Let $S$ be a del Pezzo surface and $L^{-1}_1$, $L^{-1}_2$ be two ample line bundles on $S$ such that $L_1\otimes L_2\cong K_S$. 
Denote $\beta\in H_2(X,\mathbb{Z})\cong H_2(S,\mathbb{Z})$ to be an irreducible curve class on $X=\mathrm{Tot}_S(L_1\oplus L_2)$.
Then Conjecture \ref{conj:GW/GV g=0}, \ref{conj:GW/GV g=1} are true for $\beta$.
\end{prop}
\begin{proof}
We claim that $S$ does not contain any $(-1)$ curve. In fact, if $C$ is a $(-1)$ curve, then 
\begin{align*}-2\geqslant\deg(L_1|_C)+\deg(L_2|_C)=\deg(K_S|_C)=-1. \end{align*}
So $S$ is either $\mathbb{P}^2$ or $\mathbb{P}^1\times \mathbb{P}^1$, and any curve in an irreducible class is a smooth rational 
curve. By Lemma \ref{lem on emply pair moduli}, $P_{0}(S,\beta)=\emptyset$, so $[P_0(X,\beta)]^{\mathrm{vir}}=0$ by Proposition \ref{prop on del-pezzo}. This matches with Klemm-Pandharipande's computation \cite[pp. 22, 24]{KP}, i.e. Conjecture \ref{conj:GW/GV g=1} is true for 
$\beta$.

As for the genus 0 conjecture, for any stable pair $(s:\oO_S\to F)\in P_{1}(S,\beta)$, 
$F$ is stable and supported on some $C\cong \mathbb{P}^1$ in $S$. Then $F=\oO_C$ and the morphism
\begin{equation}\phi: P_{1}(S,\beta)\stackrel{\cong}{\to}M_{1,\beta}(S), \quad (\oO_S\to F)\mapsto F, \nonumber \end{equation}
to the moduli space $M_{1,\beta}(S)$ of 1-dimensional stable sheaves $F$'s on $S$ with $[F]=\beta$ and $\chi(F)=1$ is an isomorphism.

As for the moduli space $\overline{M}_{0,0}(X,\beta)$ of stable maps, we have isomorphisms 
\begin{equation}\overline{M}_{0,0}(X,\beta)\cong \overline{M}_{0,0}(S,\beta)\cong M_{1,\beta}(S),  \nonumber \end{equation}
where the first isomorphism is by the negativity of $L_i$ $(i=1,2)$ and the second one is defined by mapping $f:\mathbb{P}^1\to S$ 
to $\oO_{f(\mathbb{P}^1)}$. 

Next, we compare obstruction theories. By Proposition \ref{prop on del-pezzo}, 
the `half' obstruction space of $P_1(X,\beta)$ at $(s:\oO_X\to \oO_C)$ is $\Ext^1_S(\oO_C,\oO_C\otimes L_1)$ 
which fits into the exact sequence 
\begin{equation}0\to H^1(C,L_1|_C)\to\Ext^1_S(\oO_C,\oO_C\otimes L_1)\to H^0(C,L_1|_C\otimes N_{C/S})\to 0. \nonumber \end{equation}

Since $S$ is either $\mathbb{P}^2$ or $\mathbb{P}^1\times \mathbb{P}^1$ and $\beta$ is irreducible, all stable maps are embedding.
The obstruction space of $\overline{M}_{0,0}(X,\beta)$ at $f:\mathbb{P}^1\to S$ is $H^1(C,N_{C/X})\cong H^1(\mathbb{P}^1,f^*TX)$ with $C=f(\mathbb{P}^1)$, which fits into 
the exact sequence 
\begin{equation}0=H^1(\mathbb{P}^1,f^*TS) \to H^1(\mathbb{P}^1,f^*TX)\to H^1(\mathbb{P}^1,f^*(L_1\oplus L_2)) \to 0. 
\nonumber \end{equation}
Note that 
\begin{equation}H^1(\mathbb{P}^1,f^*(L_1\oplus L_2))\cong H^1(C,L_1|_C)\oplus H^0(C,L_1|_C\otimes N_{C/S})^*. \nonumber \end{equation}
%and the GW obstruction space in this case is $H^1(C,N_{C/X})$ for $C=f(\mathbb{P}^1)$.
A family version of these computations shows the virtual classes satisfy 
\begin{equation}[P_1(X,\beta)]^{\mathrm{vir}}=[M_{0,0}(X,\beta)]^{\mathrm{vir}}, \nonumber \end{equation}
up to a sign (for each connected component of the moduli space). It is easy to match the insertions and then
verify Conjecture \ref{conj:GW/GV g=0}. More specifically, when $S=\mathbb{P}^2$, $P_{1,1}([\mathrm{pt}])=n_{0,1}([\mathrm{pt}])=-1$ and 
when $S=\mathbb{P}^1\times \mathbb{P}^1$, $P_{1,(1,0)}([\mathrm{pt}])=n_{0,(1,0)}([\mathrm{pt}])=P_{1,(0,1)}([\mathrm{pt}])=n_{0,(0,1)}([\mathrm{pt}])=1$
for certain choices of orientations.
\end{proof}

\subsection{Small degree curve classes on local surfaces}\label{small deg section}
We learned from discussions with Kool and Monavari \cite{KM} that by using relative Hilbert schemes and techniques developed 
in Kool-Thomas \cite{KT}, one can do explicit computations of pair invariants in small degrees for non-compact 
CY 4-folds 
\begin{align*} \mathrm{Tot}_{\mathbb{P}^2}(\oO(-1)\oplus \oO(-2)), \quad   \mathrm{Tot}_{\mathbb{P}^1\times \mathbb{P}^1}(\oO(-1,-1)\oplus \oO(-1,-1)).  \end{align*}
We list the results as follows (where pair invariants are defined with respect to certain of choices of orientation). \\
${}$ \\
Let $X=\mathrm{Tot}_{\mathbb{P}^2}(\oO(-1)\oplus \oO(-2))$, then 
\begin{enumerate}
\item $P_{0,1}=P_{0,2}=0$,\, $P_{0,3}=-1$,\, $P_{0,4}=2$, \\
\item $P_{1,1}([\mathrm{pt}])=-1$,\, $P_{1,2}([\mathrm{pt}])=1$,\, $P_{1,3}([\mathrm{pt}])=-1$,\, $P_{1,4}([\mathrm{pt}])=3$.  \\
\end{enumerate}
Let $X=\mathrm{Tot}_{\mathbb{P}^1\times \mathbb{P}^1}(\oO(-1,-1)\oplus \oO(-1,-1))$, then 
\begin{enumerate}
\item $P_{0,(2,2)}=1$,\, $P_{0,(2,3)}=2$,\, $P_{0,(2,4)}=5$,\, $P_{0,(3,3)}=10$,\\
\item $P_{1,(2,2)}([\mathrm{pt}])=2$,\, $P_{1,(2,3)}([\mathrm{pt}])=5$. 
\end{enumerate}
Comparing with \cite[pp. 22, 24]{KP}, we know our Conjecture \ref{conj:GW/GV g=0}, \ref{conj:GW/GV g=1} hold
in all above cases.

\section{Local curves}
Let $C$ be a smooth projective curve of genus $g(C)=g$, and
\begin{align}\label{X:tot2}
p \colon
X=\mathrm{Tot}_C(L_1 \oplus L_2 \oplus L_3) \to C
\end{align}
be the total space of a split rank three vector bundle on it.
Assuming that
\begin{align}\label{L123}
L_1 \otimes L_2 \otimes L_3 \cong \omega_C,
\end{align}
then the variety (\ref{X:tot2}) is a non-compact CY 4-fold.
Below we set $l_i \cneq \deg L_i$ and may assume that
$l_1\geqslant l_2\geqslant l_3$ without loss of generality.

Let $T=(\mathbb{C}^{\ast})^{ 3}$ be the three dimensional complex torus
%\begin{align*}
%T=(\mathbb{C}^{\ast})^{\times 3}.
%\end{align*}
which acts on the fibers of $X$. Its restriction to the subtorus
\begin{align*}
T_0=\{t_1 t_2 t_3=1\} \subset T
\end{align*}
preserves the CY 4-form on $X$ and also the Serre duality pairing on $P_{n}(X,\beta)$.
In this section, we aim to define equivariant virtual classes of $P_{n}(X,\beta)$
using a localization formula with respect to the $T_0$-action \cite{CL, CMT}, and investigate their relations with equivariant GW invariants.

Let $\bullet$ be the point $\Spec \mathbb{C}$ with trivial $T$-action, 
$\mathbb{C} \otimes t_i$ be the one dimensional $T$-representation with weight $1$,
and $\lambda_i \in H_T^{\ast}(\bullet)$ be its first Chern class.
They are generators of equivariant cohomology rings:
\begin{align}\label{H:lambda}
H_{T}^{\ast}(\bullet)=\mathbb{C}[\lambda_1, \lambda_2, \lambda_3], \quad \
H_{T_0}^{\ast}(\bullet)=\frac{\mathbb{C}[\lambda_1, \lambda_2, \lambda_3]}{(\lambda_1+\lambda_2+\lambda_3)} \cong \mathbb{C}[\lambda_1, \lambda_2].
\end{align}

\subsection{Localization for GW invariants}
Let $j \colon C \hookrightarrow X$
be the zero section of the projection (\ref{X:tot2}). We have
\begin{align*}
H_2(X, \mathbb{Z})=\mathbb{Z}[C],
\end{align*}
where $[C]$ is the fundamental class of $j(C)$. For $m\in \mathbb{Z}_{>0}$, we consider the diagram
\begin{align*}
\xymatrix{
\cC  \ar[r]^{f} \ar[d]^{\pi} & C \\
\overline{M}_h(C, m[C]), &}
\end{align*}
where $\cC$ is the universal curve and $f$ is the universal stable map.

The $T$-equivariant GW invariant of $X$ is defined by
\begin{align*}
\mathrm{GW}_{h, d[C]}(X)=
\mathrm{GW}_{h, d}(X) \cneq \int_{[\overline{M}_h(C, d[C])]^{\rm{vir}}}
e(-\dR h_{\ast}f^{\ast}N) \in \mathbb{Q}(\lambda_1, \lambda_2, \lambda_3),
\end{align*}
where
$N$ is the
$T$-equivariant normal bundle of $j(C) \subset X$:
\begin{equation}\label{nor bdl N}
 N=(L_1 \otimes t_1) \oplus (L_2 \otimes t_2) \oplus (L_3 \otimes t_3). 
\end{equation}

If $g(C)>0$, the vanishing of genus zero GW invariants
\begin{align*}
\mathrm{GW}_{0, d}(X)=0, \ g(C)>0, \ d \in \mathbb{Z}_{>0}
\end{align*}
follows from $\overline{M}_0(C, d[C])=\emptyset$.

If $g(C)=0$, we have
\begin{align*}
\mathrm{GW}_{0, d}(X)
=\int_{[\overline{M}_0(\mathbb{P}^1, d)]}
e\Big(-\dR h_{\ast}f^{\ast}\left(\oO_{\mathbb{P}^1}(l_1)t_1 \oplus
\oO_{\mathbb{P}^1}(l_2)t_2 \oplus
\oO_{\mathbb{P}^1}(l_3)t_3
\right)\Big).
\end{align*}
For example in the $d=1$ case,
$\overline{M}_0(\mathbb{P}^1, 1)$ is one point and
\begin{align}\label{GW01}
\mathrm{GW}_{0, 1}(X)
=\lambda_1^{-l_1-1} \lambda_2^{-l_2-1} \lambda_3^{-l_3-1}.
\end{align}
In the $d=2$ case, a straightforward localization calculation
with respect to the $(\mathbb{C}^{\ast})^2$-action on $\mathbb{P}^1$ gives
\begin{align}\label{GW02}
\mathrm{GW}_{0, 2}(X)=&\frac{1}{8}
\lambda_1^{-2l_1-1}\lambda_2^{-2l_2-1}\lambda_3^{-2l_3-1}
\left\{(\overline{l}_1^2-(\overline{l}_1-1)^2+\cdots)
\lambda_1^{-2} + \right. \\
\notag &\left.
(\overline{l}_2^2-(\overline{l}_2-1)^2+\cdots) \lambda_2^{-2}
+(\overline{l}_3^2-(\overline{l}_3-1)^2+\cdots) \lambda_3^{-2}+ \right. \\
 \notag & \left.
l_1 l_2 \lambda_1^{-1} \lambda_2^{-1}
+l_2 l_3 \lambda_2^{-1} \lambda_3^{-1}
+l_1 l_3 \lambda_1^{-1} \lambda_3^{-1} \right\}.
\end{align}
Here we write $\overline{l}=l$ for $l\geqslant 0$
and $\overline{l}=-l-1$ for $l<0$.

\subsection{Localization for stable pairs}\label{localiza for stable pair}
%Let $[C] \in H_2(X, \mathbb{Z})$ be the fundamental class of the zero section of $X$.
Similarly, for $m \in \mathbb{Z}_{\geqslant 0}$, 
we want to define (equivariant) stable pair invariant
\begin{equation}\label{localization for local curve}
P_{n,m[C]}(X)=[P_{n}(X,m[C])^{T_{0}}]^{\rm{vir}}\cdot e( \dR \hH om_{\pi_P}(\mathbb{I}, \mathbb{I})_0^{\rm{mov}})^{1/2},   \end{equation}  
where $\mathbb{I}=(\oO_{X\times P_n(X, m[C])}\to \mathbb{F}) \in D^{b}(X\times P_n(X, m[C]))$ is the universal stable pair and $\pi_P:X\times P_n(X, m[C])\to P_n(X, m[C])$ is the projection. 
Of course, the above equality is not a definition as the virtual class of the fixed locus as well as the square root needs justification. 
We will make this precise in specific cases where we compare with GW invariants of $X$.  \\

Let us first describe stable pairs $(s \colon \oO_X \to F)\in P_{n}(X,m[C])^{T}$ which are fixed by the full-torus $T$: 
decompose $F$ into $T$-weight space
\begin{align*}
p_*F=\bigoplus_{(i_1, i_2, i_3) \in \mathbb{Z}^3} F^{i_1, i_2, i_3},
\end{align*}
where the $T$-weight of $F^{i_1, i_2, i_3}$ is $(i_1, i_2, i_3)$. We denote an index set
\begin{align}\label{index set}
\Delta:=\{(i_1, i_2, i_3) \in \mathbb{Z}_{\geqslant 0}^3 :
F^{-i_1, -i_2, -i_3} \neq 0\}.
\end{align}
We also have the decomposition
\begin{align*}
p_{\ast} \oO_X=\bigoplus_{(i_1, i_2, i_3)\in \mathbb{Z}_{\geqslant 0}^3}
L_1^{-i_1} \otimes L_2^{-i_2} \otimes L_3^{-i_3}
\end{align*}
into direct sum of weight $(-i_1, -i_2, -i_3)$ factor $L_1^{-i_1} \otimes L_2^{-i_2} \otimes L_3^{-i_3}$.

The $T$-equivariance of $s$
induces morphisms
\begin{align*}
s^{i_1, i_2, i_3} \colon
L_1^{-i_1} \otimes L_2^{-i_2} \otimes L_3^{-i_3} \to F^{-i_1, -i_2, -i_3}
\end{align*}
in $\Coh(C)$ which are surjective in dimension one.
It follows that each $F^{-i_1, -i_2, -i_3}$ is
either zero or can be written as
\begin{align*}
F^{-i_1, -i_2, -i_3}=L_1^{-i_1} \otimes L_2^{-i_2} \otimes L_3^{-i_3}
\otimes \oO_C(Z_{i_1, i_2, i_3})
\end{align*}
for some effective divisor $Z_{i_1, i_2, i_3} \subset C$.
Moreover, the $p_{\ast}\oO_X$-module structure
on $F$ gives a morphism
\begin{align*}
F^{-i_1, -i_2, -i_3} \otimes L_1^{-1} \to F^{-i_1-1, -i_2, -i_3}
\end{align*}
which commutes with $s^{i_1, i_2, i_3}$ and $s^{i_1+1, i_2, i_3}$.
Similar morphisms replacing $L_1$ by $L_2, L_3$ exist and have similar commuting property. Hence, for $(i_1, i_2, i_3) \in \Delta$, we have
\begin{align*}
Z_{i_1-1, i_2, i_3},\,Z_{i_1, i_2-1, i_3},\,
Z_{i_1, i_2, i_3-1}\, \leqslant Z_{i_1, i_2, i_3},
\end{align*}
as divisors in $C$. So the set $\Delta$ (\ref{index set})
%\begin{align*}
%\Delta=\{(i_1, i_2, i_3) \in \mathbb{Z}_{\geqslant 0}^3 : F^{-i_1, -i_2, -i_3} \neq 0\}\end{align*}
is a three dimensional Young diagram,
which is finite by the coherence of $F$.   \\

In general, it is difficult to explicitly determine $T_0$-fixed stable pairs. 
In fact, a $T_0$-fixed stable pair is not necessarily $T$-fixed.
%At least when the curve class is zero, a stable pair corresponds to 
%an ideal sheaf of a zero dimensional subscheme, in which case we have 
%\begin{exam}
%Let $X=\mathbb{C}\times \mathbb{C}^3$ and $T\cong(\mathbb{C}^*)^3$ act on $X$ by 
%\begin{equation}t\cdot(x_0,x_1,x_2,x_3)=(x_0,t_1x_1,t_2x_2,t_3x_3). \nonumber \end{equation}
%Denote $T_0=\{t\in T\,|\,t_1t_2t_3=1\}$ to be a subtorus. Then the ideal
%\begin{equation}I_Z=\langle x_1^2,x_2^2,x_3^2,\,x_0+x_1x_2x_3 \rangle \nonumber \end{equation} 
%defines a $T_0$-fixed zero dimensional subscheme which is not $T$-fixed.
%\end{exam}
Nevertheless, for a $T_0$-fixed stable pair $(s \colon \oO_X \to F)$, $\oO_{C_F}:=\Imm s$ and the corresponding ideal sheaf $I_{C_F}$ are actually $T$-fixed.
\begin{lem}\label{lem:Tfix}
Let $I=(s \colon \oO_X \to F)\in P_{n}(X,m[C])^{T_{0}}$ be a $T_0$-fixed stable pair and $\oO_{C_F}:=\Imm s \subseteq F$. Then
the ideal sheaf $I_{C_F}\subseteq \mathcal{O}_X$ is $T$-fixed.
\end{lem}
\begin{proof}
Since $I_{C_F}=\hH^0(I)$, it is $T_0$-fixed.
For $t \in T$, we have the diagram
\begin{align*}
\xymatrix{
0 \ar[r] & I_{C_F}  \ar[r] &\oO_X \ar[d]^{\cong} \ar[r] & \oO_{C_F} \ar[r] & 0 \\
0 \ar[r] & t^{\ast}I_{C_F}  \ar[r] &t^{\ast}\oO_X \ar[r] &
t^{\ast}\oO_{C_F} \ar[r] & 0.
}
\end{align*}
The above diagram induces the morphism
$u \in \Hom(I_{C_F}, t^{\ast}\oO_{C_F})$. It is enough to show $u=0$.
For a general point $c \in C$, let $X_c=p^{-1}(c)=\mathbb{C}^3$
be the fiber of $p$ at $c$.
Then $I_{C_F}|_{X_c}$ is an
ideal sheaf of $T_0$-fixed zero dimensional
subscheme of $\mathbb{C}^3$. % i: X_c\to X is flat by base-change, so Rq_*i^*\oO_Z=i^*Rp_*\oO_Z which is finite dim. 
% and \oO_{X_c}=i^*\oO_X\to i^*\oO_Z is surjective so I^*\oO_Z is a the structure sheaf of a zero dim subscheme. Thus 
% i^*I_Z is ideal sheaf of a zero dim subscheme.
Then it is also $T$-fixed by~\cite[Lemma~4.1]{BBr}.
This implies that the morphism restricted to $X_c$ is a zero map.
Then $\Imm u \subset t^{\ast}\oO_{C_F}$ is
zero on the general fiber of $p$, hence
$\Imm u=0$ by the purity of $C_F$.
\end{proof}
Another convenient way to determine $T_0$-fixed stable pairs is in the case when 
$P_{n}(X,m[C])^{T}$ is smooth and $\Hom_X(I,F)^{T_0}=\Hom_X(I,F)^{T}$ for any $I=(\oO_X\to F)\in P_{n}(X,m[C])^{T}$ 
(see e.g. \cite[Sect. 3.3]{PT3} on toric 3-folds). 
Then one has 
$P_{n}(X,m[C])^{T}=P_{n}(X,m[C])^{T_0}$.
In the examples below, we will explicitly determine the $T_0$-fixed locus mainly using Lemma \ref{lem:Tfix}.

\subsection{$P_{1, m[C]}(X)$ and genus zero conjecture}
Let $C=\mathbb{P}^{1}$ be a smooth rational curve and 
$X=\mathcal{O}_{\mathbb{P}^{1}}(l_1,l_2,l_3)$ with $l_1+l_2+l_3=-2$.
This serves as the local model for a neighbourhood of a rational curve in a CY 4-fold. 

For some special choice of $(l_1,l_2,l_3)$, we can determine $P_{1,m[C]}(X)$ for all $m$.
\begin{prop}\label{general local curves}
If $X=\mathcal{O}_{\mathbb{P}^{1}}(-1,-1,0)$,  
then $P_{1,m[\mathbb{P}^{1}]}(X)$ is well-defined and satisfies 
\begin{equation}P_{1,[\mathbb{P}^{1}]}(X)=\pm\,\lambda_3^{-1},\quad P_{1,m[\mathbb{P}^{1}]}(X)=0, \,\,\textrm{when}\,\,m>1.   \nonumber \end{equation}
If $X=\mathcal{O}_{\mathbb{P}^{1}}(-2,0,0)$, 
then $P_{1,m[\mathbb{P}^{1}]}(X)$ is well-defined and satisfies 
\begin{equation}P_{1,[\mathbb{P}^{1}]}(X)=\pm\,\frac{\lambda_1}{\lambda_2\lambda_3} ,\quad P_{1,m[\mathbb{P}^{1}]}(X)=0, \,\,\textrm{when}\,\,m>1.   \nonumber \end{equation}
\end{prop}
\begin{proof}
Let $(s:\oO_X\to F)$ be a $T_0$-fixed stable pair and $\oO_{C_F}=\Imm(s)$. Then   
\begin{equation}1=\chi(F)=\chi(\oO_{C_F})+\chi(F/\oO_{C_F}).  \nonumber \end{equation}
So $\chi(\oO_{C_F})=1$ or $0$. By Lemma \ref{lem:Tfix}, $(s:\oO_X\twoheadrightarrow \oO_{C_F})$ is $T$-fixed. 
From the characterization of $T$-fixed stable pairs, it is of the form
\begin{equation}\oO_X \twoheadrightarrow \bigoplus_{(i_1, i_2, i_3) \in \Delta}
\mathcal{O}_{\mathbb{P}^{1}}(l_1)^{-i_1} \otimes \mathcal{O}_{\mathbb{P}^{1}}(l_2)^{-i_2} \otimes \mathcal{O}_{\mathbb{P}^{1}}(l_3)^{-i_3}
\otimes \mathcal{O}_{\mathbb{P}^{1}}(Z_{i_1,i_2,i_3}).
\nonumber \end{equation}
If $(l_1,l_2,l_3)=(-1,-1,0)$ or $(-2,0,0)$, it is obvious that the only possibility is $\chi(\oO_{C_F})=1$ (so $F\cong \oO_{C_F}$) and $C_F$ is the zero section of $X$. So $P_1(X,m[\mathbb{P}^{1}])=\emptyset$ unless $m=1$. 

By (\ref{localization for local curve}), we have 
\begin{align*}P_{1,[\mathbb{P}^{1}]}(X)&=\pm\,\frac{\sqrt{(-1)^{\frac{1}{2}\ext^2_X(I_{\mathbb{P}^1},I_{\mathbb{P}^1})_0}\cdot e_{T_0}\big(\Ext^2_X(I_{\mathbb{P}^1},I_{\mathbb{P}^1})_0\big)}}{e_{T_0}\big(\Ext^1_X(I_{\mathbb{P}^1},I_{\mathbb{P}^1})_0\big)} \\
&=\pm\,\frac{e_{T_0}\big(H^1(\mathbb{P}^1,L_1 \otimes t_1 \oplus L_2 \otimes t_2 \oplus L_3 \otimes t_3)\big)}
{e_{T_0}\big(H^0(\mathbb{P}^1,L_1 \otimes t_1 \oplus L_2 \otimes t_2 \oplus L_3 \otimes t_3)\big)}. \end{align*}
Then the calculation is straightforward.
\end{proof}
By comparing the above computations with the corresponding GW invariants, we obtain the following equivariant analogue of 
Conjecture \ref{conj:GW/GV g=0} (note from the above proof, we know $P_{0,m[C]}(X)=0$ ($m\geqslant1$) since $P_{0}(X,m[C])=\emptyset$).
\begin{cor}
Let $X=\mathcal{O}_{\mathbb{P}^{1}}(-1,-1,0)$ or $\mathcal{O}_{\mathbb{P}^{1}}(-2,0,0)$. Then  
\begin{equation}\mathrm{GW}_{0,m}(X)=\sum_{k|m,\,k\geqslant1}\frac{1}{k^3}P_{1,(m/k)[\mathbb{P}^{1}]}(X), \nonumber \end{equation}
for suitable choices of orientation in defining the RHS.
\end{cor}
\begin{proof}
If $X=\mathcal{O}_{\mathbb{P}^{1}}(-1,-1,0)$, by Aspinwall-Morrison formula, we have 
\begin{equation}\mathrm{GW}_{0,m}(X)=\frac{1}{m^3}\,\lambda_3^{-1}. \nonumber \end{equation}
If $X=\mathcal{O}_{\mathbb{P}^{1}}(-2,0,0)$, from GW invariants of $K_{\mathbb{P}^1}$ (e.g. \cite[Thm 1.1]{Maulik}), 
we can conclude
\begin{equation*}\mathrm{GW}_{0,m}(X)=\frac{1}{m^3}\cdot\frac{\lambda_1}{\lambda_2\lambda_3}.  
\nonumber  \end{equation*}  
Comparing with Proposition \ref{general local curves}, we are done.
\end{proof}

For general local curve $X=\mathcal{O}_{\mathbb{P}^{1}}(l_1,l_2,l_3)$, we study $P_1(X,m[\mathbb{P}^{1}])$ for $m=1,2$ as follows.

${}$ \\
\textbf{Degree one class}. When $m=1$, it is easy to show the canonical section 
\begin{equation}(s:\oO_X\twoheadrightarrow \oO_{\mathbb{P}^{1}}) \nonumber \end{equation} 
gives the only $T_0$-fixed stable pair in $P_{1}(X,[\mathbb{P}^{1}])$. Similar to Proposition \ref{general local curves}, we have 
\begin{align*}P_{1,[\mathbb{P}^{1}]}(X)&=\frac{e_{T_0}\big(H^1(\mathbb{P}^1,L_1 \otimes t_1 \oplus L_2 \otimes t_2 \oplus L_3 \otimes t_3)\big)}{e_{T_0}\big(H^0(\mathbb{P}^1,L_1 \otimes t_1 \oplus L_2 \otimes t_2 \oplus L_3 \otimes t_3)\big)} \\
&= \lambda^{-l_1-1}_1\lambda^{-l_2-1}_2\lambda^{-l_3-1}_3, \end{align*}
which coincides with corresponding GW invariant (\ref{GW01}). 
Here we have chosen the plus sign in defining $P_{1,[\mathbb{P}^{1}]}(X)$.

${}$ \\
\textbf{Degree two class}. When $m=2$, let $(s:\oO_X\to F)\in P_1(X,2[\mathbb{P}^1])$ be a $T_0$-fixed stable pair. 
Then $F$ is thickened into one of the $L_i$-direction, i.e. 
\begin{align*}
p_{\ast}F=F_0 \oplus (F_i \otimes t_i^{-1}),
\end{align*}
where $F_0$, $F_i$ are line bundles on $\mathbb{P}^1$, hence $F$ is also $T$-fixed. As the $T$-weight of $F_i$ is not of form $(l,l,l)$, so $T_0$-invariant sections
of $F$ are also $T$-invariant.
%Similarly to the description of $T$-fixed stable pairs, 
So we have a
commutative diagram 
\begin{align*}
\xymatrix{
\oO_{\mathbb{P}^1}\otimes L^{-1}_i  \ar[r]^{s^0} \ar[d]^{=} & F_0\otimes L_i^{-1}  \ar[d]^{\phi}\\
L^{-1}_i \ar[r]^{s^i} &  F_i\otimes t_i^{-1}, }
\end{align*}
where $s^{0}$ and $s^{i}$ are injective, and surjective in dimension one, $\phi$ defines the $p_*\oO_X$-module structure (which is also 
injective, and surjective in dimension one by the diagram).

Denote $F_0=\oO_{\mathbb{P}^1}(d_0)$ and $F'_i=F_i\otimes L_i=\oO_{\mathbb{P}^1}(d_i)$, the above diagram is equivalent to
a commutative diagram 
\begin{align*}
\xymatrix{
\oO_{\mathbb{P}^1}   \ar[r]^{s^0\,\,\,\,} \ar[d]^{=} & \oO_{\mathbb{P}^1}(d_0)  \ar[d]^{\phi}\\
\oO_{\mathbb{P}^1}  \ar[r]^{s^i\,\,\,\,} &  \oO_{\mathbb{P}^1}(d_i),}
\end{align*}
where $s^{0}$, $s^i$ and $\phi$ are injective. These impose conditions
\begin{equation}0\leqslant d_0\leqslant d_i, \quad d_0+d_i=l_i-1, \nonumber \end{equation}
where the last equality is because $\chi(F)=1$. It is not hard to show the following 
\begin{lem} 
We have the following isomorphism
\begin{align}\label{identify T0-fixed locus in degree two}
P_1(X,2[\mathbb{P}^1])^{T_0}\stackrel{\cong}{\to}\, \coprod_{i=1}^3
\coprod_{\begin{subarray}{c}
(d_0, d_i) \in \mathbb{Z}^2 \\
d_0+d_i=l_i-1 \\
0\leqslant d_0 \leqslant (l_i-1)/2
\end{subarray}}
\Pic^{(d_0, d_i)}(\mathbb{P}^1)\times \mathbb{P}(H^0(\oO_{\mathbb{P}^1}(d_0))),
\end{align}
where $\Pic^{(a, b)}(\mathbb{P}^1)$ denotes the moduli space of triples
\begin{align*}
(L, L', \iota), \ 
(L, L') \in \Pic^{a}(\mathbb{P}^1) \times \Pic^{b}(\mathbb{P}^1), \ 
\iota \colon L \hookrightarrow L',
\end{align*}
and $\iota$ is an inclusion of sheaves.
\end{lem}
To determine the virtual class $[P_1(X,2[\mathbb{P}^1])^{T_0}]^{\mathrm{vir}}$ and the square root in (\ref{localization for local curve}), we take a $T_0$-fixed stable pair $I=(s:\oO_X\to F)$ and view it as an element in the $T_0$-equivariant $K$-theory of $X$. Then
\begin{align*}
\chi(I,I)_0=\chi(F,F)-\chi(\oO_X,F)-\chi(F,\oO_X)\in K_{T_0}(\mathrm{pt}),
\end{align*}
where both sides of the equality can be written using grading into $T_0$-weight space.

Similar to \cite[Sect. 4.4]{CMT}, we set 
\begin{align}\label{choose square root in deg two}
\chi(F,F)^{1/2}:=\chi(j_*F_0,j_*F_0)+\chi(j_*F_0,j_*F_i)t^{-1}_i,\,\,\,\, \chi(I,I)^{1/2}_0:=\chi(F,F)^{1/2}-\chi(\oO_X,F) \end{align}
where $F=F_0+F_i\otimes t_i^{-1}\in K_{T_0}(X)$ and $j$ is the inclusion of zero section of $X$.

The $T_0$-fixed and movable part satisfies  
\begin{align*}
\chi(I,I)^{1/2,\mathrm{fix}}_0&=\chi(F,F)^{1/2,\mathrm{fix}}-\chi\big(\oO_{\mathbb{P}^1},\oO_{\mathbb{P}^1}(d_0)\big), \\
\chi(I,I)^{1/2,\mathrm{mov}}_0&=\chi(F,F)^{1/2,\mathrm{mov}}-\chi\big(\oO_{\mathbb{P}^1},\oO_{\mathbb{P}^1}(d_i-l_i)\big)\cdot t^{-1}_i,
\end{align*}
where $\chi(F,F)^{1/2,\mathrm{fix}}$ and $\chi(F,F)^{1/2,\mathrm{mov}}$ was computed in \cite[Sect. 4.4]{CMT}.
In particular, 
\begin{equation}\dim_{\mathbb{C}}\chi(F,F)^{1/2,\mathrm{fix}}=1-d_i+d_0, \quad \dim_{\mathbb{C}}\chi\big(\oO_{\mathbb{P}^1},\oO_{\mathbb{P}^1}(d_0)\big)=d_0+1. \nonumber \end{equation} 
So $\dim_{\mathbb{C}}(-\chi(I,I)^{1/2,\mathrm{fix}}_0)=d_i$ is the dimension of $P_1(X,2[\mathbb{P}^1])^{T_0}$. Thus the virtual class of the associated $T_0$-fixed locus $P_1(X,2[\mathbb{P}^1])^{T_0}$ may be defined to be its usual fundamental class. \\

We can now give a definition of $P_{1,2[\mathbb{P}^1]}(X)\in \mathbb{Q}(\lambda_1, \lambda_2)$ based on the localization formula (\ref{localization for local curve}) and the above discussion.
Denote 
\begin{align*}
(\fF_0, \fF_i', \iota), \
\iota \colon \fF_0 \hookrightarrow \fF_i'
\end{align*}
to be the universal object on $\Pic^{(d_0, d_i)}(\mathbb{P}^1) \times \mathbb{P}^1$, where $\fF_0, \fF_i'$ are
line bundles on $\Pic^{(d_0, d_i)}(\mathbb{P}^1) \times \mathbb{P}^1$
and $\iota$ is the universal injection.
Let $\fF_i \cneq \fF_i' \boxtimes L_i^{-1}$, and consider its push-forward
\begin{align*}
j_{\ast} \fF_i \in \Coh(\Pic^{(d_0, d_i)}(\mathbb{P}^1) \times X), \,\ i=1, 2, 3.
\end{align*}
From (\ref{localization for local curve}) and (\ref{choose square root in deg two}), we define $P_{1,2[\mathbb{P}^1]}(X)$ as an element in $\mathbb{Q}(\lambda_1, \lambda_2)$ by
\begin{equation}\label{def of deg two g=0 inv}
P_{1,2[\mathbb{P}^1]}(X):=\sum_{i=1}^3\sum_{\begin{subarray}{c}(d_0, d_i) \in \mathbb{Z}^2 \\d_0+d_i=l_i-1 \\
0\leqslant d_0 \leqslant (l_i-1)/2 \end{subarray}}
\Bigg(\int_{\Pic^{(d_0, d_i)}(\mathbb{P}^1)}e_{T_0}(\mathcal{N}_1) \cdot \int_{\mathbb{P}(H^0(\oO_{\mathbb{P}^1}(d_0)))}e_{T_0}(\mathcal{N}_2)\Bigg),    \end{equation}
where 
\begin{equation}\mathcal{N}_1:=\dR \hH om_{\pi_1}(j_{\ast}\fF_0, j_{\ast}\fF_0)^{\rm{mov}}+ 
\dR \hH om_{\pi_1}(j_{\ast}\fF_0, j_{\ast}\fF_i \cdot t_i^{-1})^{\rm{mov}}, \nonumber \end{equation}
\begin{equation}\mathcal{N}_2:=-\dR(\pi_2)_*(\oO_{\mathbb{P}^{d_0}\times\mathbb{P}^1}(1,d_i-l_i)\cdot t_i^{-1}). \nonumber \end{equation}
Here $\pi_1:\Pic^{(d_0, d_i)}(\mathbb{P}^1)\times X\to\Pic^{(d_0, d_i)}(\mathbb{P}^1)$ and 
$\pi_2:\mathbb{P}^{d_0}\times\mathbb{P}^1 \to\mathbb{P}^{d_0}$ are natural projections.
and we have used the isomorphism (\ref{identify T0-fixed locus in degree two}). 

In the above definition, the second integration can be easily shown to be $1$ and the first one has been explicitly 
determined before \cite[Corollary 4.9]{CMT}. So we obtain
\begin{prop}\label{compute deg two pair inv}
Let $X=\mathcal{O}_{\mathbb{P}^{1}}(l_1,l_2,l_3)$ with $l_1+l_2+l_3=-2$ and $l_1\geqslant l_2\geqslant l_3$. Then  
\begin{align*}
P_{1,2[\mathbb{P}^1]}(X)=
&-\lambda_1^{-2l_1-2}\lambda_2^{-2l_2-2}(\lambda_1+\lambda_2)^{-2l_3-2} \\
& \cdot \left(
\sum_{\begin{subarray}{c}
1\leqslant k\leqslant l_1, \\
k \equiv l_1 \ (\mathrm{mod} 2)
\end{subarray}}
A(l_1, l_2, l_3, k)
+
\sum_{\begin{subarray}{c}
1\leqslant k\leqslant l_2, \\
k \equiv l_2 \ (\mathrm{mod} 2)
\end{subarray}}
B(l_1, l_2, l_3, k) \right),
\end{align*}
where
\begin{align*}
&A(l_1, l_2, l_3, k) :=
\mathrm{Res}_{h=0}
\left\{h^{-k}(-\lambda_1+h)^2 (\lambda_2+h)^{k+l_2} \right. \\
&\left.(-\lambda_1-\lambda_2+h)^{k+l_3}
(-\lambda_1+\lambda_2+h)^{l_1-l_2-k}
(-2\lambda_1-\lambda_2+h)^{l_1-l_3-k}
(-2\lambda_1+h)^{k-2-2l_1} \right\},
\end{align*}
\begin{align*}
&B(l_1, l_2, l_3, k):=\mathrm{Res}_{h=0}
\left\{h^{-k}(-\lambda_2+h)^2 (\lambda_1+h)^{k+l_1} \right. \\
&\left. (-\lambda_2-\lambda_1+h)^{k+l_3}
(-\lambda_2+\lambda_1+h)^{l_2-l_1-k}
(-2\lambda_2-\lambda_1+h)^{l_2-l_3-k}
(-2\lambda_2+h)^{k-2-2l_2} \right\}.
\end{align*}
\end{prop}
We pose the following equivariant version of Conjecture \ref{conj:GW/GV g=0} (note in this case 
$P_{0,[\mathbb{P}^{1}]}(X)=0$ as $\chi(\oO_{\mathbb{P}^1})>0$). 
It is consistent with our previous conjecture on one dimensional stable sheaves \cite[Conj. 4.10]{CMT}.
\begin{conj}\label{equi genus zero conj in deg two}
Let $X=\mathcal{O}_{\mathbb{P}^{1}}(l_1,l_2,l_3)$ for $l_1+l_2+l_3=-2$. Then 
\begin{equation}\mathrm{GW}_{0,2}(X)=P_{1,2[\mathbb{P}^{1}]}(X)+\frac{1}{8}P_{1,[\mathbb{P}^{1}]}(X). \nonumber \end{equation}
\end{conj}
Combining Proposition \ref{compute deg two pair inv} and \cite[Thm. 4.12]{CMT}, we can verify the conjecture in a large number of examples.
\begin{thm}\label{g=0 thm on local curve}
Conjecture \ref{equi genus zero conj in deg two} is true if $|l_1|\leqslant10$ and $|l_2|\leqslant10$. 
\end{thm}

\subsection{$P_{0, m[C]}(X)$ and genus one conjecture}
To complete the heuristic argument for our genus one conjecture in Section \ref{heuristic argument}, 
we consider $X=\mathrm{Tot}_{C}(L_{1}\oplus L_{2}\oplus L_{3})$ where $C$ is an elliptic curve and $L_1 \otimes L_2 \otimes L_3 \cong \omega_C\cong \oO_C$.
\begin{lem}\label{lem:isom:ext}
Let $I \subset \oO_X$ be the ideal sheaf of a closed
subscheme $Z \subset X$ with $\dim Z \leqslant1$. Then we have canonical isomorphisms
\begin{align}\label{isom:ext}
\Ext_X^1(I, I)_0 &\cong H^0(X, \eE xt_X^1(I, I)), \\
\notag
\Ext_X^2(I, I)_0 &\cong H^0(X, \eE xt_X^2(I, I)) \oplus
H^1(X, \eE xt_X^1(I, I)).
\end{align}
Furthermore, if $p_{\ast} \eE xt^1_X(I, I)$ and $p_{\ast} \eE xt^2_X(I, I)$ are locally free, then
\begin{align*}H^1(X, \eE xt_X^1(I, I)) \cong H^0(X, \eE xt_X^2(I, I))^{\vee}. \end{align*}
And $H^0(X, \eE xt_X^2(I, I))$, $H^1(X, \eE xt_X^1(I, I))$ are maximal isotropic subspaces of $\Ext_X^2(I, I)_0$ with respect to Serre duality pairing.
\end{lem}
\begin{proof}
We have the local to global spectral sequence
\begin{align*}
E_2^{p, q}=H^p(X, \eE xt_X^q(I, I)_0) \Rightarrow
\Ext_X^{p+q}(I, I)_0.
\end{align*}
And
\begin{align*}
\eE xt_X^0(I, I)_0=0, \quad
\eE xt_X^{\geqslant 1}(I, I)_0\cong\eE xt_X^{\geqslant 1}(I, I)
\end{align*}
are supported on $Z$. Therefore we have
$E_2^{p, 0}=0$ and
$E_2^{p, q}=0$ for $p\geqslant2$, $q\geqslant 1$.
Then the above spectral sequence degenerates and (\ref{isom:ext}) holds.
The latter statement follows from the adjunction 
$$\Ext^i_X(p^*\oO_C, \eE xt_X^j(I, I))=\Ext^i_C(\oO_C, p_*\eE xt_X^j(I, I)), $$ 
and the Grothendieck duality
\begin{align*}
\dR p_{\ast} \dR \hH om_X(I, I)_0[4] \cong
\dR \hH om_C(\dR p_{\ast} \dR \hH om_X(I, I)_0, \omega_C[1])
\end{align*}
for the projection $p: X\to C$ (\ref{X:tot2}).
\end{proof}
We describe the torus fixed locus $P_0(X, m[C])^{T_0}$ as follows.
\begin{lem}\label{lem:ell:fixed}
Let $C$ be an elliptic curve and $L_i \in \Pic^0(C)$.
Then $(\oO_X \to F) \in P_0(X, m[C])$
is $T_0$-fixed if and only if it is of the form
\begin{align}\label{pair:ell}
\oO_X \twoheadrightarrow \bigoplus_{(i_1, i_2, i_3) \in \Delta}
L_1^{-i_1} \otimes L_2^{-i_2} \otimes L_3^{-i_3}
\end{align}
for some three dimensional Young diagram
 $\Delta \subset \mathbb{Z}_{\geqslant 0}^3$.
In particular, we have
\begin{equation}P_0(X, m[C])^T=P_0(X, m[C])^{T_0}  \nonumber \end{equation}
in this case.
\end{lem}
\begin{proof}
The stable pair (\ref{pair:ell}) is obviously $T$-fixed, hence $T_0$-fixed.
Conversely, for $T_0$-fixed stable pair $(s \colon \oO_X \to F)$ with $\chi(F)=0$, we denote $\oO_Z=\Imm s$ and then
$I_Z$ is $T$-fixed by Lemma~\ref{lem:Tfix}.
It follows that $\oO_X \twoheadrightarrow \oO_Z$ is of the form
\begin{equation}
\oO_X \twoheadrightarrow \bigoplus_{(i_1, i_2, i_3) \in \Delta}
L_1^{-i_1} \otimes L_2^{-i_2} \otimes L_3^{-i_3}\otimes \oO_C(Z_{i_1,i_2,i_3}).
\nonumber \end{equation}
Since $c_1(L_i)=0$ and $F/\oO_Z$ is zero dimensional, then
\begin{align*}
0=\chi(F)=\chi(\oO_Z)+\chi(F/\oO_Z)\geqslant\chi(F/\oO_Z)\geqslant 0.
\end{align*}
So $F\cong\oO_Z$ and $Z_{i_1,i_2,i_3}=0$.
\end{proof}
%\begin{question}
%In general $P_n(X, m[C])^{T}=P_n(X, m[C])^{T_0}$ holds?
%\end{question}
We determine stable pair invariants for $X=\mathrm{Tot}_{C}(L_{1}\oplus L_{2}\oplus L_{3})$
when line bundles $L_i \in \Pic^0(C)$ over the elliptic curve $C$ are general.
\begin{thm}\label{local elliptic curve}
Let $C$ be an elliptic curve, $L_i \in \Pic^0(C)$ $(i=1,2,3)$ general line bundles satisfying $L_1 \otimes L_2 \otimes L_3 \cong \omega_C$ and $X=\mathrm{Tot}_C(L_1\oplus L_2\oplus L_3)$. 

Then stable pair invariants $P_{0, m[C]}(X)$ (\ref{localization for local curve}) are well-defined and fit into generating series
\begin{align*}
\sum_{m\geqslant 0} P_{0, m[C]}(X)\,q^m=M(q),
\end{align*}
where $M(q):=\prod_{k\geqslant 1}(1-q^{k})^{-k}$ is the MacMahon function.
\end{thm}
\begin{proof}
By Lemma~\ref{lem:ell:fixed}, $P_{0}(X, m[C])^{T_0}$ is
the finite set of three dimensional
partitions of $m$, and any $(s:\oO_X\to F)\in P_{0}(X, m[C])^{T_0}$ satisfies $F\cong\oO_W$ for some Cohen-Macaulay curve $W$ in $X$. We denote $I$ to be the ideal sheaf of $W$.

Let $U \subset C$ be an open subset on which
$L_i$ are trivial.
Then $p^{-1}(U) \cong U \times \mathbb{C}^3$ and
$I|_{p^{-1}(U)}$ is isomorphic to
$\pi^{\ast}I_{Z}$ for the
$T$-fixed zero-dimensional subscheme $Z \subset \mathbb{C}^3$
corresonding to $\Delta$. Therefore we have an isomorphism of $T$-equivariant sheaves on $U$
\begin{align}\label{isom:U}
p_{\ast} \eE xt_X^k(I, I)|_U \cong
\Ext_{\mathbb{C}^3}^k(I_Z, I_Z) \otimes_{\mathbb{C}} \oO_U.
\end{align}
Let
\begin{align*}
\Ext_{\mathbb{C}^3}^k(I_Z, I_Z) =\bigoplus_{(i_1, i_2, i_3) \in \mathbb{Z}^3}
V^{k}_{i_1, i_2, i_3} \otimes t_1^{i_1} t_2^{i_2} t_3^{i_3}
\end{align*}
be the decomposition into $T$-weight spaces.
By (\ref{isom:U}), we have
\begin{align*}
p_{\ast} \eE xt_X^k(I, I) \cong
\bigoplus_{(i_1, i_2, i_3)
\in \mathbb{Z}^3}
V^{k}_{i_1, i_2, i_3} \otimes L_1^{i_1} \otimes L_2^{i_2} \otimes
L_3^{i_3} \otimes t_1^{i_1} t_2^{i_2} t_3^{i_3}.
\end{align*}
The relation (\ref{L123}) and Lemma~\ref{lem:isom:ext} imply that 
\begin{align*}
\Ext_X^1(I, I)_0^{T_0}=\bigoplus_{i\in \mathbb{Z}} V_{i, i, i}^1, \,\, \
\Ext_X^2(I, I)_0^{T_0}=\bigoplus_{i\in \mathbb{Z}} V_{i, i, i}^1 \oplus
V_{i, i, i}^2.
\end{align*}
By~\cite[Lemma~4.1]{BBr}, we have
$V_{i, i, i}^1=V_{i, i, i}^2=0$ $($note $(V_{i, i, i}^2)^{\vee}\cong V_{-i,-i,-i}^1$$)$.
Therefore  
\begin{align*}
[P_0(X, m[C])^{T_0}]^{\rm{vir}}=[P_0(X, m[C])].
\end{align*}
For the movable part, there are decompositions 
\begin{align*}
\Ext_X^1(I, I)_0^{\rm{mov}} &=\bigoplus_{(i_1-i_3, i_2-i_3) \neq (0, 0)}
V_{i_{1}, i_2, i_3}^1 \otimes H^0(L_1^{i_1-i_3} \otimes L_2^{i_2-i_3}) \otimes
t_1^{i_1-i_3} t_2^{i_2-i_3}, \\
\Ext_X^2(I, I)_0^{\rm{mov}} &=
\bigoplus_{(i_1-i_3, i_2-i_3) \neq (0, 0)}
(V_{i_{1}, i_2, i_3}^1 \oplus V_{i_1, i_2, i_3}^2)
 \otimes H^0(L_1^{i_1-i_3} \otimes L_2^{i_2-i_3}) \otimes
t_1^{i_1-i_3} t_2^{i_2-i_3}.
\end{align*}
For a general choice of $(L_1, L_2)$, we have
$H^0(L_1^a  \otimes L_2^b)=H^1(L_1^a \otimes L_2^b)=0$
for any $(a, b) \neq (0, 0)$, so the movable part also vanishes. Thus
\begin{align*}
P_{0, m[C]}(X)=\sharp \big(P_0(X, m[C])^{T_0}\big),
\end{align*}
which is the number of three dimensional partitions of $m$.
\end{proof}

\section{Appendices}

\subsection{Stable pairs and one dimensional sheaves for irreducible curve classes}
When $\beta\in H_2(X,\mathbb{Z})$ be an irreducible curve class on a smooth projective CY 4-fold $X$, we have a morphism
\begin{equation}\phi_n: P_n(X,\beta)\to M_{n,\beta}(X) \nonumber \end{equation}
to the moduli scheme of 1-dimensional stable sheaves with Chern character $(0,0,0,\beta,n)$ (e.g. \cite[pp. 270]{PT2}), whose 
fiber over $[F]$ is $\mathbb{P}(H^0(X,F))$. Note that $M_{n,\beta}(X)$ is in general a stack instead of scheme when $\beta$ is arbitrary. 
The virtual dimension of $M_{n,\beta}(X)$ satisfies
%\begin{equation}\mathrm{vir. dim}_{\mathbb{R}}(P_n(X,\beta))=2n, \quad  \mathrm{vir. dim}_{\mathbb{R}}(M_{n,\beta}(X))=2. \nonumber \end{equation}
\begin{equation}\mathrm{vir. dim}_{\mathbb{R}}(M_{n,\beta}(X))=2, \nonumber \end{equation}
by \cite{BJ, CMT}.
One could use the virtual class to define invariants.

For integral classes $\gamma_i \in H^{m_i}(X, \mathbb{Z})$, $1\leqslant i\leqslant l$, let 
\begin{align*}
\tau \colon H^{m}(X)\to H^{m-2}(M_{n,\beta}(X)), \
\tau(\gamma)=\pi_{P\ast}(\pi_X^{\ast}\gamma \cup\ch_3(\mathbb{F}) ),
\end{align*}
where $\pi_X$, $\pi_M$ are projections from $X \times M_{n,\beta}(X)$
to corresponding factors, $\mathbb{F}\to X\times M_{n,\beta}(X)$ is the universal sheaf, and $\ch_3(\mathbb{F})$ is the
Poincar\'e dual to the fundamental cycle of $\mathbb{F}$.

Then we define $\DT_4$ invariant 
\begin{align*}\DT_4(n,\beta\textrm{ }|\textrm{ } \gamma_1,\ldots,\gamma_l):=\int_{[M_{n,\beta}(X)]^{\rm{vir}}} \prod_{i=1}^{l}\tau(\gamma_i). \end{align*}
We propose the following conjecture.
\begin{conj}\label{conj appendix}
For an irreducible class $\beta\in H_2(X,\mathbb{Z})$, the invariants 
\begin{align*}
\DT_4(n,\beta\textrm{ }|\textrm{ } \gamma_1,\ldots,\gamma_l)  \end{align*}
are independent of the choice of $n$ for certain choices of orientation in defining them.
\end{conj}
In all compact examples studied in this paper, one can check Conjecture \ref{conj appendix} holds in these cases. In particular, when $X=Y\times E$ 
is the product of a CY 3-fold $Y$ with an elliptic curve $E$ and the irreducible class $\beta\in H_2(Y,\mathbb{Z})\subseteq H_2(X,\mathbb{Z})$ sits
inside $Y$, then Conjecture \ref{conj appendix} reduces to a special case of the multiple cover formula (\cite[Conjecture 6.20]{JS}, \cite[Conjecture 6.3]{Toda2}):
\begin{align*}N_{n, \beta}=\sum_{k\geqslant 1, k|(n, \beta)}
\frac{1}{k^2} N_{1, \beta/k},
\end{align*}
for any $\beta$ in a CY 3-fold $Y$, where $N_{n, \beta} \in \mathbb{Q}$ is the generalized DT invariant \cite{JS}
which counts one dimensional semistable sheaves $E$ on $Y$ with $[E]=\beta$, $\chi(E)=n$. 
The above formula is proved when $\beta$ is primitive in \cite[Lemma~2.12]{Toda3} (see also \cite[Appendix A.]{CMT}).

It is an interesting question to define `generalized $\DT_4$ type invariant' counting semistable sheaves on CY 4-folds and 
search for similar multiple cover formula on CY 4-folds.

\subsection{An orientability result for moduli spaces of stable pairs on CY 4-folds}
${}$ \\
Let $X$ be a smooth projective CY 4-fold and $c\in H^{\mathrm{even}}(X)$. 
For a moduli stack $M_c$ of coherent sheaves on $X$ with Chern character $c$, 
we define
\begin{equation}\mathcal{L}:=det(\dR (p_{M})_\ast \dR \hH om(\mathbb{F},\mathbb{F}))  \nonumber \end{equation}
to be the determinant line bundle, where $\mathbb{F}\to M_c\times X$ is the universal sheaf of $M_c$ and $p_M:M_c\times X\to M_c$ is the projection.
By Serre duality, we have a non-degenerate pairing
\begin{equation}Q: \mathcal{L}\times \mathcal{L}\to \oO_{M_c}, \nonumber \end{equation}
which defines a $O(1,\mathbb{C})$-structure on $\lL$. The quadratic line bundle $(\lL,Q)$ is called \textit{orientable} if its structure group can be reduced to $SO(1,\mathbb{C})=\{1\}$. An orientability result is recently proved on arbitrary CY 4-folds \cite{CGJ}, where the proof uses involved tools like semi-topological $K$-theory.
To be self-contained, we include here a simpler proof of an orientability result for CY 4-folds with technical assumptions $\mathrm{Hol}(X)=SU(4)$ and $H^{\mathrm{odd}}(X,\mathbb{Z})=0$.

\begin{lem}\label{ori existence for moduli stack}
Let $X$ be a CY 4-fold with $\mathrm{Hol}(X)=SU(4)$ and $H^{\mathrm{odd}}(X,\mathbb{Z})=0$. 
Let $M_c$ be a finite type open substack of the moduli stack of coherent sheaves with Chern character $c\in H^{\mathrm{even}}(X)$.
Then the quadratic line bundle $(\mathcal{L},Q)$ is orientable.
\end{lem}
\begin{proof}
By the work of Joyce-Song \cite[Thm. 5.3]{JS}, the moduli stack $M_c$ is 1-isomorphic to a finite type 
moduli stack of holomorphic vector bundles on $X$ via Seidel-Thomas twists, under which the universal family can be identified (so is the determinant line bundle and 
Serre duality pairing). Thus we may assume $M_c$ to be a moduli stack of (rank $n$) holomorphic bundles without loss of generality.

Fix a base point $x_0\in X$, a \textit{framing} $\phi$ of a vector bundle $E$ is an isomorphism
\begin{equation}\phi: E|_{x_0}\cong \mathbb{C}^n.  \nonumber \end{equation}
There is a natural $GL(n,\mathbb{C})$-action on $\phi$ changing the framing.

Let $M_c^{\mathrm{framed}}$ denote the moduli stack of framed holomorphic bundles with Chern character $c$,
on which $GL(n,\mathbb{C})$ acts by changing framings. 
Note that $M_c^{\mathrm{framed}}$ is a scheme as the stabilizer is trivial and we have a 1-isomorphism 
\begin{equation}[M_c^{\mathrm{framed}}/GL(n,\mathbb{C})]\cong M_c, \quad (E,\phi)\mapsto E,\nonumber \end{equation} 
%\begin{equation}(E,\phi)\mapsto E, \nonumber \end{equation} 
of Artin stacks. The universal family 
\begin{equation}\mathcal{E}\to M_c^{\mathrm{framed}}\times X \nonumber \end{equation}
descends to the universal sheaf $\mathbb{F}$ of $M_c$. 
Let 
\begin{equation}\mathcal{L}:=det(\dR p_\ast \dR \hH om(\mathcal{E},\mathcal{E}))\nonumber \end{equation} 
be the determinant line bundle of $M_c^{\mathrm{framed}}$, 
where $p: M_c^{\mathrm{framed}}\times X\to M_c^{\mathrm{framed}}$ is the projection. 
One may reduce the orientability problem of $M_c$ to the orientability of ($\mathcal{L},Q)$, where $Q$ is the quadratic form on $\lL$ 
defined by Serre duality. 

We view a holomorphic bundle as an integrable $\overline{\partial}$ connection on its underlying topological bundle. 
Then there is a natural embedding of $M_c^{\mathrm{framed}}$ (with induced complex analytic topology)
\begin{equation}M_c^{\mathrm{framed}}\hookrightarrow \widetilde{\mathcal{B}}_E \nonumber \end{equation}
into the space $\widetilde{\mathcal{B}}:=\mathcal{A}\times_{\mathcal{G}}E_{x_0}$ of framed (not necessarily integrable) connections on the underlying topological bundle $E$. The determinant line bundle $\lL$ on $M_c^{\mathrm{framed}}$ is the pull-back of a line bundle 
$\lL_{\widetilde{\mathcal{B}}}$ on $\widetilde{\mathcal{B}}_E$ defined as the determinant of the index bundle of certain twisted Dirac operators 
and the quadratic form $Q$ on $\lL$ extends to $\lL_{\widetilde{\mathcal{B}}}$ defined using the spin structure of $X$ (see \cite[pp. 50-51]{CL2}).
By \cite[Thm. 1.3]{CL2}, the quadratic line bundle $(\lL_{\widetilde{\mathcal{B}}},Q)$ is orientable. Hence we are done.
\end{proof}
Then orientability for moduli spaces of stable pairs follows from the orientability of moduli stacks of one dimensional 
sheaves.
\begin{thm}\label{ori existence}
Let $X$ be a CY 4-fold with $\mathrm{Hol}(X)=SU(4)$ and $H^{\mathrm{odd}}(X,\mathbb{Z})=0$. Then
for any $\beta\in H_{2}(X,\mathbb{Z})$ and $n\in\mathbb{Z}$, the quadratic line bundle $(\mathcal{L},Q)$ over $P_n(X, \beta)$ is orientable.
\end{thm}
\begin{proof}
There is a morphism
\begin{equation}\phi:P_n(X, \beta)\rightarrow M(0,0,0,\beta,n),  \nonumber \end{equation}
\begin{equation}\phi(F,s)=F \nonumber \end{equation}
to the moduli stack of 1-dimensional sheaves on $X$ with Chern character $(0,0,0,\beta,n)$.

Parallel to the quadratic line bundle $\big(\mathcal{L}=det(\dR (\pi_{P})_{\ast} \dR \hH om(\mathbb{I}, \mathbb{I})_0),Q\big)$ over $P_n(X, \beta)$, there exists a determinant line bundle
\begin{equation}\mathcal{L}_{M}=det(\dR (\pi_{M})_{\ast} \dR \hH om(\mathbb{F}, \mathbb{F})) \nonumber \end{equation}
with a quadratic form $Q_{M}$ over $M(0,0,0,\beta,n)$, where $\pi_{M}: M(0,0,0,\beta,n)\times X\rightarrow M(0,0,0,\beta,n)$ is the projection,
and we use $\mathbb{F}$ to denote universal sheaf for both $P_n(X, \beta)$ and $M(0,0,0,\beta,n)$.

Via the morphism $\phi$, we have an isomorphism
\begin{equation}\label{equ0}\phi^{*}\mathcal{L}_{M}\cong det(\dR (\pi_P)_{\ast} \dR \hH om(\mathbb{F}, \mathbb{F})),  \end{equation}
where $\pi_P:X\times P_n(X, \beta)\rightarrow P_n(X, \beta)$ is the projection.

Since $\mathbb{I}=(\mathcal{O}_{X\times P_n(X, \beta)}\rightarrow \mathbb{F})$, we have a distinguished triangle
\begin{equation}\dR \hH om(\mathbb{F}, \mathbb{F})\rightarrow \dR \hH om(\mathcal{O}_{X\times P_n(X, \beta)}, \mathbb{F})\rightarrow \dR \hH om(\mathbb{I}, \mathbb{F}),  \nonumber \end{equation}
which gives an isomorphism
\begin{equation}\label{equ1}det(\dR (\pi_P)_{\ast}\dR \hH om(\mathbb{F}, \mathbb{F}))\otimes det(\dR (\pi_P)_{\ast}\dR \hH om(\mathbb{I}, \mathbb{F}))\cong
det(\dR (\pi_P)_{\ast}\dR \hH om(\mathcal{O}_{X\times P_n(X, \beta)}, \mathbb{F}))  \end{equation}
between determinant line bundles.

Similarly, from the distinguished triangle
\begin{equation}\dR \hH om(\mathbb{I}, \mathbb{F})\rightarrow  \dR \hH om(\mathbb{I},\mathbb{I})_{0}[1] \rightarrow\dR \hH om(\mathbb{F},\mathcal{O}_{X\times P_n(X, \beta)})[2],  \nonumber \end{equation}
we have an isomorphism
\begin{equation}\label{equ2}det(\dR (\pi_P)_{\ast}\dR \hH om(\mathbb{I}, \mathbb{F}))\otimes det(\dR (\pi_P)_{\ast}\dR \hH om(\mathbb{F},\mathcal{O}_{X\times P_n(X, \beta)}))\cong \big(det(\dR (\pi_P)_{\ast}\dR \hH om(\mathbb{I}, \mathbb{I})_{0})\big)^{-1}.  \end{equation}
Combining (\ref{equ1}), (\ref{equ2}) and Serre duality, we obtain
\begin{equation}det(\dR (\pi_P)_{\ast}\dR \hH om(\mathbb{I}, \mathbb{I})_{0})\cong det(\dR (\pi_P)_{\ast}\dR \hH om(\mathbb{F}, \mathbb{F})), 
\nonumber \end{equation}
under which the natural quadratic forms on them are identified.

By Lemma \ref{ori existence for moduli stack}, the structure group of quadratic line bundle $(\mathcal{L}_{M},Q_{M})$ can be reduced to $SO(1,\mathbb{C})$, so is $(\mathcal{L},Q)$ via the pull-back (\ref{equ0}).
\end{proof}

\end{document}